\documentclass[11pt]{amsart}
\usepackage{amssymb, amsmath}
\newcommand{\Int}{\displaystyle \int}

\def\dF{\delta\!F}
\def\dG{\delta\!G}
\def\dH{\delta\!H}
\def\dR{\delta\!R}
\def\dU{\delta\!U}
\def\dA{\delta\!A}
\def\dZ{\delta\!Z}
\def\da{\delta\!a}
\def\du{\delta\!u}
\def\dv{\delta\!v}
\def\dP{\delta\!P}
\def\dPi{\delta\!\Pi}
\def\dg{\delta\!g}
\def\df{\delta\!f}

\let\pa=\partial
\let\al=\alpha

\let\d=\delta
\let\e=\varepsilon

\let\r=\rho
\let\s=\sigma
\let\f=\frac

\let\p=\psi

\let\D=\Delta

\let\Om=\Omega
\let\wt=\widetilde

\def\cA{{\mathcal A}}
\def\cB{{\mathcal B}}
\def\cC{{\mathcal C}}

\def\cF{{\mathcal F}}

\def\cM{{\mathcal M}}

\def\cR{{\mathcal R}}
\def\cS{{\mathcal S}}

\def\pa{\partial}
\def\grad{\nabla}

\def\ga{\gamma}

\def\virgp{\raise 2pt\hbox{,}}
\def\cdotpv{\raise 2pt\hbox{;}}
\def\eqdefa{\buildrel\hbox{\footnotesize def}\over =}
\def\Id{\mathop{\rm Id}\nolimits}

\def\C{\mathop{\mathbb C\kern 0pt}\nolimits}
\def\DD{\mathop{\mathbb D\kern 0pt}\nolimits}
\def\EE{\mathop{{\mathbb E \kern 0pt}}\nolimits}
\def\K{\mathop{\mathbb K\kern 0pt}\nolimits}
\def\N{\mathop{\mathbb N\kern 0pt}\nolimits}
\def\Q{\mathop{\mathbb Q\kern 0pt}\nolimits}
\def\R{\mathop{\mathbb R\kern 0pt}\nolimits}
\def\SS{\mathop{\mathbb S\kern 0pt}\nolimits}
\def\ZZ{\mathop{\mathbb Z\kern 0pt}\nolimits}
\def\TT{\mathop{\mathbb T\kern 0pt}\nolimits}
\def\P{\mathop{\mathbb P\kern 0pt}\nolimits}

\newcommand{\la}{\lambda}

\newcommand{\Z}{{\ZZ}}

\def\dive{\mathop{\rm div}\nolimits}

\def\na{\nabla}
\def\p{\partial}
\def\th{\theta}

\newcommand{\beq}{\begin{equation}}
\newcommand{\eeq}{\end{equation}}
\newcommand{\ben}{\begin{eqnarray}}
\newcommand{\een}{\end{eqnarray}}
\newcommand{\beno}{\begin{eqnarray*}}
\newcommand{\eeno}{\end{eqnarray*}}
\newcommand{\andf}{\quad\hbox{and}\quad}

\newtheorem{defi}{Definition}[section]
\newtheorem{thm}{Theorem}[section]
\newtheorem{lem}{Lemma}[section]
\newtheorem{rmk}{Remark}[section]
\newtheorem{cor}{Corollary}[section]
\newtheorem{prop}{Proposition}[section]

\begin{document}
\title[Global solutions to inhomogeneous Navier-Stokes equations]
{Inhomogeneous Navier-Stokes equations in
the half-space,  with only bounded density}
\author[R. Danchin]{Rapha\"el Danchin}
\address [R. Danchin]
{Universit\'{e} Paris-Est,  LAMA (UMR 8050), UPEMLV, UPEC, CNRS, Institut Universitaire de France,
 61 avenue du G\'{e}n\'{e}ral de Gaulle, 94010 Cr\'{e}teil Cedex, France.}
\email{danchin@univ-paris12.fr}
\author[P. Zhang]{Ping Zhang}%
\address[P. Zhang]
 {Academy of Mathematics and Systems Science and  Hua Loo-Keng Key Laboratory of Mathematics,
  Chinese Academy of Sciences, Beijing 100190, CHINA} \email{zp@amss.ac.cn}
\date\today

\begin{abstract}
In this paper, we establish  the global existence and uniqueness of
solutions to the inhomogeneous Navier-Stokes system in the
half-space. The initial density only has to be bounded and close
enough to a positive constant,  the  initial velocity belongs to
some critical Besov space, and the $L^\infty$ norm of the
inhomogeneity plus the critical norm to the horizontal component of
the initial velocity has to be very small compared to the
exponential of the norm to the vertical component of the initial
velocity. With a little bit more regularity for the initial
velocity, those solutions are proved to be unique. In the last
section of the paper, our results are partially extended to the
bounded domain case.
\end{abstract}
\maketitle

\vskip 0.2cm


\setcounter{equation}{0}
\section{Introduction}

We are concerned with  the  global well-posedness issue for  the
initial boundary value problem pertaining to  the following incompressible inhomogeneous Navier-Stokes
equations:
\begin{equation}
 \left\{\begin{array}{lll}
\displaystyle \pa_t \rho + \dive (\rho u)=0&\hbox{in}&\R_+\times\Omega, \\
\displaystyle \pa_t (\rho u) + \dive (\rho u\otimes u) -\mu\D u+\grad\Pi=0&\hbox{in}&\R_+\times\Omega, \\
\displaystyle \dive\, u = 0&\hbox{in}&\R_+\times\Omega, \\
\displaystyle u=0&\hbox{on}&\R_+\times\partial\Omega,\\
\displaystyle \rho|_{t=0}=\rho_0,\quad \rho u|_{t=0}=\r_0u_0&\hbox{in}&\Omega,
\end{array}\right. \label{1.1gh}
\end{equation}
where $\rho=\rho(t,x)\in\R_+,$ $u=u(t,x)\in\R^d$ and
$\Pi=\Pi(t,x)\in\R$  stand for the density, velocity field  and
pressure of the fluid, respectively, depending on  the time variable
$t\in\R_+$ and on the space variables $x\in\Omega.$ The positive
real number $\mu$ stands for the viscosity coefficient. We mainly
consider the case  where $\Omega$ is  the half-space $\R^d_+,$
except in the last section of the paper where it stands for a smooth
bounded domain of $\R^d$ $(d\geq2).$ \smallbreak The above system
describes a  fluid that is  incompressible but has nonconstant
density. Basic examples are mixture of incompressible and non
reactant flows, flows with complex structure (e.g. blood flow or
model of rivers), fluids containing a melted substance, etc.
 \medbreak
 A number of recent works have been dedicated to the mathematical study
of the above system. Global weak solutions with finite energy have
been constructed by J. Simon in \cite{Simon} (see also the book by
P.-L. Lions \cite{LP} for the variable viscosity case). In the case
of smooth data with no vacuum, the existence of strong unique
solutions goes back to the work of O. Ladyzhenskaya and V.
Solonnikov in \cite{LS}. More recently, the first author  \cite{D1}
established the well-posedness of the above system in the whole
space $\R^d$ in the so-called \emph{critical functional framework}
for small perturbations of some positive constant density. The basic
idea is to use functional spaces (or norms) that have the same \emph{scaling
invariance} as \eqref{1.1gh}, namely
\begin{equation}\label{eq:scalinginhomo}
(\rho,u,\Pi)(t,x)\longmapsto (\rho,\lambda u,\lambda^2\Pi)
(\lambda^2 t,\lambda x),\qquad
(\rho_0,u_0)(x)\longmapsto (\rho_0,\lambda u_0)(\lambda x).
\end{equation}
More precisely, in \cite{D1}, global well-posedness was established assuming that
for some small enough constant $c$ one has
$$\|\rho_0-1\|_ {\dot B^{\frac d2}_{2,\infty}(\R^d)\cap L^\infty(\R^d)}+\mu^{-1}\|u_0\|_{\dot B^{-1+\frac d2}_{2,1}(\R^d)}\leq c.$$
Above $\dot B^{\sigma}_{p,r}(\R^d)$ stands for a homogeneous Besov
space on $\R^d$ (see Definition \ref{def:Besov} below). This result
was extended to more general Besov spaces  by H. Abidi in
\cite{Abidi}, and H. Abidi and M. Paicu in \cite{AP}, and to the
half-space setting   in \cite{DM}. The smallness assumption on the
initial density was removed recently in \cite{AGZ2,AGZ3}.

Very recently, in a joint work with M. Paicu in \cite{PZ2}, the second author established that  if $p\in(1,2d)$
 then there exists a constant $c$ so that  for any data $(\rho_0,u_0)$ satisfying
 $$\Bigl(\|\rho_0-1\|_ {\dot B^{\frac d{p}}_{p,1}(\R^d)}+\mu^{-1}\|u_0^h\|_{\dot
B^{-1+\frac{d}p}_{p,1}(\R^d)}\Bigr)\exp\Bigl(c^{-1}\mu^{-2}\|u_0^d\|_{\dot B^{-1+\frac{d}p}_{p,1}(\R^d)}\Bigr)\leq c$$
the incompressible inhomogeneous Navier-Stokes equations admit a global unique solution.
Above,   we agreed that $u_0^h=(u_0^1,\cdots,u_0^{d-1})$ and $u_0^d$ stand for the horizontal and vertical  components of $u_0.$
\smallbreak
Given that in all those works the density has to be at least in the Besov space $\dot B^{\f dp}_{p,\infty}(\R^d),$
  one cannot capture  discontinuities  across an hypersurface. In effect,
  the Besov regularity of the characteristic function of a smooth domain is only $\dot B^{\f1p}_{p,\infty}(\R^d).$
  Therefore, those results do not apply to a  mixture of two fluids with different densities.
\smallbreak
In \cite{dm},  the first author and P. Mucha noticed that
it was possible to establish existence \emph{and} uniqueness of a
solution in the case of a small discontinuity, in  a critical
functional framework. More precisely, the global existence and
uniqueness was established for any data $(\rho_0,u_0)$ such that for
some $p\in[1,2d)$ and small enough constant $c,$ we have
\begin{equation}\label{eq:small}
\|\rho_0-1\|_{\cM(\dot B^{-1+\frac
dp}_{p,1}(\R^d))}+\mu^{-1}\|u_0\|_{\dot B^{-1+\frac
dp}_{p,1}(\R^d)}\leq c.
\end{equation}
Above, $\|\cdot\|_{\cM(\dot B^{-1+\frac dp}_{p,1}(\R^d))}$ is the
\emph{multiplier norm} associated to the Besov space $\dot
B^{-1+\frac dp}_{p,1}(\R^d),$ which turns out to be finite for
characteristic functions of $C^1$ domains whenever $p>d-1.$
Therefore, initial densities with a discontinuity across an
interface may be considered (although the jump has to be small owing
to \eqref{eq:small}). As observed later on in \cite{dm2},  large
discontinuities may be considered if the initial  velocity is
smoother. In fact, therein, \emph{any} initial density  bounded and
bounded away from $0$ is admissible. Let us emphasize that in both
works (\cite{dm} and \cite{dm2}), using Lagrangian coordinates was
the key to the proof of uniqueness. \smallbreak
 A natural question is whether it is still possible to get existence \emph{and}
uniqueness in a critical functional framework where  $\rho_0$ is
only bounded and bounded away from zero. As regards the existence issue, a
positive answer has been given recently by J. Huang, M. Paicu and
the second author in \cite{HPZ3}, in the whole space setting, and
uniqueness was obtained if assuming slightly more regularity for the
velocity field. Let us emphasize that once again using Lagrangian
coordinates is the key to uniqueness. Therefore, assumptions on the
initial velocity have to ensure the velocity $u$ to have gradient
 in $L^1_{loc}(\R_+;L^\infty(\R^d))$ for  Eulerian and Lagrangian formulations of the system
 to be  equivalent.  While this property of the velocity field holds true if $u_0$ is in  $\dot B^{-1+\f dp}_{p,1}(\R^d),$
 it fails  if $u_0$ is only in $\dot B^{-1+\f dp}_{p,r}(\R^d)$ for some $r>1.$
As a matter of fact, the question of uniqueness in a critical  Besov framework for the velocity is open unless $r=1$ (but this latter
case requires  stronger assumptions on the density, as pointed out in  \cite{dm}).
\medbreak In the present work, we aim at  extending the results of
\cite{HPZ3} to the half-space setting. Because we shall consider
only perturbations of the reference density $1,$ it is natural to
set  $a=1/\rho-1$ so that  System \eqref{1.1gh} translates into
\begin{equation}\label{INS}
 \quad\left\{\begin{array}{lll}
\displaystyle \pa_t a + u \cdot \grad a=0&\hbox{in}&\R_+\times\Omega, \\
\displaystyle \pa_t u + u \cdot \grad u+ (1+a)(\grad\Pi-\mu\D u)=0&\hbox{in}&\R_+\times\Omega, \\
\displaystyle \dive\, u = 0&\hbox{in}&\R_+\times\Omega, \\
\displaystyle u=0&\hbox{on}&\R_+\times\partial\Omega,\\
 \displaystyle (a, u)|_{t=0}=(a_0, u_{0})&\hbox{in}&\Omega.
\end{array}\right.
\end{equation}
As in the whole space case considered in \cite{HPZ3}, the functional
framework for solving \eqref{INS} is motivated by classical maximal
regularity estimates for the evolutionary Stokes system. Indeed, the velocity field may be seen as the solution to the following
Stokes system:
\begin{equation}\label{eq:velocity}
\p_tu-\mu\Delta u+\nabla\Pi=-u\cdot\nabla u+a(\mu\Delta u-\nabla\Pi),\qquad
\dive u=0.
\end{equation}
In the whole space case,  we have for any $1<p,r<\infty$ and $t>0,$
$$
\|(\p_tu,\mu\nabla^2u,\nabla\Pi)\|_{L_t^r(L^p)}\leq C\Bigl(\mu^{1-\frac1r}\|u_0\|_{\dot B^{2-\f2r}_{p,r}}+\|u\cdot\nabla u\|_{L_t^r(L^p)}
+\|a(\mu\Delta u-\nabla\Pi)\|_{L_t^r(L^p)}\Bigr)
$$
where we have used the notation
$\|z\|_{L_t^r(L^p)}\eqdefa\|z\|_{L^r((0,t);L^p(\R^d))}.$ \smallbreak
Given that $\|a(t)\|_{L^\infty}=\|a_0\|_{L^\infty}$ for all time, it
is clear that the last term may be absorbed by the l.h.s. if
$\|a_0\|_{L^\infty}$ is small enough. Now, it is standard (see
Corollary \ref{c:besov} below) that in the simpler case where $u$
just solves the heat equation with initial data $u_0,$ having
$\Delta u$ in $L^r(0,T;L^p)$ is equivalent to $u_0\in \dot B^{-1+\f
dp}_{p,r}$ \emph{provided  $-1+\f dp=2-\f2r\cdotp$} This implies
that $p$ and $r$ have to be interrelated through $p=\f{dr}{3r-2},$
and thus  $\f d3<p<d.$

\smallbreak Here we aim at extending this simple  idea to the half-space
setting, or to  $C^2$ bounded domains.

 \medbreak
%Before stating our main global existence result in
%the half-space case, let us introduce a few notation.
In the half-space case, according to the above heuristics and because homogeneous
Dirichlet boundary conditions are prescribed for the velocity,  the natural  solution space for
$(u,\nabla\Pi)$ is
$$
X^{p,r}_T\eqdefa\Bigl\{(u,\nabla\Pi)\ \hbox{ with }\ u\in\cC([0,T];\dot
\cB^{2-\f2r}_{p,r}(\R^d_+))\quad\hbox{and}\quad
\p_tu,\nabla^2u,\nabla\Pi\in L^r(0,T;L^p_+)\Bigr\},
$$
where $L^p_+\eqdefa L^p(\R^d_+)$ denotes  the Lebesgue space over
$\R^d_+,$ and  $\dot \cB^{2-\f2r}_{p,r}(\R^d_+)$ stands for the set
of divergence free vector fields  on $\R^d_+$ with Besov regularity
$\dot B^{2-\f2r}_{p,r}(\R^d_+)$ and null trace at the boundary (the
exact meaning will be given in Definition \ref{d:nulltrace} below).
\medbreak We also  introduce the following norm for all $T>0$:
\beq\label{normxpr} \begin{split}&
\|(u,\nabla\Pi)\|_{X^{p,r}_T}\eqdefa \|u\|_{\mathfrak{X}^{p,r}_T}+
\|\na\Pi\|_{L^r(0,T;L^{p}_+)}\andf\\
&\|u\|_{\mathfrak{X}^{p,r}_T}\eqdefa
\mu^{1-\f1{r}}\|u\|_{L^\infty(0,T;\dot B^{2-\f2r}_{p,r}(\R^d_+))}+
\|(\p_tu,\mu\na^2u)\|_{L^r(0,T;L^{p}_+)}, \end{split}\eeq and agree
that $X^{p,r}$ and $\|\cdot\|_{X^{p,r}}$ correspond to the above
definition with $T=+\infty.$
%and $\dot B^{2-\f2r}_{p,r}(\R^d_+)$ is the restriction in the
%distributional meaning of the homogeneous Besov space $\dot B^{2-\f2r}_{2,r}(\R^d),$ to $\R^d_+.$
\medbreak
%However, as pointed out above, for such an exponent $p$ having $(u,\nabla\Pi)\in X^{p,r}$ does not imply
%$\nabla u\in L^1_{loc}(\R^+;L^\infty_+).$ Hence we shall have to assume that $u$ is also in $X^{\wt p,r}$ for
%a \emph{larger} value of $\wt p,$ in order to get uniqueness.
Before stating our main results,  let us clarify what we mean  by \emph{a weak solution to \eqref{INS}}:
\begin{defi}\label{defi2.1} {\sl  A global weak solution of \eqref{INS} is any couple $(a,u)$ satisfying:
\begin{itemize}
\item for any test function $\phi\in
C^\infty_c([0,\infty)\times\overline{\Omega}),$ there holds
\beq\label{def2.1a}\begin{aligned}
\int_0^\infty\!\!\!\int_{\Omega}a(\p_t\phi&+u\cdot\na\phi)\,dx\,dt+\int_{\Omega}\phi(0,\cdot)a_0\,dx=0,\\
&\int_0^\infty\!\!\!\int_{\Omega}\phi\dive u \,dx\,dt=0, \end{aligned}\eeq

\item for any vector valued function $\Phi=(\Phi^1,\cdots,\Phi^d)\in
C_c^\infty([0,\infty)\times\overline{\Omega})$, one has
\beq\label{def2.1b}
\int_0^\infty\!\!\!\int_{\Omega}\Bigl\{u\cdot\p_t\Phi+u\otimes u
:\na \Phi +(1+a)\bigl(\mu\D u
-\na\Pi\bigr)\Phi\Bigr\}\,dx\,dt+\int_{\Omega}u_0\cdot\Phi(0,\cdot)\,dx=0.
\eeq
\end{itemize}}
\end{defi}

Our main statement reads:
\begin{thm}\label{th:main}
{\sl  Let $a_0\in L^\infty_+$ and $ u_0=(u_0^h,u_0^d)\in
\dot{\cB}^{-1+\f{d}{p}}_{p,r}(\R^d_+)\cap
\dot{\cB}^{-1+\f{d}{p}}_{\wt p,r}(\R^d_+)$ with
$p\eqdefa\f{dr}{3r-2}\leq\wt p\leq \frac{dr}{r-1}$ and  $r\in
(1,\infty).$
  There exist two positive constants $c_0=c_0(r,d)$ and $c_1=c_1(r,d)$  so that if
\beq \label{small1} \eta_0\eqdefa
\bigl(\mu\|a_0\|_{L^\infty_+}+\|u_0^h\|_{\dot{B}^{-1+\f{d}{p}}_{p,r}(\R^d_+)}\bigr)\exp
\Bigl(c_1\mu^{-2r}\|u_0^d\|_{\dot{B}^{-1+\f{d}{p}}_{p,r}(\R^d_+)}^{2r}\Bigr)
\leq c_0\mu \eeq then \eqref{INS} has a global  solution
$(a,u,\nabla\Pi)$ in the meaning of Definition \ref{defi2.1} with
$\Omega=\R^d_+,$ satisfying
$\|a(t)\|_{L^\infty_+}=\|a_0\|_{L^\infty_+}$ for all $t>0$ and $u\in
X^{p,r}\cap X^{\wt p,r}.$ Moreover, there exist $C_i=C_i(r,d),$
$i=1,2,3,4,$   so that
  \begin{equation}\label{thm1aa}
\begin{split}
  &\|u^h\|_{\mathfrak{X}^{p,r}}+\mu^{1-\f1{2r}}\|\nabla u^h\|_{L^{2r}(\R_+;L^{\f{dr}{2r-1}}_+)}\leq C_1\mu^{1-\f1r}\eta_0,\\
&\|(u,\nabla\Pi)\|_{X^{p,r}}+\mu^{1-\f1{2r}}\|\nabla
u\|_{L^{2r}(\R_+;L^{\f{dr}{2r-1}}_+)}\leq C_2\mu^{1-\f1r}
  \|u_0\|_{\dot{B}^{-1+\f{d}{p}}_{p,r}(\R^d_+)},
\end{split}
\end{equation}
  and, if $\f1\alpha\eqdefa \f1{\wt p}-\f{r-1}{dr},$
  \begin{equation}\label{thm1ab}
  \begin{split}
  \|(u,\nabla\Pi)\|_{X^{\wt p,r}}+\mu^{1-\f1{2r}}\|\na
  u\|_{L^{2r}(\R_+;L^{\alpha}_+)}\leq C_3\mu^{1-\f1r}&\|u_0\|_{\dot{B}^{-1+\f{d}{p}}_{\wt
 p,r}(\R^d_+)}\\
 &\times\exp\Bigl(C_4\mu^{-2r}{\|u_0^d\|_{\dot{B}^{-1+\f{d}{p}}_{p,r}(\R^d_+)}^{2r}}\mu^{-2r}\Bigr). \end{split}  \end{equation}
    If  in addition $\wt p>d$ then $\nabla u\in
L^{q}(\R_+;L^\infty_+)$  with $q=\frac{2\wt p}{d+(\f2r\,-1)\wt
p}$ and
\begin{equation}\label{thm1b}
\mu^{\f1q}\|\nabla u\|_{L^{q}(\R_+;L^\infty_+)}\leq
C_3\|u_0\|_{\dot{B}^{-1+\f{d}{p}}_{\wt
p,r}(\R^d_+)}\exp\Bigl(C_4\mu^{-2r}{\|u_0^d\|_{\dot{B}^{-1+\f{d}{p}}_{p,r}(\R^d_+)}^{2r}}\Bigr),
\end{equation}
and uniqueness holds true.
}\end{thm}
%\begin{rmk}  Owing to  the use of extension properties from the domain where the system is
%considered, to the whole space, the additional condition $d/p-1<1/p$ is
%needed. This is equivalent to $r<(2d-2)/(2d-3).$\end{rmk}
%\begin{rmk}{\sl  In contrast with the whole space case in \cite{HPZ3, PZ2}, it is not clear
%that one may improve the \emph{isotropic} smallness condition \eqref{small1}
%to an \emph{anisotropic} one allowing for arbitrarily large vertical velocity.
%nonlinear smallness assumption on $(a_0,u_0^h,u_0^d)$ as in \cite{HPZ3, PZ2} for the whole
%space case, which basically requires $\|a_0\|_{L^\infty}$ and
%$\|u_0^h\|_{\dot{B}^{-1+\f{d}{p}}_{p,r}(\R^d)}$ being exponentially
%small compared with $\|u_0^h\|_{\dot{B}^{-1+\f{d}{p}}_{p,r}(\R^d)}.$
%The reason why is that  even for solutions to the linear Stokes system in the half-space, the
%horizontal component of the velocity $u^h$ depends on both  $u_0^h$
%and $u_0^d$ (see the formula in Theorem \ref{th2.1} below).}
%\end{rmk}
\begin{rmk}{\sl  We shall extend this statement to a more general critical (or almost critical) Besov setting,
see Theorems \ref{th:main2} and \ref{th:main3} below.} We could also
establish  the local well-posedness of \eqref{INS} with arbitrarily large
velocity and small inhomogeneity. For simplicity, we skip the details here.
\end{rmk}
Let us briefly describe the plan of the rest of the paper. The next
section is devoted to the linearized  velocity equation of
\eqref{INS} \emph{in the half-space}, that is the evolutionary
Stokes system. We first derive an explicit solution formula in the
spirit of that of  S. Ukai in \cite{Ukai}, and then deduce maximal
regularity type estimates similar to those of the whole space. We
consider the general situation with prescribed (possibly nonzero)
value for $\dive u$ as it will be needed when reformulating
\eqref{INS} in Lagrangian coordinates. The next two sections are
devoted to the proof of the existence part of Theorem \ref{th:main},
first under a stronger assumption on the density, and next in the
rough case corresponding to the hypotheses of the theorem. The case
of more general Besov spaces will be examined in Section
\ref{s:general}. The proof of uniqueness is postponed in Section
\ref{s:uniqueness}. In the final section, we partially generalize
Theorem \ref{th:main} to the bounded domain setting. Some technical
lemmas  related to maximal regularity and $L^p(L^q)$ estimates for
the heat equation in the whole space (or Stokes system in bounded
domains)  are presented in Appendix.

%%%%%%%%%%%%%%%%%%%%%%%%%%%%%%%%%%%%%%%

\setcounter{equation}{0}

\section{The evolutionary Stokes system in the half-space}\label{s:stokes}

This section is devoted to the study of following  system in the half-space: \beq\label{eq:stokes}
\left\{\begin{array}{lll}
\displaystyle \p_tu-\mu\D u+\na \Pi=f &\quad \mbox{in}\quad &\R_+\times\R^d_+, \\
\displaystyle \dive u=g &\quad \mbox{in}\quad &\R_+\times\R^d_+,\\
\displaystyle u=0 &\quad \mbox{on}\quad &\R_+\times\p\R^{d}_+,\\
\displaystyle u|_{t=0}=u_0 &\quad \mbox{on}\quad &\R^d_+.
\end{array}\right.
\eeq We shall first derive an explicit formula for the solution to
this system, and next prove the key {\it a priori} estimates that
are needed for getting the main results of our paper.

\subsection{A solution formula}

This part extends a prior work by S. Ukai \cite{Ukai} (see also
\cite{CPS}) to the case where there is a source term $f$ in the
velocity equation, and where the divergence constraint is nonhomogeneous. Let us recall that
in \cite{Ukai}, it was assumed that $f=0$ and $g=0$ (but $u$ need
not be zero at the boundary), and that in \cite{CPS} nonzero $f$
was considered (but  still $u$ is divergence free and $f$ has trace zero).  Furthermore,  the
gradient of the pressure was not computed therein, a computation that turns out to be
  essential for us  as
$\nabla\Pi$ appears in the right-hand side of the velocity equation \eqref{eq:velocity}.
\medbreak Before writing out the formula, let  us  introduce a few
notations.   We denote  $\Delta=\sum_{\ell=1}^{d}\p_{x_\ell}^2$ and
$\D_h=\sum_{\ell=1}^{d-1}\p_{x_\ell}^2,$ and define $|D|^{\pm1}$ and  $|D_h|^{\pm1}$ to be the Fourier multipliers with symbol
$$|\xi_h|^{\pm1}=\biggl(\sum_{i=1}^{d-1}\xi_i^2\biggr)^{\pm\f12}\quad\hbox{and}\quad|\xi|^{\pm1}=\biggl(\sum_{i=1}^{d}\xi_i^2\biggr)^{\pm\f12}, \ \hbox{ respectively.}$$
The notations $R_j$ and $S_j$ stand for the Riesz transforms over
$\R^d$ and $\R^{d-1}_h,$ namely
$$R_j\eqdefa\p_j|D|^{-1}\ \hbox{ for }\  j=1,\cdots, d\qquad\hbox{and}\qquad
S_j\eqdefa\p_j|D_h|^{-1}\ \hbox{ for }\  j=1,\cdots, d-1.$$ We
further set $R_h\eqdefa (R_1,\cdots, R_{d-1})$  and $S\eqdefa
(S_1,\cdots,S_{d-1}).$ \medbreak As in \cite{Ukai},  we define the operators $V_d$
and $V_h$ by \beq\label{2.1} V_du\eqdefa -S\cdot
u^h+u^d\quad\mbox{and}\quad V_hu\eqdefa u^h+Su^d. \eeq
We shall see later on that both $V_hu$ and $V_du$ satisfy a heat equation, this is the main
motivation for considering those two quantities.
\medbreak

We denote $x=(x_h,x_d)$ with  $x_h=(x_1,\cdots,x_{d-1}).$ Let $r$ be
the restriction operator from $\R^d$ to $\R^d_+,$ that is $rf\eqdefa
f|_{\R^d_+},$ and $e_0(f), e_a(f), e_s(f)$ be the extension
operators given by \beq \label{2.2}
\begin{aligned} & e_0(f)= \left\{\begin{array}{l}
\displaystyle f\ \ \mbox{for}\ \ x_d\geq 0, \\
\displaystyle 0\ \ \mbox{for}\ \ x_d< 0,
\end{array}\right. \quad e_a(f) =
\left\{\begin{array}{l}
\displaystyle f(x)\ \ \qquad\quad \  \mbox{for}\ \ x_d\geq 0, \\
\displaystyle -f(x_h,-x_d)\ \ \mbox{for}\ \ x_d< 0,
\end{array}\right.\\ \quad\mbox{and}\qquad
&e_s(f)= \left\{\begin{array}{l}
\displaystyle f(x)\  \qquad\quad \  \mbox{for}\ \ x_d\geq 0, \\
\displaystyle f(x_h,-x_d)\ \ \ \mbox{for}\ \ x_d< 0.
\end{array}\right.
\end{aligned}
\eeq
When solving  \eqref{eq:stokes}, we shall repeatedly consider
the following two equations:
 \beq\label{eq:U} (\p_d+|D_h|)w=|D_h|f\quad\mbox{in}\quad
\R^d_+,\quad\mbox{and}\quad \gamma w=0\quad\hbox{on}\quad\p\R^d_+.
\eeq
 \beq\label{eq:P} (\p_d+|D_h|)v=f\quad\mbox{in}\quad
\R^d_+,\quad\mbox{and}\quad \gamma v=0\quad\hbox{on}\quad\p\R^d_+.
\eeq
Here and in what follows, $\ga$ stands for  the trace operator on $\p\R^d_+.$
\medbreak
 Finally we denote by $H$ the harmonic extension operator from the hyperplane $\p\R^d_+$
 to the half-space $\R^d_+.$ More precisely, for any given $b:\p\R^d_+\to\R,$
we set  $Hb$ to be  the unique
solution going to zero at $\infty$ of \beq\label{2.4} \D Hb=0\quad\mbox{in}\quad
\R^d_+,\quad\mbox{and}\quad \gamma Hb=b\quad\hbox{on}\quad\p\R^d_+.
\eeq
 Introducing the Fourier transform $\cF_h$ with respect to the horizontal component $x_h,$ the function
 $\cF_h(Hb)$ is explicitly given by the formula
 \begin{equation}\label{eq:Hb}
 \cF_h(Hb)(\xi_h,x_d)=e^{-|\xi_h|x_d}\cF_h b(\xi_h),\quad \xi_h\in\R^{d-1},\: x_d>0,
\end{equation}
hence in particular
 \begin{equation}\label{eq:Hbb}
 (\pa_d+|D_h|)Hb=0\quad\hbox{in }\ \R^d_+.
\end{equation}
\begin{lem}\label{l:U}{\sl
For any smooth enough data $f$ decaying at $\infty,$ Equation  \eqref{eq:U} has a unique solution going
to $0$ at $\infty,$ and Equation \eqref{eq:P} has a unique solution with gradient going to $0$ at infinity.
Furthermore, denoting by $U$ and $P$ the solution operators for \eqref{eq:U} and \eqref{eq:P}, one has
\begin{equation}\label{eq:Uf}
Uf=r\,R_h\cdot S\bigl(R_h\cdot S\,e_a(f)+R_de_s(f)\bigr)
\quad\!\hbox{and}\quad\!
Pf(x_h,x_d)=\!\int_0^{x_d}\!(I-U)f(x_h,y_d)\,dy_d
\end{equation}
and the following identities are satisfied:
\begin{enumerate}
\item $\nabla_hU=U\nabla_h$;\vspace{2pt}
\item\label{eq:deux} $\p_dU=(I-U)|D_h|$;\vspace{2pt}
\item $\nabla_hP=SU$;\vspace{2pt}
\item $\p_dP=I-U$;\vspace{2pt}
\item $\Delta P=\p_d-|D_h|$;\vspace{2pt}
\item $[P,\p_d]=-H\gamma.$
\end{enumerate}}
\end{lem}
\begin{proof}
If $w$ is a solution to \eqref{eq:U} then it also satisfies the following Poisson equation:
$$
\left\{\begin{array}{l}-\Delta w=(|D_h|-\p_d)|D_h|f\\
w|_{x_d=0}=0,\end{array}\right.
$$
the unique solution (decaying to $0$ at infinity) of which  is given
by
$$
w=r(-\Delta)^{-1}e_a\bigl((|D_h|-\p_d)|D_h|f\bigr).
$$
As  $e_a(|D_h|^2f)=|D_h|^2e_a(f)$ and
$e_a(\p_d|D_h|f)=\p_d(e_s(|D_h|f))=\p_d|D_h|e_s(f),$ we get  the formula for $Uf.$
\smallbreak
It is obvious that $U$ commutes with $\nabla_h.$ As regards commutation with $\p_d,$ we notice that,
by definition of $Uf,$
$$
\p_dUf+|D_h|Uf=|D_h|f,
$$
which yields \eqref{eq:deux}.
\medbreak
It is clear that $|D_h|Pf$ satisfies \eqref{eq:U}, hence $|D_h|Pf=Uf$
and $\nabla_hPf=SUf.$
Similarly, the equation for $Pf$ yields
$$
\p_dPf=f-|D_h|Pf=f-Uf,
$$
hence integrating with respect to the vertical variable gives the expression for $P.$
\medbreak
The next item is a direct consequence of the definition of $P$ (just
apply $\p_d-|D_h|$ to the equation). Finally, we have by definition of
$P\p_df,$
$$
(\p_d+|D_h|)P\p_df=\p_df=(\p_d+|D_h|)f-|D_h|f.
$$
Hence, using \eqref{eq:Hbb},
$$
(\p_d+|D_h|)(P\p_df-f+H\gamma f)=-|D_h|f\quad\hbox{and}\quad
(P\p_df-f+H\gamma f)|_{x_d=0}=0.
$$
Therefore, using the definition of $U,$ one may write
$$
P\p_df-f+H\gamma f=-Uf.
$$
Because $\p_dP=I-U,$ it is easy to complete the proof of the last item.
\end{proof}
\begin{rmk}{\sl For functions vanishing at $x_d=0,$ operator
$U$ coincides with the expression
$$rR_h\cdot S(R_h\cdot S+R_d)e_0$$
that has been introduced in \cite{Ukai} and   plays the role of the
left-inverse  of $(\Id+|D_h|^{-1}\pa_d)$ for functions of $\R^d_+$
vanishing at $\pa\R^d_+.$ Our definition of operator $U$ is slightly
more general as it allows us to consider functions that \emph{do not
vanish at $x_d=0.$}}
\end{rmk}
The main result of this subsection reads:
\begin{thm}\label{th2.1}
{\sl Given smooth and decaying  data $u_0,$ $f$ and $g$ with $g=\dive Q,$  the unique solution  $(u, \na\Pi)$
of \eqref{eq:stokes} is given by
 \beq\label{2.5}
\begin{aligned}
 u^h=&re^{\mu t\D}e_a(V_hu_0)- SPg-SU e^{\mu
 t\D}e_a(V_du_0)\\
& -SU\int_0^te^{\mu(t-\tau)\D}e_a\bigl(\wt{N}f+Gk\bigr)\,d\tau
 +r\int_0^te^{\mu(t-\tau)\D}e_a(\wt{M}f+SGk\bigr)\,d\tau,\\
 u^d=& Pg+U e^{\mu
 t\D}e_a(V_du_0)+U\int_0^te^{\mu(t-\tau)\D}e_a\bigl(\wt{N}f+Gk\bigr)\,d\tau,
 \end{aligned}
 \eeq
 and
 \beq\label{2.6}
 \begin{aligned}
\nabla_h\Pi&=r\bigl(R_h\cdot S(SR_d+R_h)\bigr)\bigl(S\cdot e_a(f^h)+e_s(f^d)\bigr)+S(U-I)Gk\\
 &+\mu(\nabla_h-S\p_d)g+S\p_t\bigl(S\cdot UQ^h+(I-U)Q^d-H\gamma Q^d\bigr) +SU\wt Nf\\
 &+\bigl(r(|D_h|-\p_d)\nabla_h +SU\Delta\bigr)\biggl(e^{\mu t\D}e_a(V_du_0)
 +\int_0^te^{\mu(t-\tau)\D}e_a\bigl(\wt{N}f+Gk\bigr)\,d\tau\biggr),\\[2ex]
  \pa_d\Pi&=f^d +\mu(\p_d-|D_h|)g-\p_t\bigl(S\cdot UQ^h+(I-U)Q^d-H\gamma Q^d\bigr)-U\bigl(\wt Nf+Gk\bigr)\\
  &-\bigl(r(|D_h|-\p_d)|D_h|+U\Delta\bigr)\biggl(e^{\mu t\D}e_a(V_du_0)
   +\int_0^te^{\mu(t-\tau)\D }e_a\bigl(\wt{N}f+Gk\bigr)\,d\tau\biggr),
 \end{aligned}
 \eeq
 where $Gk,$ $\wt{M}f$ and $\wt{N}f$ are given by
 \beq\label{2.7}
 \begin{aligned}
Gk&=-r\bigl(R_d-R_h\cdot S\bigr)\bigl(R_h\cdot e_a(k^h)+R_de_s(k^d)\bigr)\quad\hbox{with}\quad k=\p_tQ-\mu\nabla g,\\[1ex]
 \wt{N}f&=r\bigl\{\bigl[1+R_d^2-R_dR_h\cdot S\bigr]e_s(f^d)+R_d^2S\cdot e_a(f^h)
 +R_dR_h\cdot e_a(f^h)\bigr\},\\[1ex]
\wt{M}f&=rS\bigl[R_d-R_h\cdot S\bigr]\bigl(R_h\cdot
e_a(f^h)+R_d e_s(f^d)\bigr)+V_hf.
\end{aligned}
\eeq}
\end{thm}
\begin{proof} We shall essentially follow the arguments in \cite{CPS,Ukai}.
  Note that  setting
\begin{equation}\label{eq:change}
\begin{array}{ll} u_{new}(t,x)=\mu u_{old}(\mu^{-1}t,x), &\ \Pi_{new}(t,x)=
\Pi_{old}(\mu^{-1}t,x),\\[1ex]  f_{new}(t,x)=
f_{old}(\mu^{-1}t,x),&\  g_{new}(t,x)=\mu g_{old}(\mu^{-1}t,x)\end{array}
\end{equation}
 reduces the study to the
case $\mu=1,$ an assumption that we are going to make in the rest of
the proof. \medbreak
The basic  idea is to reduce the study to that of the heat equation for the auxiliary functions
$V_hu$ and $V_du.$
 As a first step, let us compute $\Pi$ in terms
of $\dive f, g$ and of its trace at $\partial\R^d_+.$ Taking space
divergence to \eqref{eq:stokes} yields  \beno \left\{\begin{array}{l}
\displaystyle -\D \Pi=\p_tg-\D g-\dive f\quad\mbox{in}\quad \R^d_+, \\[1ex]
\displaystyle \gamma\Pi=b\quad\mbox{on}\quad\p\R^d_+\end{array}\right.\qquad \Pi\to
0\quad \mbox{as}\quad |x|\to \infty,
\eeno the solution of which is given by
$$
\Pi=r(-\D)^{-1}e_a(\dive(k-f))+Hb\quad\hbox{with }\ k=\p_tQ-\na g,
$$ which along with
\eqref{eq:Hbb} implies that
 \beq \label{2.8}
 \begin{array}{lll}
(\p_d+|D_h|)\Pi&=&-r(\p_d+|D_h|)(-\D)^{-1}e_a(\dive(f-k))\\[1ex]
&=& Me_a(\dive(f-k)) \end{array} \eeq
with
\begin{equation}\label{eq:M}
 Mh\eqdefa-r(\p_d+|D_h|)(-\D)^{-1}h.
\end{equation}
Let $z\eqdefa(\p_d+|D_h|)u^d$ and $Nf\eqdefa (\p_d+|D_h|)f^d-\p_dMe_a(\dive f).$
We infer from \eqref{eq:stokes} and \eqref{2.8} that
 \beno
\left\{\begin{array}{l}
\displaystyle \p_tz-\D z=Nf+\p_dM e_a(\dive k)\quad\mbox{in}\quad \R_+\times\R^d_+, \\
\displaystyle (z-g)|_{x_d=0}=0,\\
\displaystyle z|_{t=0}=(\p_d+|D_h|)u_0^d.
\end{array}\right.
\eeno Note from  the definition of $M$ in \eqref{eq:M} that
\begin{equation}\label{eq:Mea}
\p_dMe_a(h)=re_a(h)-r|D_h|(\p_d\!+\!|D_h|)(-\D)^{-1}\bigl[e_a(h)\bigr].
\end{equation}
Therefore, because
\begin{equation}\label{eq:div}
e_a(\dive k)=\dive_he_a(k^h)+\p_de_s(k^d)
\end{equation}
and
\begin{equation}\label{eq:Gk}
Me_a(\dive k)=Gk,
\end{equation}
 we get
  \beno
\begin{aligned}
Nf&=|D_h|f^d-r\dive_he_a(f^h)+r|D_h|(\p_d+|D_h|)(-\D)^{-1}\bigl[\dive_he_a(f^h)+\p_de_s(f^d)\bigr]\\
&=|D_h|\wt{N}f.\end{aligned}
\eeno Now, using \eqref{eq:Mea} with $h=\p_tg-\Delta g=\dive k$ and \eqref{eq:div},  we thus obtain
\beno \left\{\begin{array}{ll}
\displaystyle \p_t(z-g)-\D (z-g)=|D_h|\bigl(\wt{N}f+Gk\bigr)&\quad\mbox{in}\quad \R_+\times\R^d_+, \\
\displaystyle (z-g)|_{x_d=0}=0&\quad\mbox{in}\quad \R_+\times\p\R^d_+,\\
\displaystyle (z-g)|_{t=0}=|D_h|V_du_0&\quad\mbox{in}\quad \R^d_+.
\end{array}\right.
\eeno Taking advantage of the solution formula for the heat equation
in $\R^d_+$ with homogeneous Dirichlet boundary conditions, we
deduce that
$$\begin{aligned}
(z-g)(t)=r|D_h|\biggl(&e^{t\D}e_a(V_du_0)+\int_0^te^{(t-\tau)\D}e_a\bigl(\wt
Nf+Gk\bigr)\,d\tau\biggr).
\end{aligned}$$
As $z-g=(|D_h|+\p_d)(u^d-Pg)$ and $u^d-Pg$ vanishes at $x_d=0,$
 keeping in mind the definition of  $U,$  we get  the second equality of \eqref{2.5}.
\medbreak
 To derive the solution formula for $u^h$, we look at the equation satisfied by  $V_hu.$ Thanks to \eqref{eq:stokes}, \eqref{2.7} and \eqref{2.8}, we get, observing
 $$ V_hf-S(|D_h|+\p_d)\Pi=V_hf-SMe_a(\dive(f-k)) = \wt{M}f+SGk,$$ that
  \beno
\left\{\begin{array}{ll}
\displaystyle \p_tV_hu-\D V_hu=\wt{M}f+SGk&\quad\mbox{in}\quad \R_+\times\R^d_+,  \\
\displaystyle V_hu|_{x_d=0}=0&\quad\mbox{in}\quad \R_+\times\R^d_+, \\
\displaystyle V_hu|_{t=0}=V_hu_0&\quad\mbox{in}\quad \R^d_+,
\end{array}\right.
\eeno so that $$
V_hu(t)=re^{t\D}e_a(V_hu_0)+r\int_0^te^{(t-\tau)\D}e_a\bigl(\wt{M}f+SGk\bigr)\,d\tau.$$
As $u^h=V_hu-Su^d,$ combining the above identity and the second
formula of \eqref{2.5} yields
\beq\label{vhu0}
u^h(t)=re^{t\Delta}e_a(V_hu_0)+r\int_0^te^{(t-\tau)\Delta}e_a(\wt Mf+SGk)\,d\tau-Su^d
\eeq
whence  the solution formula for $u^h.$
\medbreak Let us finally derive \eqref{2.6}.  By virtue of \eqref{eq:stokes} and \eqref{2.5}, we may write
 \beno \begin{aligned} \p_d\Pi&=f^d-(\p_t-\D) u^d\\
&=f^d-(\p_t-\D)Pg-(\p_t-\D)U\biggl(e^{t\D}e_a(V_du_0)
+\int_0^te^{(t-\tau)\D}e_a\bigl(\wt{N}f+Gk\bigr)\,d\tau\biggr).
\end{aligned}
\eeno However, because
$$
\Delta Uh=r(\p_d-|D_h|)|D_h|h,
$$
one may write
 \beno
(\p_t-\D)Uh= U(\p_t-\D)h+r(|D_h|-\p_d)|D_h|h+U\Delta h.\eeno
Note also that, by virtue of Lemma \ref{l:U},
$$
\p_tPg=\p_t(P\dive_hQ^h+P\p_dQ^d)=\p_t(S\cdot UQ^h+(I-U)Q^d-H\gamma Q^d)
$$
and that
$$
\Delta Pg=(\p_d-|D_h|)g,
$$
 which ensures that
  \begin{multline}\label{2.10}
\p_d\Pi=f^d-\p_t(S\cdot UQ^h+(I-U)Q^d-H\gamma Q^d)+(\p_d-|D_h|)g-U\bigl(\wt Nf+Gk\bigr)
\\-\bigl(r(|D_h|-\p_d)|D_h|+U\Delta\bigr)\biggl(e^{t\D}e_a(V_du_0)
+\int_0^te^{(t-\tau)\D}e_a(\wt{N}f+Gk)\bigr)\,d\tau\biggr)\cdotp
\end{multline}
On the other hand, by virtue of \eqref{2.8}, one has
 \beno
|D_h|\Pi=- \p_d\Pi+ Me_a(\dive(f-k)), \eeno and  \beno
\begin{aligned}
Me_a(\dive f)&=-(\p_d+|D_h|)(-\D)^{-1}\bigl(\dive_he_a(f^h)+\p_de_s(f^d)\bigr)\\
&=e_s(f^d)-(\p_d+|D_h|)(-\D)^{-1}\bigl(\dive_he_a(f^h)+|D_h|e_s(f^d)\bigr),
\end{aligned}
\eeno which together with \eqref{2.10} gives rise to \eqref{2.6}.
This completes the proof of Theorem \ref{th2.1}.
\end{proof}
The following remark will be the key to the proof of the anisotropic smallness condition in Theorem \ref{th:main}
as it gives an expression of the horizontal components of the free solution to the Stokes system \emph{independent of the vertical component}.
\begin{rmk} In the case $f\equiv0$ and $g\equiv0,$
if we further  assume that $\dive u_0=0$ and $\ga u_0^d=0$ then the horizontal components of
$u$ also obey the following formula:
\begin{equation}\label{add.2}
u^h(t)=\bigl(r+SUS\cdot\bigr)e^{\mu t\D}e_a(u_0^h)+rS\bigl(R_dR_he_s+R_h\cdot
SR_he_a\bigr)\cdot e^{\mu t\D}e_s(u_0^h).\end{equation}
\end{rmk}
\begin{proof}
Given our assumptions, formula \eqref{2.5} reduces to
$$
u^h(t)=re^{\mu t\Delta}e_a(V_hu_0)-SUe^{\mu t\Delta}e_a(V_du_0).
$$
Hence, by virtue of \eqref{2.1},
\beq\label{add.19} u^h(t)=\bigl(r+SUS\cdot\bigr)e^{\mu t\D}e_a(u_0^h)+(r-U)Se^{\mu
t\D}e_a(u_0^d).
\eeq Now, as $rSe^{\mu t\D}e_a(u_0^d)=rSe_a\bigl(e^{\mu t\D}e_a(u_0^d)\bigr),$ we deduce from \eqref{eq:Uf} that
$$
\begin{array}{lll}
(r-U)Se^{\mu t\D}e_a(u_0^d)&\!\!\!=\!\!\!&rS\bigl(1-(R_h\cdot
S)^2\bigr)e_a\bigl(e^{\mu t\D}e_a(u_0^d)\bigr)-rSR_h\cdot
SR_de_s\bigl(e^{\mu
t\D}e_a(u_0^d)\bigr)\\[1ex]
&\!\!\!=\!\!\!&-rS|D|^{-2}\p_d^2e_a\bigl(e^{\mu
t\D}e_a(u_0^d)\bigr)-rSR_h\cdot S|D|^{-1}\p_de_s\bigl(e^{\mu
t\D}e_a(u_0^d)\bigr).
\end{array}
$$ Note that $\ga\bigl(e^{\mu t\D}e_a(u_0^d)\bigr)=0$ implies
$\p_de_a\bigl(e^{\mu t\D}e_a(u_0^d)\bigr)=e_s\bigl(e^{\mu
t\D}\p_de_a(u_0^d)\bigr),$ so that there holds
$$-rS|D|^{-2}\p_d^2e_a\bigl(e^{\mu t\D}e_a(u_0^d)\bigr)=-rS|D|^{-2}\p_de_s\bigl(e^{\mu
t\D}\p_de_a(u_0^d)\bigr).$$

As $\ga u_0^d=0$ implies
$\p_de_a(u_0^d)=e_s(\p_du_0^d),$ and because  $\dive u_0=0,$ we thus get
 \beno -rS|D|^{-2}\p_d^2e_a\bigl(e^{\mu t\D}e_a(u_0^d)\bigr)
=rS|D|^{-2}\p_de_s\bigl(e^{\mu t\D}e_s(\dive_hu_0^h)\bigr) \eeno
and \beno rSR_h\cdot S|D|^{-1}\p_de_s\bigl(e^{\mu
t\D}e_a(u_0^d)\bigr)=-rSR_h\cdot S|D|^{-1}e_a\bigl(e^{\mu
t\D}e_s(\dive_hu_0^h)\bigr). \eeno Therefore, we arrive at \beno
(r-U)Se^{\mu t\D}e_a(u_0^d)=rS\bigl(R_dR_he_s+R_h\cdot
SR_he_a\bigr)\cdot e^{\mu t\D}e_s(u_0^h). \eeno
Inserting this latter equality in
 \eqref{add.19} leads to \eqref{add.2}.
\end{proof}

%%%%%%%%%%%%%%%%%%%%%%%%%%%%%%%%%%%%%%%%%%%%%%%%%

\subsection{A priori estimates}

Let us first briefly recall the definition of homogeneous Besov spaces in $\R^d.$
 Let  $\chi:\R^d\rightarrow[0,1]$ be  a smooth nonincreasing radial function
supported in $B(0,1)$ and such that $\chi\equiv1$ on $B(0,1/2),$ and let
$$\varphi(\xi)\eqdefa\chi(\xi/2)-\chi(\xi).$$
The homogeneous Littlewood-Paley
decomposition of any tempered distribution $u$ on $\R^{d}$
is defined by
$$\dot\Delta_ku\eqdefa\varphi(2^{-k}D)u=
{\mathcal F}^{-1}\bigl(\varphi(2^{-k}\cdot){\mathcal F}u\bigr),\qquad k\in\Z
$$
where  $\cF$ stands for the Fourier transform on $\R^{d}.$
\begin{defi}\label{def:Besov} {\sl For any $s\in\R$ and $(p,r)\in[1,+\infty]^2,$
the homogeneous Besov space $\dot B^s_{p,r}(\R^d)$ stands
for the set of tempered distributions $f$ so that
$$
\|f\|_{\dot B^s_{p,r}(\R^{d})}\eqdefa \bigl\|2^{sk}\|\dot\Delta_kf\|_{L^p(\R^{d})}\bigr\|_{\ell^r(\Z)}<\infty
$$ and for all smooth compactly supported function $\theta$
  over $\R^{d},$ we have
  \begin{equation}\label{eq:lf}
  \lim_{\lambda\rightarrow+\infty}\theta(\lambda D)f=0\quad\hbox{in}\quad
  L^\infty(\R^{d}).
  \end{equation}}
\end{defi}
\begin{rmk}{\sl  Condition \eqref{eq:lf} means that functions in homogeneous Besov spaces are required to have some decay at infinity
(see \cite{BCD} for more details).
In particular, we have
\begin{equation}\label{eq:LP}
f=\sum_{k\in\Z} \dot \Delta_k f\quad\hbox{in }\ \cS'(\R^d)
\end{equation}
  whenever $f$ satisfies \eqref{eq:lf}.
In this paper, we will only consider exponents $s<d/p$ so that for $f$ with
finite $\dot B^s_{p,r}(\R^d)$ semi-norm, \eqref{eq:lf} and \eqref{eq:LP} are equivalent.}
\end{rmk}
The homogeneous Besov spaces on the half-space are defined by restriction:
\begin{defi} {\sl For any $s\in\R,$ and $(p,r)\in[1,+\infty]^2$ we denote by
$\dot B^s_{p,r}(\R^d_+)$ the set of distributions $u$ on $\R^d_+$  admitting
some extension $\wt u\in \dot B^s_{p,r}(\R^d)$ on $\R^d.$
Then we set
$$\|u\|_{\dot B^s_{p,r}(\R^d_+)}\eqdefa \inf \|\wt u\|_{\dot B^s_{p,r}(\R^d)}$$
where the infimum is taken on all the extensions of $u$ in $\dot B^s_{p,r}(\R^d).$}
\end{defi}
We also need to introduce some spaces of divergence free
vector fields  vanishing at the boundary $\p\R^d_+.$
We proceed as follows:
\begin{defi}\label{d:nulltrace} {\sl For $1<p<\infty$ and $0<s<2,$ we denote by $\dot \cB^s_{p,r}(\R^d_+)$
the completion of the set of divergence free vector fields with
coefficients in $W^{2,p}(\R^d_+)\cap W^{1,p}_0(\R^d_+)$ where
$W^{1,p}_0(\R^d_+)$ stands for the subspace of $W^{1,p}(\R^d_+)$
functions with null trace at $\p\R^d_+$ for the norm
$\|\cdot\|_{\dot B^s_{p,r}(\R^d_+)}.$}
\end{defi}
It is classical (see  e.g. \cite{DM}) that spaces $(\dot
B^s_{p,r}(\R^d_+))^d$ (with the divergence free condition) and
$\dot\cB^s_{p,r}(\R^d_+)$ coincide whenever $1<p,r<\infty$ and
$0<s<1/p.$
%Let us also point out that the above definition ensures the
%operators $e_0,$ $e_s$ and $e_a$ to map  $\dot B^s_{p,r,0}(\R^d_+)$ in  $\dot B^s_{p,r}(\R^d).$
\medbreak
The following result extends Lemma 3.2 of \cite{Ukai} to the context of Besov spaces.
\begin{lem}\label{l:extension}{\sl
 Operators $R_h,$ $R_d$ and $S$  map $L^p(\R^d)$ in itself for any $1<p<\infty,$
 and, with no restriction on $s,p,r,$ we have
 $$
 \|\cR z\|_{\dot B^s_{p,r}(\R^d)}\leq C \|z\|_{\dot B^s_{p,r}(\R^d)}\quad\hbox{for }\
 \cR\in \{R_h,R_d,S\}.
 $$
  Operators  $V_h,$ $V_d,$  $U,$ $G,$ $\wt M$ and $\wt N$ map $L^p_+$ in itself if $1<p<\infty,$
 and  $\dot B^s_{p,r}(\R^d_+)$ in itself  if  $1<p,r<\infty$ and $0<s<2.$}
\end{lem}
\begin{proof}
The result in the Lebesgue spaces just follows  from the fact that all those operators are combinations of  Riesz transforms
so that  Calderon-Zygmund theorem applies.
The result in homogeneous Besov spaces stems from the fact that the
Riesz operators are Fourier multipliers of degree $0,$ hence map any homogeneous Besov space in itself.
\end{proof}

We are now ready to establish a first family of a priori estimates for System \eqref{eq:stokes}.
\begin{prop}\label{p:stokes}  {\sl Let $1<p,r<\infty$ and  the data $u_0,$ $f,$  $g$
fulfill the following hypotheses:
\begin{itemize}
\item  $u_0\in\dot \cB^{2-\f2r}_{p,r}(\R^d_+)$;\vspace{2pt}
\item $f\in L^r(\R_+;L^p_+)$;\vspace{2pt}
\item $g$ is locally integrable, $\nabla g\in L^r(\R_+;L^p_+),$  $g =\dive Q$ with $\p_tQ\in L^r(\R_+;L^p_+)$
and the following compatibility conditions are fulfilled\footnote{That  $u_0^d$
has a trace is ensured by  $u_0\in\dot \cB^{2-\f2r}_{p,r}(\R^d_+).$
 As for $Q,$ it stems from the fact that $\dive Q$ is quite smooth.}
\begin{equation}\label{eq:comp} \gamma u_0^d=0,\quad
g|_{t=0}=0, \quad\hbox{and}\quad \p_t(\gamma Q^d)=0.
\end{equation}
\end{itemize}
 Then System \eqref{eq:stokes}  has a unique solution $(u,\nabla\Pi)$
with
$$
u\in\cC_b(\R_+;\dot B^{2-\f2r}_{p,r}(\R^d_+))\quad{and}\quad
\pa_tu, \nabla^2u,\nabla\Pi\in L^r(\R_+;L^p_+),
$$
with also $\nabla u\in L^q(\R_+;L^m_+)$ whenever
 $q\in [r,\infty)$ and $m\in[p,\infty]$  satisfy
$$0\leq1-\f2r+\f2q\leq\f dp\quad\hbox{and}\quad \f dm=\f dp-1+\f2r-\f2q\cdotp$$
Furthermore,  the following inequality is fulfilled for all $T>0$:
\begin{multline}\label{eq:ineq1}
\mu^{1-\f1r}\|u\|_{L_T^\infty(\dot B^{2-\f2r}_{p,r}(\R^d_+))}
+\mu^{1-\f1r+\f1q}\|\nabla u\|_{L_T^q(L^m_+)}+\|(\pa_tu,\mu\nabla^2u,\nabla\Pi)\|_{L^r_T(L^p_+)}\\
\leq C\Bigl(\mu^{1-\f1r}\|u_0\|_{\dot B^{2-\f2r}_{p,r}(\R^d_+)}+\|(f,\mu\nabla g,\p_tQ)\|_{L^r_T(L^p_+)}\Bigr)\cdotp
\end{multline}
Finally, if $g\equiv0$ then we have  $u\in\cC_b(\R_+;\dot
\cB^{2-\f2r}_{p,r}(\R^d_+)).$ }\end{prop}
\begin{rmk}{\sl
As regards the bounds for $\nabla u,$ we shall  often  use the following two cases:
\begin{itemize}
\item $p=\f{dr}{3r-2},$ $q=2r$ and $m=\f{dr}{2r-1},$
\item $p>d,$ $q=\f{2p}{d+(\f2r-1)p}$ and $m=+\infty.$
\end{itemize}}
\end{rmk}
\begin{proof}
 We concentrate on the proof of the estimates in
$L^r(0,T;L^p_+)$ for $\nabla^2u$ and $\nabla\Pi.$ Indeed,
once the pressure has been determined, $u$ may be seen as a solution
of the heat equation with source term in $L^r(\R_+;L^p_+),$ and is thus given by
\begin{equation}\label{eq:heatformula}
u(t)=r\biggl(e^{\mu t\Delta}e_a(u_0)+\int_0^te^{\mu(t-\tau)\Delta}e_a(f-\nabla \Pi)\,d\tau\biggr)\cdotp
\end{equation}
Therefore combining Corollary \ref{c:besov}, Lemma \ref{lem3.1} and
Lemma \ref{l:extension} allows to bound $u(t)$ in $\dot
B^{2-\f2r}_{p,r}(\R^d_+)$ in terms of  the data and of the norm of
$\nabla \Pi$ in $L^r(0,t;L^p_+).$
In addition, because any function in  $L^r(\R_+;L^p_+)$ may be approximated by
smooth functions compactly supported in $\R_+\times\R^d_+,$ and because $u_0$ may be approximated
by functions in $W^{2,p}(\R^d_+)\cap W^{1,p}_0(\R^d_+),$
the above formula guarantees that $u$ is continuous in time, with values in  $\dot B^{2-\f2r}_{p,r}(\R^d_+)$
(or in $\dot\cB^{2-\f2r}_{p,r}(\R^d_+)$ if $\dive u\equiv0$).
\medbreak
In what follows, we assume that $\mu=1,$ which is not restrictive
owing to  the change of variables \eqref{eq:change}.
Of course, when proving estimates,  one may  consider separately the
three cases where only  one element of the triplet $(u_0,f,g)$ is
nonzero, a consequence of the fact that \eqref{eq:stokes} is linear.

\subsubsection*{Step 1. Case $u_0\equiv0$ and $f\equiv0$}

Then the formula for $u^d$ given by Theorem \ref{th2.1} reduces to
$$
u^d=Pg+U\int_0^te^{(t-\tau)\Delta}e_a(Gk)\,d\tau,
$$
and using the algebraic relations provided by Lemma \ref{l:U} thus yields
$$
\begin{aligned}
&\nabla_h^2u^d=SU\nabla_h g+U\int_0^te^{(t-\tau)\Delta}\nabla_h^2e_a(Gk)\,d\tau,\\
&\nabla_h\partial_du^d=(I-U)\nabla_hg+(I-U)\int_0^te^{(t-\tau)\Delta}\nabla_h|D_h|e_a(Gk)\,d\tau,\\
&\p_d^2u^d=(\p_d+(U\!-\!I)|D_h|)g+r\int_0^t\!e^{(t-\tau)\Delta}\p_d|D_h|e_a(Gk)\,d\tau+(I\!-\!U)\!\int_0^t\!e^{(t-\tau)\Delta}\D_he_a(Gk)\,d\tau.
\end{aligned}
$$
The important fact is that all the terms corresponding initially to $Pg$ may
be written  $A\nabla g$ where $A$ stands for some $0$-th order operator
for which Lemma \ref{l:extension} applies.
%Therefore we have$$\|\nabla^2Pg\|_{L^r_T(L^p_+)}\leq C\|\nabla g\|_{L^r_T(L^p_+)}.$$
A similar observation holds for the terms with the time integral so that applying  Lemma \ref{lem3.1} eventually yields
\begin{equation}\label{eq:u1}
\|\nabla^2u^d\|_{L^r_T(L^p_+)}\leq C\bigl(\|\nabla g\|_{L^r_T(L^p_+)}
+\|Gk\|_{L^r_T(L^p_+)}\bigr).
\end{equation}
At this point, we use Lemma \ref{l:extension} to bound the right-hand side by
$\|(\nabla g,\p_tQ)\|_{L^r_T(L^p_+)},$ and   we thus get
\begin{equation}\label{eq:u2}
\|\nabla^2u^d\|_{L^r_T(L^p_+)}
\leq C\|(\nabla g,\p_tQ)\|_{L^r_T(L^p_+)}.
\end{equation}
It is clear that $\nabla^2u^h$ also satisfies \eqref{eq:u2}: indeed \eqref{vhu0} gives
\begin{equation}\label{eq:u3}
\nabla^2u^h=r\int_0^te^{(t-\tau)\Delta}\nabla^2Se_a(Gk)\,d\tau
-S\nabla^2 u^d.
\end{equation}
Let us now concentrate on the pressure.
Keeping in mind \eqref{2.10} and \eqref{eq:comp}, we may write
$$\displaylines{
 \quad\pa_d\Pi=(\p_d-|D_h|)g-S\cdot U\p_tQ^h-(I-U)\p_tQ^d-UGk
  \hfill\cr\hfill-\bigl(r(|D_h|-\p_d)|D_h|+U\Delta\bigr)\int_0^te^{(t-\tau)\D }e_a(Gk)\,d\tau.\quad}
$$
Therefore, combining Lemmas \ref{l:extension}  and  \ref{lem3.1} gives
\begin{equation}\label{eq:pi2}
\| \p_d\Pi  \|_{L^r_T(L^p_+)}\leq  C\|(\p_tQ,\nabla
g)\|_{L^r_T(L^p_+)}.
\end{equation}
Finally, because
\begin{equation}\label{eq:pi3}
\nabla_h\Pi=-SGk-S\p_d\Pi,
\end{equation}
it is clear that $\nabla_h\Pi$ also satisfies \eqref{eq:pi2}.

\subsubsection*{Step 2. Case $f\equiv0$ and $g\equiv0$}

With no loss of generality, one may assume that $u_0\in W^{1,p}_0(\R^d_+)\cap W^{2,p}(\R^d_+).$
 {}From Theorem \ref{th2.1} and Lemma \ref{l:U}, we readily get
$$
\begin{aligned}
\nabla_h^2u^d&=Ue^{t\Delta}\nabla_h^2e_a(V_d(u_0)),\\
\nabla_h\partial_du^d&=(I-U)e^{t\Delta}\nabla_h|D_h|e_a(V_d(u_0)),\\
\p_d^2u^d&=re^{t\Delta}\p_d|D_h|e_a(V_d(u_0))+r(I-U)e^{t\Delta}e_a(\D_hV_d(u_0)).
\end{aligned}
$$
Therefore, combining Corollary \ref{c:besov} and Lemma \ref{l:extension},
$$\begin{array}{lll}
\|\nabla^2u^d\|_{L^r_T(L^p_+)}&\!\!\!\leq\!\!\!&  C\Bigl(
\|e_a(\nabla_h^2V_d(u_0))\|_{\dot
B^{-\f2r}_{p,r}(\R^d)}\\&&\hspace{1cm}+\|e_a(\nabla_h|D_h|V_d(u_0))\|_{\dot
B^{-\f2r}_{p,r}(\R^d)}
+\|e_s(\p_d|D_h|V_d(u_0))\|_{\dot B^{-\f2r}_{p,r}(\R^d)}\Bigr)\\&\!\!\!\leq\!\!\!&
C\|u_0\|_{\dot B^{2-\f2r}_{p,r}(\R^d_+)}.
\end{array}
$$
Note that in order to bound the last term, we used the fact that because  $V_du_0$ is null at the boundary, we have
$$|D_h|\p_d e_a(V_d(u_0))=e_s(\p_d|D_h|V_d(u_0))\in \dot B^{-\frac 2r}_{p,r}(\R^d).$$
Owing to \eqref{vhu0}, $\nabla^2u^h$ satisfies the same inequality.
Finally,
$$
\p_d\Pi=-\bigl(r(|D_h|-\p_d)|D_h|+U\Delta\bigr)e^{t\Delta}e_a(V_du_0)
$$
hence, according to Lemma \ref{l:extension} and Corollary \ref{c:besov},
$$\begin{aligned}
\|\p_d\Pi\|_{L^r_T(L^p_+)}&\leq  C\|\nabla^2e^{t\Delta}e_a(V_du_0)\|_{L^r_T(L^p)}\\
&\leq C\|u_0\|_{\dot B^{2-\f2r}_{p,r}(\R^d_+)}.
\end{aligned}
$$
Of course, \eqref{eq:pi3} implies that $\nabla_h\Pi$ has the same bound.

\subsubsection*{Step 3. Case $u_0\equiv0$ and $g\equiv0$}

As in the previous steps, owing to $$
\begin{array}{lll}
\na^2u^h\!\!\!&=&\!\!\! r\int_0^te^{(t-\tau)\D}\na^2e_a(\wt{M}f)\,d\tau-S\na^2u^d,\\[1ex]
\na_h\Pi\!\!\!&=&\!\!\! SGf-S\p_d\Pi,
\end{array}
$$
 it suffices to bound $\nabla^2u^d$ and $\p_d\Pi.$
The formulae for the second spatial derivatives of $u^d$ now read
$$
\begin{aligned}
\nabla_h^2u^d&=U\int_0^te^{(t-\tau)\Delta}\nabla_h^2e_a(\wt Nf)\,d\tau,\\
\nabla_h\partial_du^d&=(I-U)\int_0^te^{(t-\tau)\Delta}\nabla_h|D_h|e_a(\wt Nf)\,d\tau,\\
\p_d^2u^d&=r\int_0^te^{(t-\tau)\Delta}\p_d|D_h|e_a(\wt
Nf)\,d\tau+(I-U) \int_0^te^{(t-\tau)\Delta}\D_he_a(\wt Nf)\,d\tau.
\end{aligned}
$$
Therefore applying Lemmas \ref{l:extension} and  \ref{lem3.1},
$$
\begin{aligned}
\|\nabla^2u^d\|_{L^r_T(L^p_+)}&\leq  C\|e_a(\wt Nf)\|_{L^r_T(L^p_+)}\\
&\leq C\|f\|_{L^r_T(L^p_+)}.
\end{aligned}
$$
For the pressure, we have
$$
\p_d\Pi-f^d=-U\wt Nf    -\bigl(r(|D_h|-\p_d)|D_h|+U\Delta\bigr)\int_0^te^{(t-\tau)\Delta}e_a(\wt Nf)\,d\tau,
$$
therefore, using once again Lemmas  \ref{l:extension} and  \ref{lem3.1}, we obtain
$$
\begin{aligned}
\|\p_d\Pi-f^d\|_{L^r_T(L^p_+)}&\leq  C\bigl(\|U\wt Nf\|_{L^r_T(L^p_+)}+\|e_a(\wt Nf)\|_{L^r_T(L^p)}\bigr)\\
&\leq  C\|f\|_{L^r_T(L^p_+)}.
\end{aligned}
$$

\subsubsection*{Step 4. Estimates for $\nabla u$}
The starting point is  the following classical Gagliardo-Nirenberg inequality on $\R^d$:
\begin{equation}\label{eq:GN}
\|\nabla z\|_{L^m(\R^d)}\leq C\|z\|_{\dot B^{2-\f2r}_{p,r}(\R^d)}^{1-\theta}\|\nabla^2z\|_{L^p(\R^d)}^\theta
\end{equation}
with $\theta\in(0,1],$ $m\geq p,$ $0\leq 1-\f2r+\f{2\theta}r\leq\f
dp$ and $\f dm=\f dp-1+\f2r-\f{2\theta}r\cdotp$
 \smallbreak  If $\th\in (0,1)$ then this inequality  may be easily proved by
decomposing $u$ into low and high frequencies by means of an
homogeneous Littlewood-Paley decomposition (see e.g. \cite{BCD}
Chap. 2 for the proof of similar inequalities). The case  $\th=1$ corresponds to the
classical Sobolev inequality. We omit the proof as it is standard.
  \smallbreak We claim that this inequality extends to the half-space setting if considering
functions $u\in W^{1,p}_0(\R^d_+)\cap W^{2,p}(\R^d_+).$ Indeed, we
observe that for such functions we have the following identities:
$$
\nabla_h(e_au)=e_a(\nabla_hu)\quad\hbox{and}\quad
\p_d(e_au)=e_s(\p_du).
$$
As $\nabla_h(e_au)$ also has null trace at $\pa\R^d_+,$  one can thus write
$$
\nabla_h^2(e_au)=e_a(\nabla_h^2u)\quad\hbox{and}\quad
\p_d\nabla_h(e_au)=e_s(\p_d\nabla_hu).
$$
Even though  $\p_d(e_au)$ need not be zero at  $\pa\R^d_+,$ it is symmetric with
respect to the vertical variable, whence
$$
\p_d^2(e_au)=e_a(\p_d^2u).
$$
%Now, the hypothesis  $2-2/r<1/p$ enables us  to take advantage of the continuity results for $e_a$ and $e_s$ given in Lemma  \ref{l:extension}.
Applying \eqref{eq:GN} to $z=e_au,$  and
the above relations for the second order derivatives, we thus gather
$$\begin{array}{lll}
 \|\nabla u\|_{L^m_+}&\leq&  \|\nabla(e_au)\|_{L^m(\R^d)},\\[1.5ex]
 &\leq&  C\|e_au\|_{\dot B^{2-\f2r}_{p,r}(\R^d)}^{1-\theta}\|\nabla^2(e_au)\|_{L^p(\R^d)}^\theta,\\[1.5ex]
 &\leq& C\|u\|_{\dot B^{2-\f2r}_{p,r}(\R^d_+)}^{1-\theta}\|\nabla^2u\|_{L^p_+}^\theta.
\end{array}
$$
Hence, taking the $L^q$ norm with respect to time of both sides (with $q=r/\theta$), we discover that
$$
 \|\nabla u\|_{L^q_T(L^m_+)}\leq  C\|u\|_{L^\infty_T(\dot B^{2-\f2r}_{p,r}(\R^d_+))}^{1-\theta}
 \|\nabla^2u\|_{L^r_T(L^p_+)}^\theta.
 $$
Bounding the right-hand side according to the previous steps leads
to the desired estimate of $ \|\nabla u\|_{L^q_T(L^m_+)}$ in \eqref{eq:ineq1}.
%
%While it follows from \eqref{vhu0} that \beno\p_du^h=re^{t\D}e_s(V_h(\p_du_0))+r\p_d\int_0^te^{(t-\tau)\D}e_a\bigl(\wt
%M f+SGk\bigr)\,d\tau. \eeno As $m\geq m$ and $q\geq r,$ we have
%$\dot{B}^{2-\f2r}_{p,r}(\R^d_+)\hookrightarrow\dot{B}^{2-\f2r-d(\f1p-\f1m)}_{p,r}(\R^d_+),$ and thus according to
%Corollary \ref{c:besov}, one has  \begin{multline*}\|e^{t\D}e_s(V_h(\p_du_0))\|_{L^q_t(L^m)}\lesssim
%\|e_s(V_h(\p_du_0))\|_{\dot{B}^{1-\f2r-d(\f1p-\f1m)}_{m,r}(\R^d)}\\\lesssim
%\|V_h(\p_du_0)\|_{\dot{B}^{1-\f2r-d(\f1p-\f1m)}_{m,r}(\R^d_+)}\lesssim
%\|u_0\|_{\dot{B}^{2-\f2r-d(\f1p-\f1m)}_{m,r}(\R^d_+)}\lesssim
%\|u_0\|_{\dot{B}^{2-\f2r}_{p,r}(\R^d_+)}.\end{multline*}
%While again by virtue of \eqref{eq:GN} and Lemma \ref{lem3.1}, onehas  \begin{multline*}
%\bigl\|\p_d\int_0^te^{(t-\tau)\D}e_a\bigl(\wt M f+SGk\bigr)\,d\tau\|_{L^q_t(L^m)} \lesssim \bigl\|\int_0^te^{(t-\tau)\D}e_a\bigl(\wt M
%f+SGk\bigr)\,d\tau\bigr\|_{L^\infty_t(\dot B^{2-\f2r}_{p,r})}^{1-\th}\\ \times
%\bigl\|\na^2\int_0^te^{(t-\tau)\D}e_a\bigl(\wt Mf+SGk\bigr)\,d\tau\bigr\|_{L^r_t(L^p)}^{\th}\\
%\lesssim \|e_a\bigl(\wt M f+SGk\bigr)\|_{L^r_t(L^p)}\lesssim\|(f,\na g, \p_tQ)\|_{L^r_t(L^p_+)}.\end{multline*}
%Therefore, we obtain \beno \|\p_du^h \|_{L^q_t(L^m)}\leq
%C\bigl(\|u_0\|_{\dot{B}^{2-\f2r}_{p,r}(\R^d_+)}+\|(f,\na g,\p_tQ)\|_{L^r_t(L^p_+)}\bigr). \eeno The same estimate holds for
%$\|\na_h u^h \|_{L^q_t(L^m)}.$ This completes the proof of the proposition.
\end{proof}

In order to solve System \eqref{INS} for more general data, it will be
suitable  to extend the above estimates to the case where the index of regularity of $u_0$ \emph{is not} related to $r.$
This motivates the following statement:
\begin{prop}\label{p:stokesbis}  {\sl Let $1<p,r<\infty$ and $0<s<2.$   Let $(u_0,f,g)$ satisfies the compatibility conditions
of Proposition \ref{p:stokes}, and be such that
$$u_0\in\dot \cB^{s}_{p,r}(\R^d_+),\quad
t^\alpha(f,\nabla g,\p_tQ)\in L^r(\R_+;L^p_+)\ \hbox{ with }\ \alpha\eqdefa1-\frac s2-\frac1r\cdotp
$$
If $(u,\nabla\Pi)$ is a solution of System \eqref{eq:stokes}  with $t^\alpha(\p_tu,\nabla^2u,\nabla\Pi)\in L^r(\R_+;L^p_+)$ then for all $T>0,$
\begin{equation}\label{eq:Pi}
\|t^\alpha(\p_tu,\nabla^2u,\nabla\Pi)\|_{L^r_T(L^p_+)}\leq C\Bigl(\|u_0\|_{\dot B^s_{p,r}(\R^d_+)}
+\|t^\alpha(f,\nabla g,\p_tQ)\|_{L^r_T(L^p_+)}\Bigr)\cdotp
\end{equation}
Furthermore, the following properties hold true:
\begin{enumerate}
\item For any couple $(p_2,r_2)$ so that
$$
s<1+\frac dp-\frac d{p_2}< 2+\frac2{r_2}-\frac2r,\quad p_2\geq p,\quad r_2\geq r
$$
we have  $t^\beta\nabla u\in L^{r_2}(\R_+;L^{p_2}_+)$ with
$\beta=\frac12-\frac{d}{2p_2}+\frac d{2p}-\frac s2-\frac1{r_2},$ and
$$
\|t^\beta\nabla u\|_{L^{r_2}_T(L^{p_2}_+)}\leq C\Bigl(\|u_0\|_{\dot B^{s-\frac dp+\frac d{p_2}}_{p_2,r_2}(\R^d_+)}
+\|t^\alpha(f,\nabla g,\p_tQ)\|_{L^r_T(L^p_+)}\Bigr)\cdotp
$$
\item
For any couple $(p_3,r_3)$ so that
$$s<\frac dp-\frac d{p_3}<2+\frac2{r_3}-\frac2r,\quad p_3\geq p,\quad r_3\geq r,$$
we have $t^\gamma u\in L^{r_3}(\R_+;L^{p_3}_+)$
with  $\gamma=-\frac{d}{2p_3}+\frac d{2p}-\frac s2-\frac1{r_3},$ and
$$
\|t^\gamma u\|_{L^{r_3}_T(L^{p_3}_+)}\leq C\Bigl(\|u_0\|_{\dot B^{s-\frac dp+\frac d{p_3}}_{p_3,r_3}(\R^d_+)}
+\|t^\alpha(f,\nabla g,\p_tQ)\|_{L^r_T(L^p_+)}\Bigr)\cdotp
$$\end{enumerate}
}\end{prop}
\begin{proof}
Let us first assume that only $g$ is nonzero.
Then we start with the formula
$$
t^\alpha\nabla^2_hu^d=SUt^\alpha\nabla_hg+Ut^\alpha\int_0^te^{(t-\tau)\Delta}\nabla^2_he_a(Gk)\,d\tau.
$$
{}From  Lemmas \ref{l:extension} and \ref{l:A}, we immediately infer that, if $\alpha r'<1$ then for all $T>0,$
$$
\|t^\alpha\nabla^2_hu^d\|_{L^r_T(L^p_+)}\leq C\bigl(\|t^\alpha\nabla_hg\|_{L^r_T(L^p_+)}
+\|t^\alpha e_a(Gk)\|_{L^r_T(L^p(\R^d))}\bigr),
$$
 whence
\begin{equation}\label{eq:1}
\|t^\alpha\nabla^2_hu^d\|_{L^r_T(L^p_+)}\leq C\|t^\alpha(\nabla g,\p_tQ)\|_{L^r_T(L^p_+)}.
\end{equation}
Similarly, as
$$
t^\alpha\nabla_h\p_du^d=(I-U)t^\alpha\nabla_hg+(I-U)t^\alpha\int_0^te^{(t-\tau)\Delta}\nabla_h|D_h|e_a(Gk)\,d\tau,
$$
and
$$\displaylines{
t^\alpha\p_d^2u^d=(\p_d+(U-I)|D_h|)t^\alpha
g+t^\alpha\int_0^te^{(t-\tau)\Delta}\p_d|D_h|e_a(Gk)\,d\tau
\hfill\cr\hfill+(I-U)t^\alpha\int_0^te^{(t-\tau)\Delta}\Delta_h
e_a(Gk)\,d\tau,}
$$
it is clear that $\|t^\alpha\nabla^2u^d\|_{L^r_T(L^p_+)}$ is bounded by the right-hand side of \eqref{eq:1}. Because
$$
t^\alpha \nabla^2 u^h=r\,t^\alpha\int_0^te^{(t-\tau)\Delta}\nabla^2Se_a(Gk)\,d\tau
-t^\alpha S\nabla^2u^d,$$
the same inequality holds true for $t^\alpha \nabla^2 u^h.$
\smallbreak
In order to bound the pressure, we use the fact that
$$\displaylines{
t^\alpha\p_d\Pi=t^\alpha(\p_d-|D_h|)g-t^\alpha S\cdot U\p_tQ^h-(I-U)t^\alpha \p_tQ^d-t^\alpha UGk\hfill\cr\hfill
-\bigl(r(|D_h|-\p_d)|D_h|+U\Delta\bigr)t^\alpha\int_0^te^{(t-\tau)\Delta}e_a(Gk)\,d\tau.}
$$
Note that the terms in the right-hand side may be handled
by means of Lemmas \ref{l:extension} and \ref{l:A}, exactly as we did for $t^\alpha\nabla^2_hu^d.$
Therefore, we have
$$
\|t^\alpha\p_d\Pi\|_{L^r_T(L^p_+)}\leq C\|t^\alpha(\nabla g,\p_tQ)\|_{L^r_T(L^p_+)}.
$$
Owing to \eqref{eq:pi3}, it is clear that $t^\alpha\nabla_h\Pi$ satisfies the same inequality.
\medbreak
Let us now  consider the case $f\equiv0,$ $g\equiv0$ and $u_0\in W^{1,p}_0(\R^d_+)\cap  W^{2,p}(\R^d_+)$
(with no loss of generality).  As usual, because
one may go from $u^d$ to $u^h$ through
$$u^h=re^{t\Delta}e_a(V_hu_0)-Su^d,$$
we concentrate on $t^\beta\nabla^2 u^d.$
We start with the formula
$$
t^\alpha\nabla^2_hu^d(t)=Ut^\alpha e^{t\Delta}e_a(\nabla_h^2 V_du_0)
=Ut^\alpha e^{t\Delta}\nabla_h^2 V_d e_a(u_0),
$$
which, in view of Lemmas \ref{l:extension} and \ref{l:D} ensures that
$$
\|t^\alpha\nabla^2_hu^d\|_{L_T^r(L^p_+)}\leq C \|V_d e_a(u_0)\|_{\dot B^s_{p,r}(\R^d)}
\leq C\|u_0\|_{\dot B^s_{p,r}(\R^d_+)}\quad\hbox{with }\  \alpha=1-\frac s2-\frac1r\cdotp
$$
Similarly, we have
$$
t^\alpha\nabla_h\p_du^d=(I-U)t^\alpha e^{t\Delta}\nabla_h|D_h| e_a(V_du_0),
$$
hence
\begin{equation}\label{eq:2}
\|t^\alpha\nabla_h\p_du^d\|_{L_T^r(L^p_+)}\leq C\|u_0\|_{\dot B^s_{p,r}(\R^d_+)}.
\end{equation}
Finally, $t^\alpha\p_d^2u^d$ satisfies
$$\begin{array}{ccc}
t^\alpha\p_d^2u^d&=&rt^\alpha e^{t\Delta} e_s(\p_d|D_h|V_du_0)+(I-U)t^\alpha e^{t\Delta}e_a(\Delta_hV_du_0)\\[1ex]
&=&rt^\alpha e^{t\Delta} \p_d|D_h|e_a(V_du_0)+(I-U)t^\alpha
e^{t\Delta}e_a(\Delta_hV_du_0)\end{array}
$$
because $e_s(\p_d V_du_0)=\p_de_a(V_du_0)$ owing to the fact that $V_du_0$ vanishes on $\p\R^d_+.$
Hence $t^\alpha\p_d^2u^d$ satisfies \eqref{eq:2}, too,  and we conclude that
\begin{equation}
\|t^\alpha \nabla^2 u^d\|_{L_T^r(L^p_+)}\leq C\|u_0\|_{\dot B^s_{p,r}(\R^d_+)}.
\end{equation}
Bounding $\nabla\Pi$ is strictly analogous.
\medbreak
In order to prove the estimate for $t^\alpha \nabla^2u$ in the case $g\equiv0$ and $u_0\equiv0,$
we use  that
$$\displaylines{
t^\alpha\nabla^2_hu^d=U\:t^\alpha\int_0^te^{(t-\tau)\Delta}\nabla_h^2e_a(\wt
Nf)\,d\tau, \quad
t^\alpha\nabla_h\p_du^d=(I-U)\,t^\alpha\int_0^te^{(t-\tau)\Delta}\nabla_h|D_h|e_a(\wt
Nf)\,d\tau\cr \hbox{and}\quad t^\alpha\p_d^2u^d=r\int_0^t
e^{(t-\tau)\Delta}\p_d|D_h|e_a(\wt Nf)\,d\tau +(I-U)\,t^\alpha
\int_0^te^{(t-\tau)\Delta}\Delta_he_a(\wt Nf)\,d\tau.}
$$
Then combining Lemmas \ref{l:extension} and \ref{l:A} readily gives
$$
\|t^\alpha\nabla^2u^d\|_{L_T^r(L^p_+)}\leq C\|t^\alpha f\|_{L_T^r(L^p_+)}.
$$
Bounding $t^\alpha\nabla^2 u_h$ and $t^\alpha\nabla\Pi$ works the same.
\medbreak
Let us finally go to the proof of estimates for $t^\beta\nabla u$ and $t^\gamma u.$
By virtue of \eqref{eq:heatformula} and of the definition of $\cB$ (see the appendix), we have
$$
t^\beta\nabla u(t)=r\Bigl(t^\beta\nabla e^{t\Delta}e_a(u_0)+t^\beta\cB e_a(f-\nabla\Pi)\Bigr)\cdotp
$$
Therefore, applying Lemmas \ref{l:B} and \ref{l:D} yields
$$
\|t^\beta\nabla u\|_{L^{r_2}_T(L^{p_2}_+)} \leq  C\Bigl(\|e_a(u_0)\|_{\dot B^{s_2}_{p_2,r_2}(\R^d)}
+\|t^{\alpha}e_a(f-\nabla\Pi)\|_{L_T^r(L^p(\R^d))}\Bigr),
$$
whenever $p_2\geq p,$  $r_2\geq r,$  $s_2=\frac d{p_2}-\frac dp+s,$
$$
\beta=\alpha+\frac d2\biggl(\frac1p-\frac1{p_2}\biggr)-\frac12+\f1r-\f1{r_2}=\frac12-\frac {s_2}2-\frac1{r_2}
$$
and
$$
s<1+\frac d{p}-\frac d{p_2}<2+\frac2{r_2}-\frac2r\cdotp
$$
Combining with the fact that $e_a$ is continuous on functions of $\dot B^{s_2}_{p_2,r_2}(\R^d_+)$
with null trace at the boundary,
and with \eqref{eq:Pi},  we get
$$
\|t^\beta\nabla u\|_{L^{r_2}_T(L^{p_2}_+)} \leq  C\Bigl(\|u_0\|_{\dot B^{s_2}_{p_2,r_2}(\R^d_+)}
+\|t^{\alpha}(f,\nabla g,\p_tQ)\|_{L_T^r(L^p_+)}\Bigr)\cdotp
$$
Finally, in order to bound $t^\gamma u,$ we use the formula
$$
t^\gamma u(t)=r\Bigl(t^\gamma e^{t\Delta}e_a(u_0)+t^\gamma\cC e_a(f-\nabla\Pi)\Bigr)\cdotp
$$
 Applying Lemmas \ref{l:C} and \ref{l:D} yields
$$
\|t^\gamma u\|_{L^{r_3}_T(L^{p_3}_+)} \leq  C\Bigl(\|e_a(u_0)\|_{\dot B^{s_3}_{p_3,r_3}(\R^d)}
+\|t^{\alpha}e_a(f-\nabla\Pi)\|_{L_T^p(L^r(\R^d))},
$$
with $p_3,\geq p,$ $r_3\geq r,$  $s_3=\frac d{p_3}-\frac dp+s,$
$$
\gamma=\alpha+\frac
d2\biggl(\frac1p-\frac1{p_3}\biggr)-1+\f1r-\f1{r_3}=-\frac
{s_3}2-\frac1{r_3}
$$
and
$$
s<\frac d{p}-\frac d{p_3}< 2+\frac2{r_3}-\frac2r\cdotp
$$
Combining with the fact that $e_a$ is continuous on functions of $\dot B^{s_3}_{p_3,r_3}(\R^d_+)$ with null trace,
and with \eqref{eq:Pi}, we get
$$
\|t^\gamma  u\|_{L^{r_3}_T(L^{p_3}_+)} \leq  C\Bigl(\|u_0\|_{\dot B^{s_3}_{p_3,r_3}(\R^d_+)}
+\|t^{\gamma}(f,\nabla g,\p_tQ)\|_{L_T^r(L^p_+)}\Bigr).
$$
This completes the proof of the proposition.
\end{proof}

%%%%%%%%%%%%%%%%%%%%%%%%%%%%%%%%%%%%%%%%%%%%%%%
\setcounter{equation}{0}
\section{Existence of smooth solutions}

As a first step for proving Theorem \ref{th:main}, we here establish the global existence  of strong solutions for \eqref{INS}
\emph{in the case of a globally Lipschitz bounded density}. As for
the velocity, we assume that it has  slightly sub-critical regularity.
Here is our statement (recall that the space $X^{p,r}$ has been defined in  \eqref{normxpr}):

%Before giving the statement, we recall some notation:for $1<p,r<\infty$ and $T>0,$  we set$$
%X^{p,r}_T\eqdefa\Bigl\{(u,\nabla\Pi)\in \cC([0,T]; \dot{B}^{2-\f{2}{r}}_{p,r}(\R^d_+))\times
%L^r(\R_+;L^p_+) \,:\, \p_tu,\nabla^2u\in L^r((0,T);L^p_+)\Bigr\}
%$$ with norm $\|(u,\nabla\Pi)\|_{X^{p,r}_T}$ being defined by

\begin{thm}\label{thm:auxiliary}
{\sl  Let $a_0\in W^{1,\infty}(\R^d_+)$ and $
u_0\in \dot{\cB}^{-1+\f{d}{p}}_{p,r}(\R^d_+)\cap
\dot{\cB}^{-1+\f{d}{p}}_{\wt p,r}(\R^d_+)$ with  $p=\f{dr}{3r-2},$
$r\in (1,\infty)$ and $d<\wt p \leq\f{dr}{r-1}\cdotp$
  There exist two positive constants $c_0=c_0(r,d)$ and $c_1=c_1(r,d)$  so that if
  \eqref{small1} holds,
 then
\eqref{INS} has a unique global solution  $(a,u,\nabla\Pi)$ with
$a\in L^\infty_{loc}(\R_+;W^{1,\infty}(\R^d_+)),$
$$
(u,\nabla\Pi)\in X^{p,r}\cap X^{\wt p,r}\ \hbox{ and }\ \nabla u\in
L^{2r}(\R_+;L^{\f{dr}{2r-1}}_+)\cap L^{q}(\R_+;L^\infty_+)\ \hbox{
with }\ q=\frac{2\wt p}{d+(\f2r\,-1)\wt p}\cdotp
$$
In addition, there exist $C_i=C_i(r,d),$ $ i=1,2,3,4,$  so that
  \begin{eqnarray}\label{eq:th2a}
  &\|u^h\|_{\mathfrak{X}^{p,r}}+\mu^{1-\f1{2r}}\|\nabla u^h\|_{L^{2r}(\R_+;L^{\f{dr}{2r-1}}_+)}\leq C_1\mu^{1-\f1r}\eta_0,\\
\label{eq:th2c} &\|(u,\nabla\Pi)\|_{X^{p,r}}+\mu^{1-\f1{2r}}\|\nabla
u\|_{L^{2r}(\R_+;L^{\f{dr}{2r-1}}_+)}\leq C_2\mu^{1-\f1r}
  \|u_0\|_{\dot{B}^{-1+\f{d}{p}}_{p,r}(\R^d_+)},\\
   \label{eq:th2b} &\quad \|(u,\nabla\Pi)\|_{X^{\wt p,r}}\!+\!\mu^{1-\f1{2r}}\bigl(\|\na u\|_{L^{2r}(\R^+;L^{\alpha}_+)}
 \!+\!\mu^{\f1q-\f1{2r}}\|\na u\|_{L^{q}(\R^+;L^{\infty}_+)}\bigr)\\
\nonumber&\qquad\qquad\qquad\qquad\qquad\qquad \leq
C_3\mu^{1-\f1r}\|u_0\|_{\dot{B}^{-1+\f{d}{p}}_{\wt
 p,r}(\R^d_+)}\exp\Bigl(C_4\mu^{-2r}{\|u_0^d\|_{\dot{B}^{-1+\f{d}{p}}_{p,r}(\R^d_+)}^{2r}}\Bigr),
    \end{eqnarray}  where $\eta_0$ is given by \eqref{small1} and $\alpha$ satisfies
$\f1{\wt p}=\f1\alpha+\f{r-1}{dr}\cdotp$ }\end{thm}
\begin{proof}
The general strategy to prove Theorem \ref{thm:auxiliary} is the
same as in \cite{DM}: we set $(a^0,u^0,\nabla\Pi^0)=(0,0,0)$ and
solve inductively  the following \emph{linear} system:
\begin{equation}\label{eq:scheme1}
 \quad\left\{\begin{array}{l}
\displaystyle \pa_t a^{n+1} + u^{n} \cdot \grad a^{n+1}=0,\qquad \\
\displaystyle \pa_t u^{h,n+1}+u^{d,n}\partial_du^{h,n+1}-\mu\D u^{h,n+1}
+\na_h\Pi^{n+1}=F^{h,n}\\
\displaystyle \pa_t u^{d,n+1}+u^{h,n+1}\cdot\nabla_hu^{d,n}-u^{d,n}\dive_hu^{h,n+1}-\mu\D u^{d,n+1}
+\partial_d\Pi^{n+1}=F^{d,n}\\
\displaystyle \dive\, u^{n+1} = 0, \\
\displaystyle u^{n+1}|_{\partial\R^d_+}=0,\\
 \displaystyle (a^{n+1}, u^{n+1})|_{t=0}=(a_{0},u_{0}),
\end{array}\right.
\end{equation}
where $u^n=(u^{h,n},u^{d,n})$ and $F^{n}= (F^{h,n},F^{d,n})$ with
\beq \label{add.4}
\begin{split} & F^{h,n} \eqdefa
a^{n+1}(\mu\D u^{h,n}-\na_h\Pi^n)-u^{h,n}\cdot \grad_h u^{h,n},\\
&F^{d,n} \eqdefa a^{n+1}(\mu\D u^{d,n}-\p_d\Pi^n). \end{split} \eeq \medbreak
 The global existence and uniqueness of a solution to \eqref{eq:scheme1} in the spaces given in Theorem \ref{thm:auxiliary}
 may be proved inductively. This is standard as regards the existence of $a^{n+1}$
 as   $\nabla u^n$ is
in $L^1_{loc}(\R^+;L^\infty_+).$ As for $(u^{n+1},\nabla\Pi^{n+1}),$ we apply Proposition \ref{p:stokesconv}
with  $v^d=u^{d,n}$ and $f=F^n.$

\subsubsection*{Step 1. Uniform estimates}

It is obvious that
\begin{equation}\label{eq:a0}
\|a^n(t)\|_{L^\infty_+}=\|a_0\|_{L^\infty_+}%,\quad\|a^n(t)\|_{L^1_+}=\|a_0\|_{L^1_+}
\quad\hbox{for all }\ n\in\N\
\hbox{ and }\ t\in\R^+,
\end{equation}
and that
\begin{equation}\label{add.5}
\|\nabla a^{n+1}(t)\|_{L^\infty_+}\leq \exp\Bigl({\int_0^t\|\nabla
u^n\|_{L^\infty_+}\,d\tau}\Bigr)\|\nabla a_0\|_{L^\infty_+}.
\end{equation}
In order to bound $(u^n,\nabla\Pi^n),$  we introduce
$(\wt u^n_\la,\nabla\wt\Pi^n_\la,\wt F^n_\la)(t)\eqdefa h^n_\la(t) (u^n,\nabla\Pi^n, F^n)(t)$
with
$$ h^n_\la(t)\eqdefa\exp\Bigl(-\la\mu^{1-2r}\int_0^t  \|\na
u^{d,n}(\tau)\|_{L^{\f{dr}{2r-1}}_+}^{2r}\,d\tau\Bigr)\cdotp
$$
We shall prove inductively  that there exists  $\lambda_0>0$  such
that  for all $\la\geq\la_0,$ $t>0$  and $n\in\N,$
\begin{eqnarray}\label{eq:est1}
&&\quad\|\wt
u^{h,n}_\la\|_{\mathfrak{X}^{p,r}_t}+\mu^{1-\f1{2r}}\|\nabla\wt
u^{h,n}_\la\|_{L_t^{2r}(L^{\f{dr}{2r-1}})} \leq
C_1\mu^{1-\f1r}(\|u_0^h\|_{\dot B^{-1+\f
dp}_{p,r}}+\|a_0\|_{L^\infty_+}\|u_0\|_{\dot B^{-1+\f
dp}_{p,r}}),\\\label{eq:est2} &&
\quad\|(u^n,\nabla\Pi^n)\|_{X_t^{p,r}}+\mu^{1-\f1{2r}}\|\na
u^n\|_{L^{2r}_t(L^{\f{dr}{2r-1}}_+)}
 \leq C_2\mu^{1-\f1r}\|u_0\|_{\dot B^{-1+\f dp}_{p,r}},\\\label{eq:est3}
 && \quad\|(u^n,\nabla\Pi^n)\|_{X^{\wt p,r}_t}\!+\!\mu^{1-\f1{2r}}\bigl(\|\na u^n\|_{L^{2r}_t(L^{\alpha}_+)}
 +\mu^{\f1q-\f1{2r}}\|\na u^n\|_{L^{q}_t(L^{\infty}_+)}\bigr)
 \\&&\hspace{5.5cm}\leq C_3\mu^{1-\f1r}\|u_0\|_{\dot B^{-1+\f dp}_{\wt p,r}}
 \exp\Bigl(C_4\mu^{-2r}\|u_0\|_{\dot{B}^{-1+\f{d}{p}}_{p,r}(\R^d_+)}^{2r}\Bigr)\cdotp\nonumber
 \end{eqnarray}
We shall use repeatedly the fact that \eqref{eq:est2} implies that
\begin{equation}\label{add:0}
\sup_{t\in\R_+} \frac1{h^n_\la(t)}\leq \exp\Bigl(\lambda C_2^{2r}\mu^{-2r}{\|u_0\|_{\dot{B}^{-1+\f{d}{p}}_{p,r}(\R^d_+)}^{2r}}\Bigr)\cdotp
\end{equation}
In order  to prove the critical regularity estimates (namely \eqref{eq:est1} and \eqref{eq:est2}), we use
the fact that by virtue of  \eqref{eq:a0}, H\"older inequality (recall that $p=\f{dr}{3r-2}$) and Sobolev embedding
\begin{equation}\label{eq:sobemb}
W^{1,\frac{dr}{2r-1}}_0(\R^d_+)\hookrightarrow
L^{\f{dr}{r-1}}(\R^d_+),
\end{equation}
we have
$$
\|\wt F^n_\la\|_{L_t^{r}(L^p_+)}\leq \|a_0\|_{L^\infty_+}\|\mu\Delta\wt u^n_\la-\nabla\wt\Pi^n_\la\|_{L_t^r(L^p_+)}
+C\|\nabla u^{h,n}\|_{L_t^{2r}(L^{\f{dr}{2r-1}})} \|\nabla \wt u^{h,n}_\la\|_{L_t^{2r}(L^{\f{dr}{2r-1}})}.
$$
Combining with Inequality \eqref{eq:aux1}, we thus get for $\lambda\geq \la_0,$
\begin{multline}\label{add.1}
\|\wt
u^{h,n+1}_\la\|_{\mathfrak{X}^{p,r}_t}+\mu^{1-\f1{2r}}\|\nabla\wt
u^{h,n+1}_\la\|_{L_t^{2r}(L^{\f{dr}{2r-1}})} \leq
C\bigl(\mu^{1-\f1r}\|u_0^h\|_{\dot B^{-1+\f
dp}_{p,r}}\\+\|a_0\|_{L^\infty_+}\|\mu\Delta\wt
u^n_\la-\nabla\wt\Pi^n_\la\|_{L_t^r(L^p_+)} +\|\nabla
u^{h,n}\|_{L_t^{2r}(L^{\f{dr}{2r-1}})} \|\nabla \wt
u^{h,n}_\la\|_{L_t^{2r}(L^{\f{dr}{2r-1}})}\bigr).
\end{multline}
Therefore, taking advantage of \eqref{eq:est1}, \eqref{eq:est2} and \eqref{add:0}, and assuming that $C_4\geq\lambda C_2^{2r},$
$$\displaylines{\quad\|\wt u_\la^{h,n+1}\|_{\mathfrak{X}^{p,r}_t}+\mu^{1-\f1{2r}}\|\nabla\wt u_\la^{h,n+1}\|_{L_t^{2r}(L^{\f{dr}{2r-1}})}
\leq C\mu^{1-\f1r}\bigl(\|u_0^h\|_{\dot B^{-1+\f
dp}_{p,r}}+\|a_0\|_{L^\infty_+}\|u_0\|_{\dot B^{-1+\f
dp}_{p,r}}\bigr) \hfill\cr\hfill+CC_1 (\mu^{-1}\wt\eta_0)\,
\mu^{1-\f1{2r}}\|\nabla\wt
u^{h,n}_\la\|_{L_t^{2r}(L^{\f{dr}{2r-1}})}}
$$
with
$$\wt\eta_0\eqdefa (\|u_0^h\|_{\dot B^{-1+\f dp}_{p,r}}+\|a_0\|_{L^\infty_+}\|u_0\|_{\dot B^{-1+\f dp}_{p,r}}) \exp\Bigl(C_4\mu^{-2r}{\|u_0\|_{\dot{B}^{-1+\f{d}{p}}_{p,r}(\R^d_+)}^{2r}}\Bigr)\cdotp$$
Hence we get \eqref{eq:est1} at rank $n+1$ if, say,  $2CC_1\wt\eta_0\leq\mu$ and $C_1$ has been taken larger than $2C.$
\medbreak
Next, denoting
 $$G^n\eqdefa(F^{h,n}-u^{d,n}\pa_d u^{h,n+1},F^{d,n}-u^{h,n+1}\cdot\nabla_h u^{d,n}+u^{d,n}\dive_h u^{h,n+1}),$$
 we get, arguing as for bounding $\wt F^n_\la$:
$$\displaylines{\quad
\|G^n\|_{L_t^{r}(L^p_+)}\leq C\bigl(\|a_0\|_{L^\infty_+}\|\mu\Delta u^n-\nabla\Pi^n\|_{L_t^r(L^p_+)}
\hfill\cr\hfill+\|\nabla u^{n}\|_{L_t^{2r}(L^{\f{dr}{2r-1}})}(\|\nabla u^{h,n}\|_{L_t^{2r}(L^{\f{dr}{2r-1}})}+\|\nabla u^{h,n+1}\|_{L_t^{2r}(L^{\f{dr}{2r-1}})})\bigr).\quad}
$$
Thus, applying Proposition \ref{p:stokes} to System \eqref{eq:scheme1} guarantees that
$$\displaylines{
\|(u^{n+1},\nabla\Pi^{n+1})\|_{X^{p,r}_t}+\mu^{1-\f1{2r}}\|\nabla u^{n+1}\|_{L_t^{2r}(L^{\f{dr}{2r-1}})}
\leq C\bigl(\mu^{1-\f1r}\|u_0\|_{\dot B^{-1+\f dp}_{p,r}}\hfill\cr\hfill+\|a_0\|_{L^\infty_+}\|\mu\Delta u^n-\nabla\Pi^n\|_{L_t^r(L^p_+)}
+\|\nabla u^{n}\|_{L_t^{2r}(L^{\f{dr}{2r-1}})}(\|\nabla u^{h,n}\|_{L_t^{2r}(L^{\f{dr}{2r-1}})}+\|\nabla u^{h,n+1}\|_{L_t^{2r}(L^{\f{dr}{2r-1}})})\bigr).}
$$
Then, inserting \eqref{eq:est1} at rank $n$ and $n+1,$ and \eqref{eq:est2}  leads to
$$\displaylines{
\|(u^{n+1},\nabla\Pi^{n+1})\|_{X^{p,r}_t}+\mu^{1-\f1{2r}}\|\nabla u^{n+1}\|_{L_t^{2r}(L^{\f{dr}{2r-1}})}\hfill\cr\hfill
\leq C\mu^{1-\f1r}\|u_0\|_{\dot B^{-1+\f dp}_{p,r}}\bigl(1+C_2\|a_0\|_{L^\infty_+}+2C_1C_2\mu^{-1}\wt\eta_0\bigr).}
$$
Hence, if  $C_1=2C$ then  \eqref{eq:est2} is fulfilled at rank $n+1$ whenever $\mu^{-1}\wt\eta_0$  is  small enough.
\medbreak
Let us now turn to the proof of the regularity estimate \eqref{eq:est3}. The starting point is that,
combining  \eqref{eq:sobemb} and H\"older inequality (here the assumption $\wt p\leq\f{dr}{r-1}$
comes into play) yields,
$$
\|\wt F^n_\la\|_{L_t^{r}(L^{\wt p}_+)}\leq \|a_0\|_{L^\infty_+}\|\mu\Delta\wt u^n_\la-\nabla\wt\Pi^n_\la\|_{L_t^r(L^{\wt p}_+)}
+C\|\nabla u^{h,n}\|_{L_t^{2r}(L^{\alpha}_+)} \|\nabla \wt u^{h,n}_\la\|_{L_t^{2r}(L^{\f{dr}{2r-1}})}.
$$
which implies, according to \eqref{eq:aux2},
$$\displaylines{\quad\|\wt u^{h,n+1}_\la\|_{\mathfrak{X}_t^{\wt
p,r}}\!+\!\mu^{1-\f1{2r}}\|\na \wt u^{h,n+1}_\la\|_{L^{2r}_t(L^{\alpha}_+)}\leq
C\bigl(\mu^{1-\f1r}\|u_0^h\|_{\dot
B^{-1+\f{d}p}_{\wt{p},r}(\R^d_+)}\hfill\cr\hfill+\|a_0\|_{L^\infty}\|(\wt u^n_\la,\nabla\wt\Pi^n_\la)\|_{X_t^{\wt p,r}}
+\bigl(\|\nabla\wt u^{h,n}_\la\|_{L_t^{2r}(L^{\f{dr}{2r-1}}_+)} + \|\nabla\wt u^{h,n+1}_\la\|_{L_t^{2r}(L^{\f{dr}{2r-1}}_+)}\bigr)
\|\nabla u^n\|_{L_t^{2r}(L^\alpha_+)} \bigr),}
 $$
whence, taking advantage of \eqref{eq:est1} at ranks $n$ and $n+1$, and of  \eqref{eq:est3},
 $$\displaylines{\quad
\|\wt u^{h,n+1}_\la\|_{\mathfrak{X}_t^{\wt
p,r}}+\mu^{1-\f1{2r}}\|\na \wt u^{h,n+1}_\la\|_{L^{2r}_t(L^{\alpha}_+)}\hfill\cr\hfill\leq C
\mu^{1-\f1r}\|u_0\|_{\dot B^{-1+\f{d}p}_{\wt{p},r}(\R^d_+)}\biggl(1+C_3 \exp\Bigl({C_4 \mu^{-2r}\|u_0\|_{\dot
B^{-1+\f{d}p}_{p,r}(\R^d_+)}^{2r}}\Bigr)\|a_0\|_{L^\infty_+}+2C_1C_3\mu^{-1}\wt\eta_0\biggr)\quad}  $$
 which implies,  assuming that $\mu^{-1}\wt\eta_0$ is small enough and using \eqref{add:0},
 \begin{multline}\label{add.3}
 \|u^{h,n+1}\|_{\mathfrak{X}_t^{\wt p,r}}\!+\!\mu^{1-\f1{2r}}\|\na
u^{h,n+1}\|_{L^{2r}_t(L^{\alpha}_+)} \\\leq C'_3\mu^{1-\f1r}\|u_0\|_{\dot
B^{-1+\f{d}p}_{\wt{p},r}(\R^d_+)}\exp\Bigl({\lambda C_2^{2r}\mu^{-2r}\|u_0\|_{\dot
B^{-1+\f{d}p}_{p,r}(\R^d_+)}^{2r}}\Bigr)\cdotp
\end{multline}
 Furthermore, applying  Proposition \ref{p:stokes} to \eqref{eq:scheme1} we see that
 $$\displaylines{\quad
 \|(u^{n+1},\nabla\Pi^{n+1})\|_{X^{\wt p,r}_t}\!+\!\mu^{1-\f1{2r}}\bigl(\|\na u^{n+1}\|_{L^{2r}_t(L^{\alpha}_+)}
 \!+\!\mu^{\f1q-\f1{2r}}\|\na u^{n+1}\|_{L^{q}_t(L^{\infty}_+)}\bigr)\hfill\cr\hfill\leq
C\Bigl(\mu^{1-\f1r}\|u_0\|_{\dot B^{-1+\f dp}_{\wt
p,r}(\R^d_+)}+\|G^n\|_{L^{r}_t(L^{\wt p}_+)}\Bigr).}$$
 Now, combining  \eqref{eq:a0}, \eqref{eq:sobemb} and H\"older inequality, we discover that
 $$\displaylines{\quad
 \|G^n\|_{L^{r}_t(L^{\wt p}_+)}\leq \|a_0\|_{L^\infty_+}\bigl(\mu\|\D
u^n\|_{L^{r}_t(L^{\wt p}_+)}+\|\na\Pi^n\|_{L^{r}_t(L^{\wt
p}_+)}\bigr)+\|\nabla u^{h,n\!+\!1}\|_{L^{2r}_t(L^{\alpha}_+)}\|\na
u^{d,n}\|_{L^{2r}_t(L^{\f{dr}{2r-1}}_+)}\hfill\cr\hfill
+\bigl(\|\nabla u^{h,n}\|_{L^{2r}_t(L^{\f{dr}{2r-1}}_+)}+\|\nabla u^{h,{n\!+\!1}}\|_{L^{2r}_t(L^{\f{dr}{2r-1}}_+)}\bigr) \|\na
u^{n}\|_{L^{2r}_t(L^{\alpha}_+)}.\quad }
$$ Therefore, using \eqref{add.3} and
  the induction hypotheses   \eqref{eq:est1}, \eqref{eq:est2} and
\eqref{eq:est3},
 $$ \displaylines{
 \|(u^{n+1},\nabla\Pi^{n+1})\|_{X_t^{\wt
p,r}}+\mu^{1-\f1{2r}}\bigl(\|\na u^{n+1}\|_{L^{2r}_t(L^{\alpha}_+)}
 +\mu^{\f1q-\f1{2r}}\|\na u^{n+1}\|_{L^{q}_t(L^{\infty}_+)}\bigr)\hfill\cr\hfill\leq
C\mu^{1-\f1r}\|u_0\|_{\dot B^{-1+\f dp}_{\wt p,r}(\R^d_+)}
\Bigl(\bigl(1+C_3\|a_0\|_{L^\infty_+}+C_1C_3\wt\eta_0\mu^{-1}\bigr)
\exp\Bigl({C_4\mu^{-2r}\|u_0\|_{\dot B^{-1+\f{d}p}_{p,r}(\R^d_+)}^{2r}}\Bigr)
\hfill\cr\hfill +C_2C'_3\mu^{-1}\|u_0\|_{\dot B^{-1+\f dp}_{p,r}(\R^d_+)}
\exp\Bigl({\lambda C_2^{2r}\mu^{-2r}\|u_0\|_{\dot
B^{-1+\f{d}p}_{p,r}(\R^d_+)}^{2r}}\Bigr)\Bigr)\cdotp}
$$
  Hence, if $\wt\eta_0/\mu^{-1}$ is small enough and, say, $C_4\geq2\lambda C_2^{2r}$ then
we get \eqref{eq:est3}  at rank $n+1.$
\smallbreak
Now, by using the fact that  $ \|u_0\|_{\dot B^{-1+\f dp}_{p,r}(\R^d_+)}\leq
\eta_0+\|u_0^d\|_{\dot B^{-1+\f dp}_{p,r}(\R^d_+)}$ where $\eta_0$ has been defined in the statement of Theorem \ref{th:main},
it is obvious that under Condition \eqref{small1} for suitable constants $c_0$ and $c_1$ the smallness condition
for $\wt\eta_0/\mu$ is fulfilled.
This completes the proof of \eqref{eq:est1}, \eqref{eq:est2} and \eqref{eq:est3} for all $n\in\N.$
Furthermore, note that one may replace
$u_0$ by $u_0^d$ in the exponential term of \eqref{eq:est3} and \eqref{add.3}.

\subsubsection*{Step 2. Convergence of the sequence}

Let $\check p$ be some real number in $(\f d2,\f{dr}{2r-1})$ such that in addition\footnote{It would
be natural to take $\check p=p$ but we do not know how to handle the case $r\geq2$ with this value of $\check p.$}
 $p\leq\check p\leq\wt p.$  Arguing exactly as for proving \eqref{eq:est3}, we get
 \begin{multline}\label{eq:est4}
  \|(u,\nabla\Pi)\|_{X^{\check p,r}}\!+\!\mu^{1-\f1{2r}}\|\na u\|_{L^{2r}(\R^+;L^{m}_+)}\\
 \leq C_3\mu^{1-\f1r}\|u_0\|_{\dot{B}^{-1+\f{d}{\check p}}_{\wt
 p,r}(\R^d_+)}\exp\Bigl(C_4\mu^{-2r}{\|u_0^d\|_{\dot{B}^{-1+\f{d}{p}}_{p,r}(\R^d_+)}^{2r}}\Bigr)\quad\hbox{with }\
 \frac dm=\f d{\check p}-1+\f1r\cdot
 \end{multline}
We claim that $(a^n)_{n\in\N}$ and
$(u^n,\nabla\Pi^n)_{n\in\N}$ are  Cauchy sequences in
$\cC_b([0,T]\times \R^d_+)$ and $X^{\check p,r}_T,$ respectively, for all $T>0.$
 For proving the convergence of
  $(a^n)_{n\in\N},$ we use the fact that
$$\p_t\da^n+u^n\cdot\nabla\da^n=-\du^{n-1}\cdot\nabla a^n\quad\hbox{with}\quad\da^n\eqdefa a^{n+1}-a^n\ \hbox{ and }\ \du^n\eqdefa u^{n+1}-u^n. $$
Hence, using standard estimates for the transport equation, we get  for all positive $T,$
$$
\|\da^{n}(T)\|_{L^\infty_+} \leq \int_0^T\|\du^{n-1}\|_{L^\infty_+}\|\nabla a^n\|_{L^\infty_+}\,dt.
$$
Now, arguing exactly as in the proof of \eqref{eq:GN}, we get
 the following Gagliardo-Nirenberg inequality\footnote{It suffices to
apply the corresponding inequality in $\R^d$ to function $e_a(z).$}
for functions $z$ vanishing at $\p\R^d_+$:
\begin{equation}\label{eq:GN2}
\|z\|_{L^\infty_+}\leq C\|z\|_{\dot B^{2-\f2r}_{\check p,r}(\R^d_+)}^{\theta} \|\nabla^2z\|_{L^{\check p}_+}^{1-\theta}\quad\hbox{with }\
\theta\eqdefa\biggl(\check p-\f d2\biggr)\f r{\check p}\cdotp
\end{equation}
Hence
$$
\|\da^{n}(T)\|_{L^\infty_+} \leq \int_0^T\|\nabla a^n\|_{L^\infty_+}\|\du^{n-1}\|_{\dot B^{2-\f 2r}_{\check p,r}(\R^d_+)}^{\theta}
\|\nabla^2\du^{n-1}\|_{L^{\check p}_+}^{1-\theta}\,dt.
$$Taking advantage of Young's inequality we thus get for all positive $\e,$
\begin{equation}\label{eq:da}
\|\da^{n}\|_{L^\infty(0,T\times\R^d_+)}\leq
\e\|\nabla^2\du^{n-1}\|_{L_T^r(L^{\check p}_+)}
+C_\e\int_0^T\|\nabla a^{n-1}\|_{L^\infty_+}^{\f{2\check p}{(2\check
p-d)r}}\|\du^{n-1}\|_{\dot B^{2-\f2r}_{\check p,r}(\R^d_+)}\,dt.
\end{equation}
Next, we use the fact that, denoting $\dG^n\eqdefa G^{n+1}-G^n,$ we have
$$
\left\{\begin{array}{l}
\p_t\du^n-\mu\Delta\du^n+\nabla\dPi^n=\dG^n\\[1ex]%\quad&\hbox{in }\ \R_+\times\R^d_+
\dive\du^n=0\\[1ex]%\quad&\hbox{in }\ \R_+\times\R^d_+
u^n|_{t=0}=0\quad\hbox{and}\quad u^n|_{\p\R^d_+}=0.
\end{array}\right.
$$
 Applying Proposition \ref{p:stokes}, we see that for some constant $C_0=C_0(p,d),$
\begin{equation}\label{eq:du}
\dU^n(t)\eqdefa
\|(\du^n,\nabla\dPi^n)\|_{X^{\check p,r}_T}+\mu^{1-\f1{2r}}\|\nabla\du^n\|_{L^{2r}_T(L^m_+)}\leq  C_0\|\dG^n\|_{L^r_T(L^{\check p}_+)}.
\end{equation}
Let us  decompose $\dG^n$ into
 $$\dG^n \eqdefa(\mu\Delta u^n-\nabla\Pi^n)\da^{n}
+a^{n}(\mu\Delta\du^{n-1}-\nabla\dPi^{n-1})-\dH^n,$$ with
 $\dH^n=(\dG^{h,n},\dG^{d,n}),$
and $$
\begin{array}{lll}
 \dH^{h,n}&\!\!\!\eqdefa\!\!\!& u^{d,n}\pa_d\du^{h,n}+\du^{d,n-1}\pa_d u^{h,n}+ u^{h,n}\cdot\nabla_h\du^{h,n-1}
 +\du^{h,n-1}\cdot\nabla_h u^{h,n-1},\\[1ex]
\dH^{d,n}&\!\!\!\eqdefa\!\!\!& \du^{h,n}\cdot\nabla_h
u^{d,n}+u^{h,n}\cdot\na_h\du^{d,n-1}-u^{d,n}\dive_h\du^{h,n}-\du^{d,n-1}\dive_h
u^{h,n}.
\end{array}
$$
Using \eqref{eq:a0} and arguing as in the first step of the proof,
we easily get
$$
\displaylines{
 \|\dG^n\|_{L^r_T(L^{\check p}_+)}\leq\Bigl( \|\da^{n}\|_{L^\infty(0,T\times\R^d_+)}\|\mu\Delta u^{n}-\nabla\Pi^n\|_{L^r_T(L^{\check p}_+)}
\hfill\cr\hfill
+\|a_0\|_{L^\infty_+}\|\mu\Delta\du^{n-1}-\nabla\dPi^{n-1}\|_{L^r_T(L^{\check
p}_+)}+\|\dH^n\|_{L^r_T(L^{\check p}_+)}\Bigr)}
$$
and
$$
\displaylines{\|\dH^n\|_{L^{\check p}_+}\leq
\bigl(\|u^n\|_{L_+^{\f{dr}{r-1}}}\|\nabla\du^{h,n}\|_{L_+^{m}}
 +\bigl(\|\nabla u^{h,n-1}\|_{L_+^{\f{dr}{2r-1}}}+\|\nabla u^{h,n}\|_{L_+^{\f{dr}{2r-1}}}\bigr)\|\du^{n-1}\|_{L^{m^*}_+}
 \hfill\cr\hfill+\|u^{h,n}\|_{L_+^{\f{dr}{r-1}}}\|\nabla_h\du^{n-1}\|_{L_+^{m}}
 +\|\du^{h,n}\|_{L^{m^*}_+}\|\nabla_h u^{d,n}\|_{L_+^{\f{dr}{2r-1}}}\Bigr),}
$$
with\footnote{Here $\check p<\f{dr}{2r-1}$ comes into play.} $m^*$ such that $\f d{m^*}=\f d{\check p}-2+\f1r\cdotp$
At this point, one may combine   the following Sobolev embeddings  \eqref{eq:sobemb} and
$$
W^{1,m}_0(\R^d_+)\hookrightarrow L^{m^*}_+(\R^d_+)
$$
and eventually find out  that
$$
\displaylines{\quad\|\dH^n\|_{L_T^r(L^{\check p}_+)}^r\leq C
\biggl(\bigl(\|\nabla u^{h,n-1}\|_{L_T^{2r}(L_+^{\f{dr}{2r-1}})}+\|\nabla u^{h,n}\|_{L_T^{2r}(L_+^{\f{dr}{2r-1}})}\bigr)^r\|\nabla\du^{n-1}\|_{L_T^{2r}(L_+^{m})}^r
\hfill\cr\hfill +\int_0^T\|\nabla u^{n}\|_{L_+^{\f{dr}{2r-1}}}^r\|\nabla\du^{n}\|_{L_+^{m}}^r\,dt\biggr)\cdotp\quad}
$$
Then  combining with the following interpolation inequality which is a particular case of
\eqref{eq:GN}
$$
\|\nabla\du^{n}\|_{L_+^{m}}\leq C\|\du^n\|_{\dot B^{-1+\f dp}_{\check p,r}}^{\f12}\|\nabla^2\du^n\|_{L^{\check p}_+}^{\f12},
$$
 we get for all positive $\varepsilon$:
$$
\displaylines{
 \|\dG^n\|_{L^r_T(L^{\check p}_+)}^r\leq C\biggl(\|\da^{n}\|_{L^\infty(0,T\times\R^d_+)}^r\|(u^{n},\nabla\Pi^{n})\|_{X^{\check p,r}_T}^r+\|a_0\|_{L^\infty_+}^r\|(\du^{n-1},\nabla\dPi^{n-1})\|_{X^{\check p,r}_T}^r
 \hfill\cr\hfill+\bigl(\|\nabla u^{h,n-1}\|_{L_T^{2r}(L_+^{\f{dr}{2r-1}})}^r
 +\|\nabla u^{h,n}\|_{L_T^{2r}(L_+^{\f{dr}{2r-1}})}^r\bigr)\|\nabla\du^{n-1}\|_{L_T^{2r}(L^{m}_+)}^r
\hfill\cr\hfill +\varepsilon\mu^r\|\nabla^2\du^n\|_{L_T^r(L^{\check p}_+)}^r+\frac1{\varepsilon\mu^r}\int_0^T
\|\nabla u^n\|_{L^{\f{dr}{2r-1}}_+}^{2r}\|\du^n\|_{\dot B^{2-\f2r}_{\check p,r}}^r\,dt
 \bigg)\cdotp}
 $$
 Hence, using \eqref{eq:est1}, \eqref{eq:est2}, \eqref{add:0} and the smallness condition \eqref{small1},
 $$\displaylines{
 (C_0 \|\dG^n\|_{L^r_{T}(L^{\check p}_+)})^r\leq
 C\mu^{r-1}\|u_0\|_{\dot B^{-1+\f dp}_{p,r}}^r\exp\Bigl(rC_4\mu^{-2r}\|u_0^d\|_{\dot B^{-1+\f dp}_{p,r}}^{2r}\Bigr)
 \|\da^{n}\|_{L^\infty((0,T)\times\R^d_+)}^r
 \hfill\cr\hfill+C\|a_0\|_{L^\infty_+}^r\|(\du^{n-1},\nabla\dPi^{n-1})\|_{X^{\check p,r}_T}^r
 +C\bigl((\mu^{-1}\eta_0)\mu^{1-\f1{2r}}\|\nabla\du^{n-1}\|_{L_T^{2r}(L^m_+)}\bigr)^r
\hfill\cr\hfill+ C\varepsilon (\mu\|\nabla^2\du^{n}\|_{L^{r}_{T}(L^{\check p}_+)})^r
+\frac C\varepsilon\int_0^t\mu^{1-2r}\|\nabla u^n\|_{L^{\f{dr}{2r-1}}_+}^{2r}\,\mu^{r-1}\|\du^n\|_{\dot B^{-1+\f dp}_{\check p,r}}^r\,dt.}
  $$
 Now, plugging \eqref{eq:da} with suitably small $\e$ in the above inequality, using \eqref{add.5},
  resuming to \eqref{eq:du}  and applying Gronwall inequality  yields
 $$\displaylines{
( \dU^n(T))^r\leq C\exp\biggl(C\mu^{1-2r}\int_0^T\|\nabla u^n\|_{L^{\f{dr}{2r-1}}_+}^{2r}\,dt\biggr)
\biggl( (\mu^{-1}\eta_0)^r ( \dU^{n-1}(T))^r\hfill\cr\hfill+ A_0\biggl(\int_0^T\alpha(t)\,dt\biggr)^{r-1}
\int_0^T\alpha(t)( \dU^{n-1}(t))^r\,dt\biggr)}
 $$
 where $A_0$ depends only on the initial data, and $\alpha$ is a continuous function which may be determined
 from the bounds in \eqref{add.5} and \eqref{eq:est3}.
 \smallbreak
 Let us emphasize that the exponential term may be bounded by means of \eqref{eq:est2}.
 Therefore, summing up over $n\geq1,$ we eventually get for a large enough constant $C$
  $$\displaylines{
\sum_{n\geq0} \dU^n(T))^r\leq  \exp\Bigl(C\mu^{-2r}\|u_0\|_{\dot B^{-1+\f dp}_{p,r}}^{2r}\Bigr)
\biggl( (\dU^0(T))^r+  C(\mu^{-1}\eta_0)^r \sum_{n\geq0} \dU^n(T))^r
\hfill\cr\hfill +  CA_0\biggl(\int_0^T\alpha(t)\,dt\biggr)^{r-1}
\int_0^T\alpha(t)\biggl( \sum_{n\geq0}(\dU^{n}(t))^r\biggr)\,dt.}
$$
 If the smallness condition \eqref{small1} is satisfied then
  the second term of the right-hand side may be absorbed by the left-hand side (we have to
 take a larger constant $c_1$ in the definition of $\eta_0$ if need be).
  Applying again   Gronwall lemma, it is now easy to conclude that
  $(u^n,\nabla\Pi^n)_{n\in\N}$ is a   Cauchy sequence in  $X^{\check p,r}_{T}.$
 Then resuming to \eqref{eq:da} implies that $(a^n)_{n\in\N}$ is a Cauchy sequence in
 $\cC_b([0,T]\times \R^d_+).$
This completes the proof of the convergence.

\subsubsection*{Step 3. End of the proof of the theorem}

 Granted with the convergence result of the previous step, and the uniform bounds of the first step,
it is not difficult to pass to the limit  in \eqref{eq:scheme1}: we
conclude that the triplet $(a,u,\nabla\Pi)$ with
$a=\lim_{n\rightarrow+\infty}a^n,$ $u=\lim_{n\rightarrow+\infty}
u^n,$ and $\nabla\Pi=\lim_{n\rightarrow+\infty}\nabla\Pi^n,$
satisfies \eqref{INS} in the sense of Definition \ref{defi2.1}. In addition, as $(u^n,\nabla\Pi^n)_{n\in\N}$
is bounded in the space $X^{p,r}\cap X^{\wt p,r}$ which possesses the Fatou
property, one may conclude that  $(u,\nabla\Pi)\in X^{p,r}\cap X^{\wt p,r}$
and\footnote{Rigorously speaking we do not get the time continuity for $u,$ but it may be
recovered afterward from Proposition \ref{p:stokes} by observing that $u$ satisfies an evolutionary
Stokes equation with source term in $L^r(\R_+;L^p_+\cap L^{\wt p}_+)$ and initial
data in $\dot\cB^{2-\frac2r}_{p,r}(\R^d_+)\cap \dot\cB^{2-\frac2r}_{\wt p,r}(\R^d_+)$.}  that \eqref{eq:th2a}, \eqref{eq:th2b} are fulfilled.
Similarly, the uniform bounds
for $a^n$ allow to conclude that $a\in
L^\infty_{loc}(\R_+;W^{1,\infty}(\R^d_+)).$ This
completes the proof of Theorem \ref{thm:auxiliary}
 \end{proof}

%%%%%%%%%%%%%%%%%%%%%%%%%%%%%%%%%%%%%%%%%%%

\section{Proving the existence part of the main theorem}\label{s:existence}

\setcounter{equation}{0}

This section is dedicated to the proof of the existence part of
Theorem \ref{th:main}. It is mostly based on {\it a priori}
estimates for smooth solutions --the same as in the previous
section, and on compactness arguments.

\subsubsection*{Step 1. Constructing a sequence of smooth solutions}

This is only a matter of smoothing out the  data $(a_0,u_0)$  so as to  apply Theorem \ref{thm:auxiliary}.
We proceed as follows:\begin{itemize}
\item Let $\chi\in C_c^\infty(\R^d)$ with $\chi(x)=1$ for $|x|\leq 1.$ We extend $a_0$ to $\wt a_0$ on $\R^d$ by symmetry
then use convolution of $\chi(x/n)\wt a_0$ with a sequence of
nonnegative radially symmetric mollifiers, then restrict to the
half-space. We get a sequence $(a_0^n)_{n\in\N}$ in
$W^{1,\infty}(\R^d_+) $ with the same lower and
upper bounds as $a_0,$ and satisfying $a_0^n\to a_0$ a. e. on
$\R^d_+.$
\item   As $u_0\in \dot\cB^{-1+\f dp}_{p,r}(\R^d_+),$ one may  take the antisymmetric extension $\wt u_0=e_a(u_0)$ of $u_0$ over the whole space
 and then set
$$u_0^n=r\biggl(\sum_{|j|\leq n}\dot\Delta_j\wt u_0\biggr)\cdotp$$
   It is obvious that each term $u_0^n$ is smooth, divergence free, and antisymmetric (and thus $\gamma u_0^n\equiv0$).   Furthermore, $(u_0^n)_{n\in\N}$ converges to
$u_0$ in  $\dot B^{-1+\f dp}_{p,r}(\R^d_+)\cap  \dot
B^{-1+\f{d}p}_{\wt p,r}(\R^d_+).$ Of course, one may find $\check p$
in $(d,\f{dr}{r-1})$ so that each $u_0^n$ belongs to $
\dot\cB^{-1+\f{d}p}_{\check p,r}(\R^d_+).$
\end{itemize}

 \subsubsection*{Step 2. Uniform estimates}

 Let us solve system \eqref{INS}
 with regularized initial data $(a_0^n,u_0^n)$ according to Theorem \ref{thm:auxiliary}.  We get a global solution $(a^n,u^n,\nabla\Pi^n)$ in
 $L^\infty_{loc}(\R_+;W^{1,\infty}(\R^d_+))\times (X^{p,r}\cap X^{\check p,r})$ satisfying
 $$
 \|a^n\|_{L^\infty(\R_+\times\R^d_+)}=\|a_0^n\|_{L^\infty(\R^d_+)}\leq \|a_0\|_{L^\infty(\R^d_+)},$$
 and also
 \begin{equation}\label{3.1aa}
\begin{split}
&\|u^{h,n}\|_{\mathfrak{X}^{p,r}}+\mu^{1-\f1{2r}}\|\nabla u^{h,n}\|_{L^{2r}(\R_+;L^{\f{dr}{2r-1}}_+)}\leq C_1\mu^{1-\f1r}\eta_0,\\
&\|(u^{d,n},\nabla\Pi^n)\|_{X^{p,r}}+\mu^{1-\f1{2r}}\|\nabla
u^{d,n}\|_{L^{2r}(\R_+;L^{\f{dr}{2r-1}}_+)}\leq C_2\mu^{1-\f1r}
  \|u_0\|_{\dot{B}^{-1+\f{d}{p}}_{p,r}(\R^d_+)}.
\end{split}
 \end{equation}
In addition, by following the computations leading to
\eqref{eq:th2b} for $u^n,$ it is not difficult to see that the
assumption that $\wt p>d$ is not needed if it is only a matter of
getting a control on  the norm of the solution in $X^{\wt p,r}.$
Therefore we also have
 \begin{equation}\label{3.1bb}
 \|(u^n,\nabla\Pi^n)\|_{X^{\wt p,r}}\leq C_3\mu^{1-\f1r}\|u_0\|_{\dot B^{\f d{\wt p}-1}_{\wt p,r}(\R^d_+)}
 \exp\Bigl(C_4\mu^{-2r}\|u_0^d\|_{\dot{B}^{-1+\f{d}{p}}_{p,r}(\R^d_+)}^{2r}\Bigr).
\end{equation}
Of course, if  $\wt p>d,$ then we also have  a bound for $\nabla u^n$ in $L^q(\R_+;L^\infty_+).$

\subsubsection*{Step 3. The proof of convergence}

Owing to the low regularity of $a_0,$ it is not clear that one may still use stability estimates
in order to prove the convergence of the sequence defined in the previous step.
In effect, as pointed out in the previous section, there is a loss of one derivative in the stability estimates
for the density.
Therefore, we shall use compactness arguments instead,
borrowed from  \cite{HPZ3}.
 For completeness, we outline the proof here.

 According to the previous step,  $(\p_t u^n)_{n\in\N}$ is uniformly bounded in $L^r(\R_+;L^p_+).$ Combining with \eqref{3.1aa}, \eqref{3.1bb},
Ascoli-Arzela Theorem and compact embeddings in Besov spaces, we
conclude that there exists a subsequence, of
$(a^n,u^n,\nabla\Pi^n)_{n\in\N}$ (still denoted by
$(a^n,u^n,\nabla\Pi^n)_{n\in\N}$) and some $(a,u, \nabla\Pi)$ with
$a\in L^\infty(\R_+\times\R^d_+),$
$$\na
u\in L^{2r}(\R_+; L^{\f{dr}{2r-1}}_+)\quad\hbox{and}\quad \p_tu,\nabla^2u, \na\Pi\in
L^r(\R_+;L^{p}_+\cap L^{\wt p}_+)$$ (and also $\nabla u\in L^q(\R_+;L^\infty_+)$ if $\wt p>d$),
 such that \beq \label{3.6a}
\begin{aligned}
&a^n \rightharpoonup  a\quad \mbox{weak}\ \ast \ \mbox{in}\
L^\infty(\R_+\times\R^d_+),\\
& \nabla^2 u^n \rightharpoonup \nabla^2 u \quad\mbox{and}\quad \na\Pi^n
\rightharpoonup \na\Pi\quad \mbox{weakly}\ \ \mbox{in}\ L^r(\R_+;L_+^p),
\end{aligned}
\eeq
with in addition for all small enough $\eta>0,$
\beq\label{3.6aa}
\begin{aligned}
&  u^n \rightarrow u \quad \mbox{strongly}\  \ \mbox{in}\
L^{2r}_{loc}(\R_+; L^{\f{dr}{r-1}-\eta}_{loc}(\R^d_+)),\\
&  \na u^n \rightarrow \na u \quad \mbox{strongly}\  \ \mbox{in}\
L^{2r}_{loc}(\R_+; L^{\f{dr}{2r-1}-\eta}_{loc}(\R^d_+)).\\
\end{aligned}
\eeq
 By construction, $(a^n, u^n,\nabla\Pi^n)$ satisfies \beq\label{3.6b}
\begin{aligned}
&\int_0^\infty\int_{\R^d_+}a^n(\p_t\phi+u^{n}\cdot\na\phi)\,dx\,dt+\int_{\R^d_+}\phi(0,x)a_{0}^n(x)\,dx=0,\\
&\hspace{3cm} \int_0^\infty\int_{\R^d_+}\phi\dive u^n\,dx\,dt=0\qquad\mbox{and}\\
 &
\int_0^\infty\int_{\R^d_+}\Bigl\{u^n\cdot\p_t\Phi+((u^{n}\otimes u^{n}) :\na \Phi\bigr) +(1+a^{n})(\mu\na
u^n-\na\Pi^n)\cdot\Phi\Bigr\}\,dx\,dt\\
&\hspace{9cm}+\int_{\R^d_+}u_{0}^n\cdot\Phi(0,x)\,dx=0,
\end{aligned}
\eeq for all  test functions $\phi, \Phi$ given by Definition
\ref{defi2.1}.
\medbreak
Putting \eqref{3.6a} and \eqref{3.6aa} together,
it is easy to pass to the limit in all the terms of \eqref{3.6b}, except in $a^n(\mu\Delta u^n-\nabla\Pi^n).$
To handle that term, it suffices  to show that $a^n\to a$ in $L^m_{loc}(\R_+\times\R^d_+)$ for
some~$m\geq r'.$
\medbreak
Now,  it is easy to observe from the transport equation that
 \beno \p_t (a^n)^2+\dive (u^{n} (a^n)^2)=0, \eeno
from which, \eqref{3.6a} and \eqref{3.6aa}, we deduce that \beq \p_t
\overline{a^2}+\dive (u\overline{a^2})=0, \label{3.6c} \eeq where we
denote by $\overline{a^2}$ the weak $\ast$ limit of
$((a^n)^2)_{n\in\N}.$ \medbreak Thanks to \eqref{3.6a},
\eqref{3.6aa} and \eqref{3.6b}, there holds \beno \p_ta+\dive(u a)=0
\eeno in the sense of distributions. Moreover, as $\na u\in
L^{2r}(\R_+; L^{\f{dr}{2r-1}}(\R^d_+))$ and $\dive u=0,$  we infer by a mollifying
argument as that in \cite{LP} that \beq\label{3.6d}
\p_ta^2+\dive(ua^2)=0. \eeq Subtracting \eqref{3.6d} from
\eqref{3.6c}, we obtain \beq\label{3.6e} \p_t
(\overline{a^2}-a^2)+\dive (u(\overline{a^2}-a^2))=0, \eeq from
which and Theorem II.2 of \cite{LP} concerning the uniqueness of
solutions to transport equations, we infer
 \beno
(\overline{a^2}-a^2)(t,x)=0\quad \mbox{a. e. }\
x\in\R^d_+\quad\mbox{and}\quad t\in \R^+.\eeno Together with the
fact that $(a^n)_{n\in\N}$ is uniformly bounded in
$L^\infty(\R_+\times\R^d_+),$  this implies that \beq\label{3.6f}
a^n\ \rightarrow\ a\quad \mbox{strongly in}\quad
L^m_{loc}(\R_+\times\R^d)\quad\hbox{for all }\ m<\infty. \eeq Granted with
this new information, it is now easy to pass to the limit   in
\eqref{3.6b}. Therefore  $(a,u,\nabla\Pi)$ satisfies \eqref{def2.1a}
and \eqref{def2.1b}. Moreover, thanks to \eqref{3.1aa} and
\eqref{3.1bb}, there hold \eqref{thm1aa}, \eqref{thm1ab}  and \eqref{thm1b}.
Besides, as $(u,\nabla\Pi)$ satisfies \eqref{eq:velocity} and the r.h.s. is in $L^r(\R_+;L^p_+\cap L^{\wt p}_+),$ the time
continuity for $u$ stems from Proposition \ref{p:stokes}.  This
completes the proof of the existence part of Theorem \ref{th:main}.

%%%%%%%%%%%%%%%%%%%%%%%%%%%%%%%%%%%%%%%%%%%%%%%%%%%%%%%%%%%%%%%%%%

\section{More general data}\label{s:general}

\setcounter{equation}{0}
Until now, we  assumed that $p$ and $r$ where interrelated through
\begin{equation}\label{eq:relation}
-1+\frac dp=2-\frac 2r\cdotp
\end{equation}
It is natural to investigate  whether the inhomogeneous Navier-Stokes equations may
still be solved with initial data $(a_0,u_0)$ in $L^\infty_+\times \dot\cB^{-1+\frac dp}_{p,r}(\R^d_+)$
if \eqref{eq:relation} is \emph{not} satisfied.

The case where $1<r<2p/(3p-d)$ or, equivalently $p<\frac{dr}{3r-2}$ is not
so interesting because, by embedding one may find some $p^*\in(p,d)$ so that
$u_0\in \dot\cB^{-1+\frac d{p^*}}_{p^*,r}(\R^d_+)$ and \eqref{eq:relation} is fulfilled by $(p^*,r).$

The case where $r>2p/(3p-d)$ is more involved and cannot be solved  by taking advantage of embeddings.
In order to explain how this may be overcome anyway,
let us first focus on the toy case where $u$ satisfies  the
basic heat equation
$$\pa_tu-\Delta u=0\quad\hbox{in}\quad\R_+\times\R^d
$$
with initial data  $u^0\in \dot B^{-1+\frac dp}_{p,r}(\R^d).$
Then  by using
embedding in  $\dot B^{-1+\frac d{\wt p}}_{\wt p,\wt r}(\R^d)$ for any $\wt p\geq p$ and $\wt r\geq r,$
we  easily get (see Lemma \ref{l:D} in the appendix) that
\begin{enumerate}
\item  $t^\alpha \nabla^2u \in L^r(\R_+;L^p(\R^d))$ with $\alpha=\frac32-\frac{d}{2p}-\frac1r$
if $p>d/3,$
\item $t^\beta \nabla u\in L^{r_2}(\R_+;L^{p_2}(\R^d))$ with $\beta= 1-\frac d{2p_2}-\frac1{r_2}$
if $p_2\geq p,$ $r_2\geq r$  and $p_2>d/2,$
\item   $t^\gamma u\in L^{r_3}(\R_+;L^{p_3}(\R^d))$ with $\gamma=\frac12-\frac d{2p_3}-\frac1{r_3}$
if $p_3\geq p,$ $r_3\geq r$  and $p_3>d.$
 \end{enumerate}
As pointed out in Proposition \ref{p:stokesbis},  those properties
are still true for the free solution to the Stokes system in  the
half-space. Keeping in mind that we want to apply those types of
estimates to System \eqref{eq:velocity}, we see that we need to be
able to handle also the Stokes system with some source term $f$
satisfying $t^\alpha f\in L^r(\R_+;L^p_+).$ Still in the simpler
case of the heat equation:
$$
\pa_t v -\Delta v=f\quad\hbox{in}\quad \R_+\times\R^d\quad\hbox{with }\  t^\alpha f\in L^r(\R_+;L^p_+),
$$
  it has been observed  by the second author and collaborators in \cite{HPZ3} (see also
the Appendix) that if $\alpha r'<1$ then
\begin{enumerate}
\item  $t^\alpha \nabla^2v \in L^r(\R_+;L^p(\R^d)),$\smallbreak
\item $t^\beta \nabla v\in L^{r_2}(\R_+;L^{p_2}(\R^d))$ with $\beta=\alpha+\frac d2\bigl(\frac1p-\frac1{p_2}\bigr)-\frac12+\frac1r-\frac1{r_2}$
if $p_2\geq p,$ $r_2\geq r$ and $\frac dp-\frac d{p_2}<1+\frac2{r_2}-\frac2{r},$\smallbreak
\item   $t^\gamma v\in L^{r_3}(\R_+;L^{p_3}(\R^d))$ with $\gamma=  \alpha+\frac d2\bigl(\frac1p-\frac1{p_3}\bigr)-1+\frac1r-\frac1{r_3} $
if $p_3\geq p,$ $r_3\geq r$  and $p_3>d$ and   $\frac dp-\frac
d{p_3}<2+\frac2{r_3}-\frac2{r}\cdotp$
 \end{enumerate}
In the case we are interested in, owing to the presence of $u\cdot\nabla u$ in   \eqref{eq:velocity} and to H\"older inequality,
the following supplementary relations have to be fulfilled:
\begin{equation}\label{eq:p}
\frac 1p=\frac 1{p_2}+\frac 1{p_3},\quad \frac 1r=\frac 1{r_2}+\frac 1{r_3}\quad\hbox{and}\quad \alpha=\beta+\gamma.
\end{equation}
Under the first two conditions, if se set
\begin{equation}\label{eq:coeff}
\alpha\eqdefa \frac32-\frac{d}{2p}-\frac1r,\quad
\beta\eqdefa 1-\frac d{2p_2}-\frac1{r_2},\quad
\gamma\eqdefa\frac12-\frac d{2p_3}-\frac1{r_3},
 \end{equation}
then  the relationships above between $\alpha,$ $\beta$ and $\gamma$ are satisfied.
Let us emphasize that  if $(p,r)$ with $d/3<p<d$ has been chosen so that  $r>2p/(3p-d)$ (which is equivalent to $0<\alpha<1-1/r$),
 then one may take any $(p_2,r_2)$ such that $p_2\geq p,$ $p_2>d/2,$ $r_2\geq r$ and
 \begin{equation}\label{eq:p2}
 \frac dp-1+\frac2{r}-\frac2{r_2}<\frac d{p_2}<2-\frac2{r_2}\cdotp
 \end{equation}
 The assumption on $(p,r)$ ensures that such a couple exists.
This motivates the following statement:
\begin{thm}\label{th:main2}{\sl
Assume that $r\in(1,\infty)$ and $p\in(d/3,d)$ satisfy
$r>2p/(3p-d).$ Then for any data $u_0\in\dot\cB^{-1+\frac
dp}_{p,r}(\R^d_+)$ and $a_0\in L^\infty_+$ fulfilling the smallness
condition  \beq \label{small3}
\|u_0\|_{\dot{B}^{-1+\f{d}{p}}_{p,r}(\R^d_+)}\leq
c_2\mu\quad\hbox{and}\quad \|a_0\|_{L^\infty_+}\leq c_3 \eeq for
some sufficiently small  positive constants $c_2=c_2(r,d)$ and
$c_3=c_3(r,d),$ System \eqref{INS} has a global solution $(a,u)$
with $\|a(t)\|_{L^\infty_+}=\|a_0\|_{L^\infty_+}$ for all
$t\in\R_+,$ and
\begin{equation}\label{eq:main2}
t^\alpha(\pa_tu,\nabla^2u,\nabla\Pi)\in L^r(\R_+;L^p_+),\quad
t^\beta\nabla u\in L^{r_2}(\R_+;L^{p_2}_+),\quad
t^\gamma u\in L^{r_3}(\R_+;L^{p_3}_+)
\end{equation}
with $\alpha,\beta,\gamma$ defined in \eqref{eq:coeff}, $(p_2,r_2)$
satisfying \eqref{eq:p2}, and $(p_3,r_3)$ defined in \eqref{eq:p}.
\smallbreak Furthermore, for any $1<\sigma<\frac{r}{1+\alpha r},$
the fluctuation $u-u_L$ belongs to
$\cC(\R_+;\dot\cB^{2-\f2{\s}}_{p,\sigma}(\R^d_+))$  where $u_L$ stands for the free solution of the Stokes system with initial data $u_0.$}
\end{thm}

\begin{rmk}
By following the proof of Theorem \ref{th:main} in this paper and
that of Theorem 1.1 in \cite{HPZ3}, we can also extend Theorem
\ref{th:main3} to the case where the (weaker)  anisotropic smallness condition \eqref{small1} is fulfilled. For a
concise presentation, we shall not pursue this direction here.
\end{rmk}

\begin{proof}
We may assume that  $\mu=1,$ as the usual rescaling gives the result in the general case.

We smooth out the initial velocity $u_0$ into a sequence
$(u_0^n)_{n\in\N}$ satisfying the assumptions of theorem
\ref{th:main}: we take $p_0$ so that $-1+\frac d{p_0}=2-\frac2r$ and
require $u_{0,n}$ to be in $\dot\cB^{-1+\frac
d{p_0}}_{p_0,r}(\R^d_+)\cap \dot\cB^{-1+\frac
d{p_0}}_{p,r}(\R^d_+),$ and to converge to $u_0$ in $\dot
B^{-1+\frac d{p_0}}_{p,r}(\R^d_+).$ To construct $u_{0,n},$ one may
first consider the zero extension $\tilde u_0=e_0(u_0)$ of $u_0$ to
$\R^d,$ then approximate it by compactly supported divergence free
vector-fields by means of the stream function.

\subsubsection*{Step 1. Uniform estimates}
The corresponding solution $(a^n,u^n,\nabla\Pi^n)$  satisfies in particular
$$
\p_t u^n,\nabla^2u^n,\nabla\Pi^n\in L^r(\R_+;L^{p_0}_+\cap L^p_+).
$$
Hence H\"older inequality ensures that $t^\alpha(\p_t u^n,\nabla^2u^n,\nabla\Pi^n)\in L^r(0,T;L^p_+)$ for all $T>0.$
A similar argument ensures that $t^\beta\nabla u^n$ and $t^\gamma u^n$ are in $L^{r_2}(0,T;L^{p_2}_+)$
and  $L^{r_3}(0,T;L^{p_3}_+),$ respectively, for all $T>0.$
Now, because $(a^n,u^n,\nabla\Pi^n)$ satisfies
 $$\p_tu^n-\D u^n+\na\Pi^n=F^n \eqdefa a^n(\D u^n-\na\Pi^n)-u^n\cdot\na u^n,
$$ Proposition \ref{p:stokesbis} implies that
$$
\begin{aligned}
Z_n(t)\eqdefa&\|t^\gamma u^n\|_{L_t^{r_3}(L^{p_3}_+)}+\|t^\beta\nabla u^n\|_{L_t^{r_2}(L^{p_2}_+)}+
\|t^\alpha(\p_tu^n,\nabla^2u^n,\nabla\Pi^n)\|_{L^r_t(L^{p}_+)}\\
&\hspace{7cm}\leq C\bigl(\|u_0^n\|_{\dot B^{-1+\frac dp}_{p,r}(\R^d_+)}+\|t^\alpha F^n\|_{L^r_t(L^{p}_+)}\bigr).
\end{aligned}
$$
Taking advantage of H\"older inequality, of the relationship between $(r_2,p_2)$ and $(r_3,p_3),$ and of the conservation of
$\|a^n(t)\|_{L^\infty_+},$  we
 may write
$$
\|t^\alpha F^n\|_{L^r_t(L^{p}_+)}\leq
\|a_0\|_{L^\infty_+}\bigl(\|t^\alpha\D u^n\|_{L^r_t(L^{p}_+)}+\|t^\alpha\na\Pi^n\|_{L^r_t(L^{p}_+)}\bigr)
+\|t^\gamma u^n\|_{L^{r_3}_t(L^{p_3}_+)}\|t^\beta \nabla u^n\|_{L^{r_2}_t(L^{p_2}_+)}.
$$
 Therefore taking $c_2, c_3$ small enough in \eqref{small3}, we get that \beno Z_n(t)\leq C\bigl(\|u_0\|_{\dot
B^{-1+\f dp}_{p,r}(\R^d_+)}+Z_n^2(t)\bigr), \eeno so that as long as
$\|u_0\|_{\dot B^{-1+\f dp}_{p,r}(\R^d_+)}\leq \frac{1}{4C^2},$ we
have \beno Z_n(t)\leq 2C\|u_0\|_{\dot B^{-1+\frac
dp}_{p,r}(\R^d_+)}.\eeno

\subsubsection*{Step 2. Convergence}

{}From Step 1, we know that sequence $(a^n,u^n,\nabla\Pi^n)_{n\in\N}$ is bounded in the space defined in \eqref{eq:main2}.
In order to complete the proof of existence, we have to establish
convergence, up to extraction, to a solution $(a,u,\nabla\Pi)$
of \eqref{INS} in the desired functional space.
For that, we first notice that
$$
\Delta u^n=t^{-\alpha}\:(t^\alpha\Delta u^n).
$$
H\"older inequality guarantees that $(\Delta u^n)_{n\in\N}$ is bounded in $L^\sigma_{loc}(\R_+;L^p_+)$
for any $\sigma<r/(1+\alpha r).$  Note that  because $\alpha+1/r<1,$ one may take $\sigma>1.$
Similarly, we have $(\p_tu^n)_{n\in\N}$ and $(\nabla\Pi^n)_{n\in\N}$ bounded in $L^\sigma_{loc}(\R_+;L^p_+).$
Therefore, setting $\wt u^n\eqdefa u^n-u^n_L$ where $u^n_L$ stands for the free solution to the Stokes
system with initial data $u_0^n,$
we conclude that $(\wt u^n)_{n\in\N}$ is bounded in $\cC(\R_+;\dot
B^{2-\frac2\s}_{p,\sigma}(\R^d_+)).$

Next, writing  $\nabla u^n=t^{-\beta}\:(t^\beta\nabla u^n)$ and
$u^n=t^{-\gamma}\:(t^\gamma u^n),$
we get $(\nabla  u^n)_{n\in\N}$ and $(u^n)_{n\in\N}$ bounded in $L^{\sigma_2}(\R_+;L^{p_2}_+)$ and
$L^{\sigma_3}(\R_+;L^{p_3}_+),$ respectively.
As we may choose $\sigma_2$ and $\sigma_3$ as close to (but smaller than) $r_2/(1+\beta r_2)$ and
$r_3/(1+\gamma r_3)$ as we want,
one may ensure that
\begin{equation}\label{eq:condsigma}
\f1{\sigma_2}+\f1{\sigma_3}<1.
\end{equation}
Now, combining with the boundedness of $(\D\wt u^n)_{n\in\N}$ in
$L^\sigma_{loc}(\R_+;L^p_+)$ and using Arzela-Ascoli theorem, we
conclude that, up to extraction, sequence
$(a^n,u^n,\nabla\Pi^n)_{n\in\N}$ converges weakly to some triple
$(a,u,\nabla\Pi)$
 with $a\in L^\infty(\R_+\times\R^d_+),$
$$t^\alpha(\p_tu,\nabla^2u,\nabla\Pi)\in L^r(\R_+;L^p_+),\quad
t^\beta\nabla u\in L^{r_2}(\R_+;L^{p_2}_+),\quad
t^\gamma u\in L^{r_3}(\R_+;L^{p_3}_+).
$$
 More precisely, we have
 $$
\begin{aligned}
&a^n \rightharpoonup  a\quad \mbox{weak}\ \ast \ \mbox{in}\
L^\infty(\R_+\times\R^d_+),\\
& \nabla^2 u^n \rightharpoonup \nabla^2 u \quad\mbox{and}\quad \na\Pi^n
\rightharpoonup \na\Pi\quad \mbox{weakly}\ \ \mbox{in}\ L^\sigma(\R_+;L_+^p),
\end{aligned}
$$
with in addition for all small enough $\eta>0,$
$$
\begin{aligned}
&  u^n \rightarrow u \quad \mbox{strongly}\  \ \mbox{in}\
L^{\sigma_3}_{loc}(\R_+; L^{p_3-\eta}_{loc}(\R^d_+)),\\
&  \na u^n \rightarrow \na u \quad \mbox{strongly}\  \ \mbox{in}\
L^{\sigma_2}_{loc}(\R_+; L^{p_2-\eta}_{loc}(\R^d_+)).
\end{aligned}
$$
Because \eqref{eq:condsigma} is satisfied, passing to the limit in System \eqref{INS} follows from the same arguments
as in the proof of Theorem \ref{th:main}.
That the constructed solution  has  all the properties listed in Theorem \ref{th:main2} is left to the reader.  This completes the proof of existence.
\end{proof}

As in Theorem \ref{th:main}, assuming just critical regularity for the velocity does not seem
to be enough to ensure uniqueness.  For sure, as the velocity $u$ does not satisfy $\nabla u\in L^1_{loc}(\R_+;L^\infty_+),$
 we cannot resort to the Lagrangian approach.
This motivates the following statement.
\begin{thm}\label{th:main3}{\sl
In addition to the hypotheses of Theorem \ref{th:main2}, assume that
$u_0$ belongs to $\dot\cB^{-1+\frac dp}_{\wt p,r}(\R^d_+)$ for some finite $\wt p>p.$
Then \eqref{INS} has
a global solution $(a,u,\nabla\Pi)$ fulfilling
the properties of Theorem \ref{th:main2} and, in addition,
$$
t^\alpha(\p_tu,\nabla^2u,\nabla\Pi)\in L^r(\R_+;L^{\wt p}_+),\quad
t^\beta\nabla u\in L^{r_2}(\R_+;L^{\wt p_2}_+),
\quad t^\gamma u\in L^{r_3}(\R_+;L^{\wt p_3}_+)
$$
with $\alpha,$ $\beta$ and $\gamma$ defined as previously in
\eqref{eq:coeff},
\begin{equation}\label{eq:exp1}
\f1{\wt p_2}-\f1{p_2}=\f1{\wt p}-\f1p\quad\hbox{and}\quad
\f1{\wt p_3}-\f1{p_3}=\f1{\wt p}-\f1p,
\end{equation}
whenever $p_2$ and $r_2$ may be  chosen so  that \eqref{eq:p2} is fulfilled and
\begin{equation}\label{eq:exp2}
\frac{2d}p-\frac d{\wt p}-\f d{p_2}\leq 1+\frac2{r_2}-\frac2r\cdotp
\end{equation}
If in addition $\wt p>d$ then there exists some positive $\delta<\alpha$ and $s>r$ so that
 \begin{equation}\label{eq:lip}
 t^\delta\nabla u\in L^s(\R_+;L^\infty_+),
 \end{equation}
 and  the constructed solution is unique in its class of regularity.}
\end{thm}
\begin{rmk}{\sl
Despite the appearances, it is always possible to take $\tilde p>d$
in the above statement.
At first sight it seems not obvious because a necessary condition
for having $\tilde p>d$ in  \eqref{eq:exp2} is that $p>\f{dr}{2r-1}.$
However, by embedding, one may always find some $p_1\in(\f{dr}{2r-1},d)$
so that $u_0\in \dot B^{-1+\f d{p_1}}_{p_1,r}(\R^d_+)\cap \dot B^{-1+\f d{p_1}}_{\wt p,r}(\R^d_+)$
and thus replace $p$ by $p_1.$}
\end{rmk}
\begin{proof}
The scheme for proving existence is exactly the same as in the previous statement.
Therefore we remain at the level of a priori estimates.
Let $Z$ and $\wt Z$ be defined on $\R_+$ by
$$\begin{array}{ccc} Z(t)&\eqdefa& \|t^\gamma u\|_{L_t^{r_3}(L^{p_3}_+)}+\|t^\beta\nabla u\|_{L_t^{r_2}(L^{p_2}_+)}+
\|t^\alpha(\p_tu,\nabla^2u,\nabla\Pi)\|_{L^r_t(L^{p}_+)}\\[1ex]
 \wt Z(t)&\eqdefa& \|t^\gamma u\|_{L_t^{r_3}(L^{\wt p_3}_+)}+\|t^\beta\nabla u\|_{L_t^{r_2}(L^{\wt p_2}_+)}+
\|t^\alpha(\p_tu,\nabla^2u,\nabla\Pi)\|_{L^r_t(L^{\wt
p}_+)}.\end{array}
$$
As in  the proof of Theorem \ref{th:main2} and under Condition \eqref{small1}, we have for some $C=C(p,r,d),$
\begin{equation}\label{eq:Z}
 Z(t)\leq C\|u_0\|_{\dot B^{-1+\f dp}_{p,r}(\R^d_+)}.
 \end{equation}
 Next, keeping in mind our assumptions on the Lebesgue exponents $p,$ $p_2,$ $r_2,$ $\wt p$ and $\wt p_2,$
 Proposition \ref{p:stokesbis} ensures that for some constant $C=C(p,\wt p,r,d),$
 $$
 \wt Z(t)\leq C\Bigl(\|u_0\|_{\dot B^{-1+\f dp}_{\wt p,r}(\R^d_+)}
 +\|t^\alpha\,a(\Delta u-\nabla\Pi)\|_{L_t^r(L^{\wt p}_+)}+\|t^\alpha\,u\cdot\nabla u\|_{L_t^r(L^{\wt p}_+)}\Bigr).
   $$
   Because $\|a(t)\|_{L^\infty_+}$ is constant during the evolution,
   under Condition \eqref{small1}, the first term of the right-hand side may be absorbed by
   the left-hand side. As for the last term, we use H\"older inequality and the fact that
   $$
   \alpha=\beta+\gamma,\quad \frac1r=\frac1{r_2}+\frac1{r_3}\ \hbox{ and }\
   \frac1{\wt p}=\f1{\wt p_2}+\f1{p_3}\cdotp
   $$
   We thus end up with
   $$
   \wt Z(t)\leq C\Bigl(\|u_0\|_{\dot B^{-1+\f dp}_{\wt p,r}(\R^d_+)}
 +\|t^\beta\nabla u\|_{L_t^{r_2}(L^{\wt p_2}_+)}\|t^\gamma u\|_{L_t^{r_3}(L^{p_3}_+)}\Bigr),
 $$
 whence,
 if $c$ is small enough in \eqref{small1},
   \begin{equation}\label{eq:tildeZ}
 \wt Z(t)\leq C\|u_0\|_{\dot B^{-1+\f dp}_{\wt p,r}(\R^d_+)}.
 \end{equation}
In order to prove \eqref{eq:lip}, we first have to check whether one may take $\wt p>d,$
knowing that Conditions \eqref{eq:p2},  \eqref{eq:exp1} and \eqref{eq:exp2} have to be fulfilled.
This is in fact equivalent to $p>\f{dr}{2r-1}.$
Assuming from now on that this condition is fulfilled, and taking $\wt p>d,$ we may use the
following Gagliardo-Nirenberg inequality for all functions\footnote{As now usual,
this inequality may be deduced from the classical one on $\R^d,$ taking advantage of the
antisymmetric extension operator.}
of $L^{\wt p_3}_+\cap W^{1,\wt p}_0(\R^d_+)\cap  W^{2,\wt p}(\R^d_+)$:
\begin{equation}\label{eq:GN3}
\|\nabla u\|_{L^\infty_+}\leq C\|\nabla^2u\|_{L_+^{\wt p}}^\theta\|u\|_{L_+^{\wt p_3}}^{1-\theta}\quad\hbox{ with}\quad
\theta\eqdefa\f{1+\f d{\wt p}-\f d{p_2}}{2-\f d{p_2}}\cdotp
\end{equation}
Then, using H\"older inequality and \eqref{eq:tildeZ}, we readily get
$$
\|t^\delta \nabla u\|_{L_t^s(L^\infty_+)}\leq
C\|t^\alpha\nabla^2u\|_{L_t^r(L_+^{\wt p})}^\theta\|t^\gamma
u\|_{L_t^{r_3}(L_+^{\wt p_3})}^{1-\theta} \leq C'\|u_0\|_{\dot
B^{-1+\f dp}_{\wt p,r}(\R^d_+)}<\infty
$$
with
\begin{equation}\label{eq:s}
\delta\eqdefa \theta\alpha+(1-\theta)\gamma\ \hbox{ and }\
\f1s=\f\theta{r}+\f{1-\theta}{r_3}\cdotp
\end{equation}
Because $r_3>r$ and $\gamma<\alpha,$ it is obvious that  $s>r$ and $\delta<\alpha.$
This completes the proof of the second part of the statement.
Proving uniqueness is postponed to the next section.
\end{proof}

%%%%%%%%%%%%%%%%%%%%%%%%%%%%%%%%%%%%%%%%%%%
\setcounter{equation}{0}

\section{Uniqueness}\label{s:uniqueness}

This section is devoted to proving the uniqueness parts of Theorems \ref{th:main} and \ref{th:main3}.
As in \cite{dm, HPZ3},   it strongly relies on the fact that
for smooth enough solutions, one may use  the \emph{Lagrangian formulation} of \eqref{INS},
which turns out to be equivalent to \eqref{INS}.

\subsection{Lagrangian coordinates}

Before going into the detailed proof of uniqueness, we here  recall some basic facts
concerning Lagrangian coordinates.
Throughout, we are given some smooth enough
solution $(a,u,\nabla\Pi)$ to \eqref{INS}    (typically we assume that  $(u,\nabla\Pi)\in X_T^{p,r}\cap
X_T^{\wt p,r}$ with $(p,r)$ and $\wt p>d$  as in the statement
of Theorem \ref{th:main}).
Then we set $$
b(t,y)\eqdefa a(t, X(t,y)), \quad v(t,y)\eqdefa u(t,X(t,y))\quad\mbox{and}\quad P(t,y)\eqdefa \Pi(t, X(t,y))
$$
where, for any
$y\in\R^d,$  $X(\cdot,y)$ stands for the solution to the following ordinary differential equation on $[0,T]$:
 \beq\label{ode} \f{dX(t,y)}{dt}= u(t,
X(t,y)), \qquad X(t,y)|_{t=0}=y.
\eeq
Therefore we have
the following relation between the \emph{Eulerian coordinates} $x$ and the
\emph{Lagrangian coordinates} $y$:
\beq\label{ode1} x=X(t,y)=y+\int_0^tv(\tau, y)\,d\tau. \eeq
Let $Y(t,\cdot)$ be the
inverse mapping of $X(t,\cdot)$, then $D_xY(t,x)=(D_yX(t,y))^{-1}$
with $x=X(t,y).$  Furthermore, if
  \beq\label{lipassume} \int_0^T\|\na
v(t)\|_{L^\infty}\,dt<1\eeq
then one may write
\beq\label{h.12ad} D_xY
=(\Id+(D_yX-\Id))^{-1}=\sum_{k=0}^{\infty}(-1)^k\Bigl(\int_0^tD_yv(\tau,y)\,d\tau\Bigr)^k.\eeq
Setting $A(t,y)\eqdefa (D_yX(t,y))^{-1}=D_x Y(t,x)$ for $x=X(t,y),$
one may prove (see the Appendix of \cite{dm}) that
\beq\label{h.12}
 \nabla_x u(t,x)={}^T\!A(t,y)\nabla_y v(t,y)\quad\mbox{ and}\quad
\dive_x u(t,x)=\dive_y(A(t,y) v(t,y)). \eeq By the chain rule, we
also have\footnote{Here and in what follows, we denote by $
{}^T\!A$ the transpose matrix of $A,$  $\bigl(\na
u\bigr)_{i,j}=\bigl(\p_iu^j\bigr)_{1\leq i,j\leq d},$ and $D
u={}^T\na u=\bigl(\p_ju^i\bigr)_{1\leq i, j\leq d}.$}
\ben\label{eq:A-div} \dive_y\big(Av\big)={}^T\!A:\na_yv=D_yv: A={\rm
Tr}\,(D_yv\cdot A).\een As in \cite{dm2}, we denote
$$
\na_u\eqdefa {}^T\!A\cdot\na_y,\quad
\dive_u\eqdefa\dive_y(A\cdot)\quad\mbox{and}\quad
\D_u\eqdefa\dive_u\na_u.
$$
Note that for any $t>0,$  the  solution of \eqref{INS}
obtained in  Theorem \ref{th:main} satisfies the smoothness assumption
of Proposition 2 in \cite{dm2}, so that $(b, v, \na P)$ satisfies  \begin{equation}\label{INSL}
 \quad\left\{\begin{array}{lll}
\displaystyle b_t=0&\quad\mbox{in}\quad& \R_+\times\R^d_+,\\
\displaystyle \pa_t v -(1+b)(\Delta_u v- \grad_u P)=0&\quad\mbox{in}\quad& \R_+\times\R^d_+, \\
\displaystyle \dive_u  v = 0&\quad\mbox{in}\quad& \R_+\times\R^d_+, \\
\displaystyle v=0&\quad\mbox{on}\quad& \R_+\times\p\R^d_+,\\
\displaystyle (b,v)|_{t=0}=(a_0, u_0)&\quad\mbox{in}\quad& \R^d_+,
\end{array}\right.
\end{equation}
which is  the Lagrangian formulation of \eqref{INS}.

%%%%%%%%%%%%%%%%%%%%%%%

\subsection{Proving  uniqueness : the  ``smooth case''}

Here we prove   the uniqueness part of Theorem
\ref{th:main}.
Thanks to the rescaling \eqref{eq:change}, one may assume with no loss of generality that $\mu=1.$
 Let $(a_i,u_i,\Pi_i), i=1,2,$ be two
solutions of \eqref{INS} satisfying  \eqref{thm1aa}, \eqref{thm1ab} and \eqref{thm1b}.
%With no loss of generality, one may assume in addition that $\wt p\leq \f{dr}{r-1},$
%and also that $\wt p<\f{dr}{3r-4}$ if $r>\f43\cdotp$

For $i=1,2,$ let $X_i$ be the flow of $u_i$ (defined in \eqref{ode1}) and denote by $(v_i, P_i)$
the corresponding velocity and pressure in Lagrangian coordinates.
Let \beno \dv\eqdefa v_2-v_1,\quad
\dP=P_2-P_1.\eeno Observing that $b\equiv a_0,$ we see that   $(\dv, \nabla\dP)$ satisfies
\begin{equation}\label{differ}
 \quad\left\{\begin{array}{lll}
\displaystyle \pa_t \dv -\Delta\dv+ \grad \dP=a_0(\Delta\dv- \grad \dP)+\df_1 +\df_2\eqdefa \dF&\quad\mbox{in}\quad& \R_+\times\R^d_+, \\
\displaystyle \dive\, \dv = \dg=\dive \dR&\quad\mbox{in}\quad& \R_+\times\R^d_+,\\
\end{array}\right.
\end{equation}
where \beq\label{source} \begin{aligned} \df_1
&\eqdefa(1+a_0)[(\Id- {}^T\!A_2)\na \dP-\dA \na
P_1]\quad\hbox{with }\ \dA\eqdefa A_2-A_1,\\
\df_2 &\eqdefa(1+a_0)\dive[(A_2 {}^T\!A_2-\Id)\na {\dv}+
(A_2 {}^T\!A_2-
A_1 {}^T\!A_1)\na v_1],\\
\dg &\eqdefa (\Id-{}^T\!A_2):\nabla \dv-{}^T\!\dA:\nabla v_1,\\
\dR &\eqdefa(\Id-A_2){\dv}-\dA\: v_1.
\end{aligned}
\eeq As $\gamma\dR^d\equiv0$ and $\dv|_{t=0}=0,$
applying Proposition \ref{p:stokes} with  $p=\f{dr}{3r-2},$   and $(q,m)=(2r,\frac{dr}{2r-1})$ or
$(q,m)=(r,\frac{dr}{2r-2})$ implies that
\begin{multline}\label{eq:etoile}
\|\na\dv\|_{L^{2r}_t(L^{\f{dr}{2r-1}}_+)}+\|\na\dv\|_{L^{r}_t(L^{\f{dr}{2r-2}}_+)}+\|(\p_t\dv,\nabla^2\dv,\nabla\dP)\|_{L^r_t(L^{
p}_+)}\\ \leq C\Bigl(\|a_0(\Delta\dv-\nabla\dP)\|_{L^r_t(L^{
p}_+)}+\|(\df_1,\df_2,\nabla\dg,\p_t\dR)\|_{L^r_t(L^{p}_+)}\Bigr)\cdotp
\end{multline}
Of course, the first term in the right-hand side may be absorbed by the left-hand side if the constant $c_1$
is small enough in \eqref{small1}. So let us now bound the other terms.
 In what follows, we will use repeatedly that  (see
e.g \cite{dm})
 \beq\label{8.10} \dA(t)=\Bigl(\int_0^t
D\dv\,d\tau\Bigr)\Bigl(\sum_{k\geq 1}\sum_{0\leq
j<k}C_1^jC_2^{k-1-j}\Bigr)\quad\mbox{with}\quad
C_i(t)\eqdefa\int_0^tDv_i\,d\tau.\eeq

\subsubsection*{Bounds for $\dg$} The definition of $\dg$ implies that
$$
\displaylines{\quad \|\na\dg \|_{L^{r}_t(L^{ p}_+)}\leq \|\na
A_2\otimes \nabla\dv\|_{L^{r}_t(L^{
p}_+)}+\|(\Id-A_2)\otimes\na^2\dv\|_{L^{r}_t(L^{
p}_+)}\hfill\cr\hfill +\|\na \dA\otimes \nabla v_1\|_{L^{r}_t(L^{
p}_+)}+\|\dA\otimes\na^2v_1\|_{L^{r}_t(L^{p}_+)}.\quad}
$$
Therefore using H\"older inequality,  \eqref{lipassume} and  \eqref{8.10}, we get
$$
\displaylines{\quad
\|\na\dg \|_{L^{r}_t(L^{p}_+)}\leq C\Bigl(   \|\nabla
A_2\|_{L_t^{\infty}(L^{d}_+)}\|\nabla\dv\|_{L^r(L^{\f{dr}{2r-2}}_+)}+
\|(\Id-A_2)\|_{L^\infty_t(L^\infty_+)}\|\nabla^2\dv\|_{L^r_t(L^{
p}_+)}
\hfill\cr\hfill+\|\na\dA\|_{L^\infty_t(L^{p}_+)}\|\nabla v_1\|_{L^r_t(L^\infty_+)}
+\|\dA\|_{L^\infty_t(L^{\f{dr}{2r-2}}_+)}\|\nabla^2v_1\|_{L^r_t(L^{d}_+)}\Bigr)\cdotp\quad}
$$
Remark that  for $\th$ being determined by
$\f1d=\f\th{p}+\f{1-\th}{\wt p},$ \beno
\begin{aligned}
&\|\nabla A_2\|_{L_t^{\infty}(L^{d}_+)}\leq C
\int_0^t\|\na^2v_2\|_{L^d_+}\,dt'\leq Ct^{1-\f1{r}}\|\na^2v_2\|_{L^r_t(L^p_+)}^\th
\|\na^2v_2\|_{L^r_t(L^{\wt p}_+)}^{1-\th},\\
& \|(\Id-A_2)\|_{L^\infty_t(L^\infty_+)}\leq C t^{1-\f1q}\|\na
v_2\|_{L^q_t(L^\infty_+)},\\
&\|\nabla\dA\|_{L_t^\infty(L^p_+)}\leq C t^{1-\f1r}\|\nabla^2\dv\|_{L_t^r(L^p_+)},\\
&\|\dA\|_{L^\infty_t(L^{\f{dr}{2r-2}}_+)}\leq C t^{1-\f1r}\|\na\dv\|_{L^r_t(L^{\f{dr}{2r-2}}_+)},\\
&\|\nabla^2v_1\|_{L_t^r(L^d_+)}\leq C \|\na^2v_1\|_{L^r_t(L^p_+)}^\th
\|\na^2v_1\|_{L^r_t(L^{\wt p}_+)}^{1-\th}.
\end{aligned}
\eeno Hence we obtain \beno \|\na\dg \|_{L^{r}_t(L^{ p}_+)}\leq
\eta_1(t)\bigl(\|\na\d v\|_{L^r_t(L^{\f{dr}{2r-2}}_+)}+\|\na^2\d
v\|_{L^r_t(L^p_+)}\bigr), \eeno with $\eta_1(t)\to 0$ as $t$ goes to
$0.$

 \subsubsection*{Bounds for $\p_t\dR$}
 First, we see that
 $$
 \|\p_t[(\Id\!-\!A_2)\dv]\|_{L^{r}_t(L^{p}_+)} \leq C
\bigl(\|\na v_2\otimes \dv\|_{L^{r}_t(L^{p}_+)}+\|(\Id-A_2)\p_t\dv\|_{L^{r}_t(L^{
p}_+)}\bigr).
$$
It is easy to check that \beno \begin{aligned} \|\na v_2\otimes\dv\|_{L^{r}_t(L^{p}_+)}
\leq & \,C\|\na
v_2\|_{L^{2r}_t(L^{\f{dr}{2r-1}}_+)}\|\dv
\|_{L^{2r}_t(L^{\f{dr}{r-1}}_+)}\\
\leq & \,C\|\na v_2\|_{L^{2r}_t(L^{\f{dr}{2r-1}}_+)}\|\na\dv
\|_{L^{2r}_t(L^{\f{dr}{2r-1}}_+)},
\end{aligned}
\eeno and \beno
\begin{aligned}
\|(\Id-A_2)\p_t\dv\|_{L^{r}_t(L^{ p}_+)}\leq
&\,C\|\Id-A_2\|_{L^\infty_t(L^\infty_+)}\|\p_t\dv\|_{L^r_t(L^p_+)}\\
\leq & \,Ct^{1-\f1q}\|\na v_2\|_{L^q_t(L^\infty_+)}\|\p_t\dv\|_{L^r_t(L^p_+)},
\end{aligned}
\eeno which gives rise to \beno
\|\p_t[(\Id\!-\!A_2)\dv]\|_{L^{r}_t(L^{p}_+)} \leq
\eta_2(t)\bigl(\|\na\dv \|_{L^{2r}_t(L^{\f{dr}{2r-1}}_+)}+\|\p_t\dv\|_{L^r_t(L^p_+)}\bigr)
\quad\hbox{with}\quad\lim_{t\to0}\eta_2(t)=0. \eeno
On the other hand,  thanks to \eqref{8.10}, we have
$$\displaylines{
\|\p_t[\dA\, v_1]\|_{L^{r}_t(L^{ p}_+)}\leq C\biggl( \|v_1\otimes\na
\dv\|_{L^{r}_t(L^{p}_+)}+\|\dA
\,\p_tv_1\|_{L^{r}_t(L^{ p}_+)}%\hfill\cr\hfill
+\Bigl\|\: \int_0^t|\na \dv|\,dt' \big|\na v_{1,2} \big|
|v_1|\Bigr\|_{L^{r}_t(L^{ p}_+)}\biggr)}
$$
where $v_{1,2}$ designates components of $v_1$ or $v_2.$
\medbreak
Applying H\"older and Sobolev inequalities gives \beno \begin{aligned}
\|v_1\otimes\na\dv\|_{L^{r}_t(L^{ p}_+)}\leq&
\,C\|v_1\|_{L^{2r}_t(L^{\f{dr}{r-1}}_+)}
\|\na \dv\|_{L^{2r}_t(L^{\f{dr}{2r-1}}_+)}\\
&\leq C \|\na v_1\|_{L^{2r}_t(L^{\f{dr}{2r-1}}_+)} \|\na
\dv\|_{L^{2r}_t(L^{\f{dr}{2r-1}}_+)}.\end{aligned} \eeno
Following the computations for bounding $\dg,$ we also have
 \beno
\begin{aligned}
\|\dA\,\p_tv_1\|_{L^{r}_t(L^{p}_+)}\leq &
\,C\|\dA\|_{L^\infty_t(L^{\f{dr}{2r-2}}_+)}\|\p_tv_1\|_{L^r_t(L^{d}_+)}\\
\leq &
\,Ct^{1-\f1r}\|\p_tv_1\|_{L^r_t(L^{p}_+)}^\th\|\p_tv_1\|_{L^r_t(L^{\wt
p}_+)}^{1-\th}\|\na\dv\|_{L^r_t(L^{\f{dr}{2r-2}}_+)},
\end{aligned}
\eeno and, for $\alpha$ so that $\f1\alpha+\f{r-1}{dr}=\f1{\wt
p},$ one has \beno
\begin{aligned}
\Bigl\|\int_0^t|\na \d v|\,dt'\bigl|\na v_{1,2}\bigr|
|v_1|\Bigr\|_{L^{r}_t(L^{p}_+)}
&\leq C\|\na\dv\|_{L^1_t(L^{{\f{dr}{2r-2}}}_+)}\|\na v_{1,2}\otimes v_1\|_{L^{r}_t(L^{d}_+)}\\
&\leq C t^{\f1{r'}}
\|\na\dv\|_{L^r_t(L^{{\f{dr}{2r-2}}}_+)}\bigl(\|\na
v_{1,2}\|_{L_t^{2r}(L^{\f{dr}{2r-1}}_+)}
\|v_1\|_{L^{2r}_t(L^{\f{dr}{r-1}}_+)}\bigr)^\theta\\
&\qquad\times\bigl(\|\na v_{1,2}\|_{L_t^{2r}(L^{\al}_+)}
\|v_1\|_{L^{2r}_t(L^{\f{dr}{d-1}}_+)}\bigr)^{1-\theta}.
\end{aligned}
\eeno
As a consequence, we obtain \beno \|\p_t[\dA\, v_1]\|_{L^{r}_t(L^{
p}_+)}\leq \eta_3(t)\bigl(\|\na\dv\|_{L^{2r}_t(L^{\f{dr}{2r-1}}_+)}+\|\na\dv\|_{L^r_t(L^{\f{dr}{2r-2}}_+)}\bigr)\quad\hbox{with}\quad\lim_{t\to0}\eta_3(t)=0. \eeno

\subsubsection*{Bounds for $\df_1$}

We notice that
$$
\df_1=(1+a_0)\bigl((\Id-{}^T\!A_2)\nabla\dP-{}^T\!\dA\nabla P_1\bigr).
$$
Hence, thanks to \eqref{thm1aa}, \eqref{thm1ab}  and \eqref{thm1b}, we have
 \beno
\begin{aligned}
\|\df_1\|_{L^{r}_t(L^{ p}_+)}&\leq C \bigl(\|{}^T\!(\Id-A_2)\na\dP\|_{L^{r}_t(L^{ p}_+)}+\|{}^T\!\dA\na P_1\|_{L^{r}_t(L^{ p}_+)}\bigr)\\
&\leq C \bigl(\|\na v_2\|_{L^1_t(L^\infty_+)}\|\na\dP\|_{L^{r}_t(L^{ p}_+)}+\|\na\dv\|_{L^1_t(L^{\f{dr}{2r-2}}_+)}\|\na P_1\|_{L^{r}_t(L^{d}_+)}\bigr)\\
&\leq C \Bigl(t^{\f1{q'}}\|\nabla
v_2\|_{L_t^q(L^\infty_+)}\|\na\dP\|_{L^{r}_t(L^{
p}_+)}\!+\!t^{\f1{r'}}\|\nabla P_1\|_{L_t^r(L^{ p}_+\cap L^{\wt p}_+)}
\|\na\dv\|_{L^{r}_t(L^{{\f{dr}{2r-2}}}_+)}\Bigr)
\end{aligned}
\eeno so that \beno \|\df_1\|_{L^{r}_t(L^{ p}_+)}\leq
\eta_4(t)\bigl( \|\na\dP\|_{L^{r}_t(L^{
p}_+)}+\|\na\dv\|_{L^{r}_t(L^{{\f{dr}{2r-2}}}_+)}\bigr)\quad\hbox{with}\quad\lim_{t\to0}\eta_4(t)=0. \eeno

\subsubsection*{Bounds for $\df_2$}

We may write
$$
\|\df_2\|_{L^r_t(L^{p}_+)}\leq C\Bigl(
\|\dive((A_2{}^T\!A_2-\Id)\nabla\dv)\|_{L^r_t(L^{ p}_+)}
+\|\dive\bigl((A_2{}^T\!A_2-A_1{}^T\!A_1)\nabla
v_1\bigr)\|_{L^r_t(L^{ p}_+)}\Bigr)\cdotp
$$
Therefore H\"older inequality and \eqref{8.10} imply that
$$
\displaylines{
\|\df_2\|_{L^r_t(L^{p}_+)}\leq
C\Bigl(\|\nabla(A_2{}^T\!A_2)\|_{L^\infty_t(L^{d}_+)}
\|\nabla\dv\|_{L^r_t(L^{\f{dr}{2r-2}}_+)}+\|(A_2{}^T\!A_2-\Id)\|_{L^\infty_t(L^\infty_+)}\|\nabla^2\dv\|_{L^r_t(L^{
p}_+)}\hfill\cr\hfill+\|\nabla
v_1\|_{L^r_t(L^\infty_+)}\|\nabla(A_2{}^T\!A_2\!-\!A_1{}^T\!A_1)\|_{L_t^\infty(L^{p}_+)}
+\|\nabla^2v_1\|_{L_t^r(L^{d}_+)}
\|A_2{}^T\!A_2\!-\!A_1{}^T\!A_1\|_{L_t^\infty(L^{\f{dr}{2r-2}}_+)}\Bigr),}
$$
which leads to \beno
\begin{aligned}
\|\df_2\|_{L^r_t(L^{p}_+)}\leq &\,
C\Bigl(t^{\f1{r'}}(\|\nabla^2v_1\|_{L_t^r(L^{d}_+)}+\|\na^2v_2\|_{L^r_t(L^d_+)}\bigr)+t^{\f1{q'}}\bigl(\|\na v_1\|_{L^q_t(L^\infty_+)}\\
& \hspace{2cm}+\|\na
v_2\|_{L^q_t(L^\infty_+)}\bigr)\Bigr)\bigl(\|\nabla\dv\|_{L^r_t(L^{\f{dr}{2r-2}}_+)}+\|\na^2\dv\|_{L^{r}_t(L^{p}_+)}\bigr)\\
\leq&\,\eta_5(t)\bigl(\|\nabla\dv\|_{L^r_t(L^{\f{dr}{2r-2}}_+)}+\|\na^2\dv\|_{L^{r}_t(L^{p}_+)}\bigr)\quad\hbox{with}\quad\lim_{t\to0}\eta_5(t)=0.
\end{aligned}
\eeno
Therefore, plugging all the above inequalities in \eqref{eq:etoile}, one may  conclude that
\begin{multline*}
 \|\na\dv\|_{L^{2r}_t(L^{\f{dr}{2r-1}}_+)}+\|\nabla\dv\|_{L^r_t(L^{\f{dr}{2r-2}}_+)}+\|(\p_t\dv,\nabla^2\dv,\nabla\dP)\|_{L^r_t(L^{
p}_+)}\\ \leq \eta(t)\bigl(\|\na\dv\|_{L^{2r}_t(L^{\f{dr}{2r-1}}_+)}+\|\nabla\dv\|_{L^r_t(L^{\f{dr}{2r-2}}_+)}+\|(\p_t\dv,\nabla^2\dv,\nabla\dP)\|_{L^r_t(L^{
p}_+)}\bigr)
\end{multline*}
where $\eta(t)=\sum_{i=1}^5\eta_i(t),$ and thus goes to zero as $t\to
0.$  This yields uniqueness on a small time interval. Then a
standard continuation argument yields uniqueness on the whole
interval $[0,T].$ The proof of Theorem \ref{th:main} is now complete.

%%%%%%%%%%%%%%%%%%%%%%%

\subsection{Proving  uniqueness : the  ``rough case''}

We now assume that we are given two solutions
 $(a_i,u_i,\Pi_i), i=1,2,$ satisfying the properties of  Theorem
\ref{th:main3} with $\wt p>d.$ The difference $(\dv,\nabla\dP)$
between the two solutions in Lagrangian coordinates still satisfies System \eqref{differ}.
In order to prove uniqueness, we shall derive suitable bounds for the following quantity:
$$
\dZ(T)\eqdefa \|t^\alpha(\p_t\dv,\nabla^2\dv,\nabla\dP)\|_{L^r_T(L^p_+)}
+\|t^\beta\nabla\dv\|_{L^{r_2}_T(L^{p_2}_+)}+\|t^\gamma\dv\|_{L^{r_3}_T(L^{p_3}_+)}.
$$
To start with , let us apply Proposition \ref{p:stokesbis} with regularity exponent $-1+\f dp\cdotp$ We get for all positive $T$:
\begin{equation}\label{eq:dZ}
\dZ(T)\leq
C\bigl(\|a_0\,t^{\alpha}(\Delta\dv-\nabla\dP)\|_{L^r_T(L^{
p}_+)}+\|t^\alpha(\df_1,\df_2,\nabla\dg,\p_t\dR)\|_{L^r_T(L^{p}_+)}.
\end{equation}
Again,  the first term  of the right-hand side may be absorbed if  the constant $c_1$
is small enough in \eqref{small1}.

\subsubsection*{Bounds for $\nabla\dg$}
{}From the definition of $\dg,$ we readily have
$$
\displaylines{\quad \|t^\alpha\na\dg \|_{L^{r}_T(L^{ p}_+)}\leq \|t^\alpha\na
A_2\otimes \nabla\dv\|_{L^{r}_T(L^{
p}_+)}+\|t^\alpha(\Id-A_2)\otimes\na^2\dv\|_{L^{r}_T(L^{
p}_+)}\hfill\cr\hfill +\|t^\alpha\na \dA\otimes \nabla v_1\|_{L^{r}_T(L^{
p}_+)}+\|t^\alpha\dA\otimes\na^2v_1\|_{L^{r}_T(L^{p}_+)}.\quad}
$$
Using H\"older inequality and \eqref{lipassume}, we get
$$
 \|t^\alpha\na A_2\otimes \nabla\dv\|_{L^{r}_T(L^{p}_+)}
 \leq \|t^\alpha\nabla\dv\|_{L_T^r(L^{p^*}_+)}\int_0^T\|\nabla^2v_2\|_{L^d_+}\,dt,
 $$
where $p^*\eqdefa dp/(d-p)$ stands for the Lebesgue exponent in the critical
Sobolev embedding
\begin{equation}\label{eq:critemb}
W^{1,p}(\R^d_+)\hookrightarrow L^{p^*}_+.
\end{equation}
Therefore, remembering that $\alpha+1/r<1$ and that $L^d_+\hookrightarrow L^{p}_+\cap L^{\wt p}_+,$
$$
 \|t^\alpha\na A_2\otimes \nabla\dv\|_{L^{r}_T(L^{p}_+)}
 \leq CT^{1-\f1r-\alpha}\|t^\alpha\nabla^2\dv\|_{L_T^r(L^{p}_+)}
\|t^\alpha\nabla^2v_2\|_{L^r_T(L^p_+\cap L^{\wt p}_+)}.
$$
Next, using the fact that
\begin{equation}\label{eq:Id-A2}
\|(\Id-A_2)(T)\|_{L^\infty_+}\leq C\int_0^T\|\nabla v_2\|_{L^\infty_+}\,dt
\leq CT^{1-\f1s-\delta}\|t^\delta\nabla v_2\|_{L_T^s(L^\infty_+)},
\end{equation}
where $s$ and $\delta$ are defined in \eqref{eq:s},
we easily get
$$
\|t^\alpha(\Id-A_2)\otimes\na^2\dv\|_{L^{r}_T(L^{p}_+)}\leq CT^{1-\f1s-\delta}\|t^\delta\nabla v_2\|_{L_T^s(L^\infty_+)}
\|t^\alpha\na^2\dv\|_{L^{r}_T(L^{p}_+)}.
$$
In order to bound the third term of $\nabla\dg,$ we notice from
\eqref{8.10} that
\begin{equation}\label{eq:ndA}
\|\nabla\dA\|_{L_T^\infty(L^p_+)}\leq  Ct^{1-\f1r-\alpha}\|t^\alpha\nabla^2\dv\|_{L_T^r(L^p_+)}.
\end{equation}
Therefore, setting $\f1m=\f1r-\f1s,$
$$\begin{array}{lll}
\|t^\alpha\na \dA\otimes \nabla v_1\|_{L^{r}_T(L^{p}_+)}&\leq& CT^{1-\f1r-\alpha}\|t^\alpha\nabla^2\dv\|_{L_T^r(L^p_+)}
\|\tau^{\alpha-\delta}\|_{L^m(0,T)}\|t^\delta\nabla v_1\|_{L_T^s(L^\infty_+)}\\[1ex]
&\leq&CT^{1-\f1s-\alpha}\|t^\alpha\nabla^2\dv\|_{L^r_T(L^p_+)}\|t^\delta\nabla
v_1\|_{L_T^s(L^\infty_+)}.
\end{array}$$
Finally, we have from \eqref{8.10} and Sobolev embedding \eqref{eq:critemb},
\begin{equation}\label{eq:dA}
\|\dA\|_{L^\infty_T(L^{p^*}_+)}\leq C T^{1-\f1r-\alpha}\|t^\alpha\na^2\dv\|_{L^r_T(L^{p}_+)}.
\end{equation}
Hence we obtain
$$
\|t^\alpha\dA\otimes\na^2v_1\|_{L^{r}_T(L^{p}_+)}\leq
 CT^{1-\f1r-\alpha}\|t^\alpha\na^2\d v\|_{L^r_T(L^p_+)}\|t^\alpha\na^2v_1\|_{L^r_T(L^d_+)}.
 $$
Putting together the four above estimates, we conclude that
\begin{multline}\label{eq:ndg}
\|t^\alpha\nabla\dg\|_{L^r_T(L^p_+)}\leq C\|t^\alpha\nabla^2\dv\|_{L^r_T(L^p_+)}
\Bigl(T^{1-\f1r-\alpha}\|t^\alpha(\nabla^2v_1,\nabla^2v_2)\|_{L_T^r(L^p_+\cap L^{\wt p}_+)}
\\+T^{1-\f1s-\alpha}\|t^\delta(\nabla v_1,\nabla v_2)\|_{L_T^s(L^\infty_+)}\Bigr)\cdotp
\end{multline}

 \subsubsection*{Bounds for $\p_t\dR$}
 Recall that
 $$
 \p_t\dR=-(\p_tA_2)\dv+(\Id-A_2)\p_t\dv-(\p_t\dA)\,v_1-\dA\,\p_tv_1.
 $$
 Now, using the expression of $\p_tA_2$ and H\"older inequality, we get
  $$
 \|t^\alpha\,(\p_tA_2)\dv\|_{L^{r}_T(L^{p}_+)} \leq C
\|t^\beta \na v_2\|_{L^{r_2}_T(L^{p_2}_+)}\|t^\gamma\dv\|_{L_T^{r_3}(L^{p_3}_+)}.
$$
Next, we have, using \eqref{eq:Id-A2}
$$
\begin{array}{lll}
\|t^\alpha(\Id-A_2)\p_t\dv\|_{L^{r}_T(L^{p}_+)}&\leq& C\|\Id-A_2\|_{L^\infty_T(L^\infty_+)}
\|t^\alpha\p_t\dv\|_{L^{r}_T(L^{p}_+)}\\[1ex]
 &\leq&CT^{1-\f1s-\delta}\|t^\delta\nabla v_2\|_{L_T^s(L^\infty_+)}\|t^\alpha\p_t\dv\|_{L^{r}_T(L^{p}_+)}.
\end{array}
$$
In order to bound the third term of $\p_t\dR,$
we differentiate \eqref{8.10} with respect to time and easily find that
$$
\|t^\alpha (\p_t\dA)\,v_1\|_{L^{r}_T(L^{p}_+)}
\leq C\Bigl(\|t^\alpha v_1\otimes\nabla\dv\|_{L^r_T(L^p_+)}
+\Bigl\| t^\alpha\: \int_0^t|\na \dv|\,dt' \big|\na v_{1,2} \big|
|v_1|\Bigr\|_{L^{r}_T(L^{ p}_+)} \Bigr)
$$
where $v_{1,2}$ designates components of $v_1$ or $v_2.$
\medbreak
On the one hand, applying H\"older  inequality gives
$$
\|t^\alpha \,v_1\otimes\na\dv\|_{L^{r}_T(L^{ p}_+)}\leq
C \|t^\gamma v_1\|_{L^{r_3}_T(L^{p_3}_+)} \|t^\beta\na \dv\|_{L^{r_2}_T(L^{p_2}_+)}.
$$
On the other hand, we have, by virtue of \eqref{eq:critemb}
$$\begin{array}{lll}
\!\Bigl\| t^\alpha\! \Int_0^t|\na \dv|\,dt' \big|\na v_{1,2} \big|
|v_1|\Bigr\|_{L^{r}_T(L^{ p}_+)}&\!\!\!\!\!\leq\!\!\!\!\!& C
\|t^\alpha\nabla v_{1,2}\otimes v_1\|_{L_T^r(L^d_+)}
\Int_0^T\|\nabla\dv\|_{L^{p^*}_+}\,dt\\[1ex]
&\!\!\!\!\!\leq\!\!\!\!\!& CT^{1-\f1r-\alpha}
\|t^\alpha\nabla^2\dv\|_{L_T^r(L^p_+)} \|t^\gamma
v_1\|_{L_T^{r_3}(L^{p_3}_+)} \|t^\beta\nabla
v_{1,2}\|_{L_T^{r_2}(L_+^a)}\end{array}
$$
with $\f1a=\f1d-\f1{p_2}.$ Given that $p<d<\wt p,$ we have
$p_2<a<\wt p_2,$ hence $L^a_+\hookrightarrow L_+^{p_2}\cap L_+^{\wt
p_2}.$ \medbreak Finally, arguing as for bounding the last term of
$\nabla\dg$ yields
$$
\|t^\alpha\dA\,\p_tv_1\|_{L^{r}_T(L^{p}_+)}\leq C T^{1-\f1r-\alpha}\|t^\alpha\p_tv_1\|_{L^r_T(L^{p}_+\cap L^{\wt p}_+)}
\|t^\alpha\na^2\dv\|_{L^r_T(L^{p}_+)}.
$$
Putting all those estimates together, we conclude that
\begin{equation}\label{eq:Rt}
\|\p_t\dR\|_{L^{r}_T(L^{p}_+)}\leq \eta(T) \dZ(T)
\end{equation}
for some function $\eta$ going to $0$ at $0.$

\subsubsection*{Bounds for $\df_1$}

We notice that
$$
\df_1=(1+a_0)\bigl((\Id-{}^T\!A_2)\nabla\dP-{}^T\!\dA\nabla P_1\bigr).
$$
Hence it suffices to follow the computations leading to the bounds
for the second and fourth terms of $\nabla\dg$: we just
have to change $\nabla^2\dv$ and $\nabla^2v_1$ into $\nabla\dP$ and $\nabla P_1.$
We end up with
\begin{multline}\label{eq:df1}
 \|t^\alpha\df_1\|_{L^{r}_T(L^{ p}_+)}\leq
C\Bigl(T^{1-\f1s-\delta}\|t^\delta\nabla v_2\|_{L_T^s(L^\infty_+)}\|t^\alpha\na\dP\|_{L^{r}_T(L^{
p}_+)}\\+T^{1-\f1r-\alpha}\|t^\alpha\nabla P_1\|_{L_T^r(L^p_+\cap L_+^{\wt p})}
\|t^\alpha\na^2\dv\|_{L^{r}_T(L^{p}_+)}\Bigr)\cdotp
\end{multline}

\subsubsection*{Bounds for $\df_2$}

As we may write
$$\displaylines{\quad
\|t^\alpha\df_2\|_{L^r_T(L^{p}_+)}\leq C\Bigl(
\|t^\alpha\nabla(A_2{}^T\!A_2)\otimes\nabla\dv\|_{L^r_T(L^{ p}_+)}
+\|t^\alpha(A_2{}^T\!A_2-\Id)\otimes\nabla^2\dv\|_{L^r_T(L^{ p}_+)}\hfill\cr\hfill
+\|t^\alpha\na(A_2{}^T\!A_2-A_1{}^T\!A_1)\otimes\nabla
v_1\|_{L^r_T(L^{ p}_+)}+\|(A_2{}^T\!A_2-A_1{}^T\!A_1)\nabla^2v_1\|_{L^r_T(L^{ p}_+)}\Bigr),}
$$
one may repeat the computations leading to \eqref{eq:ndg}:
this only a matter of replacing everywhere $A_1$ and $A_2$ by $A_1{}^T\!A_1$
and $A_2{}^T\!A_2,$ respectively.
We conclude that $\|t^\alpha\df_2\|_{L_T^r(L^p_+)}$ is bounded by the right-hand side of \eqref{eq:ndg}.

\subsubsection*{Conclusion}

Plugging \eqref{eq:ndg}, \eqref{eq:Rt}, \eqref{eq:df1} and the above inequality in \eqref{eq:dZ},
we conclude that whenever both $v_1$ and $v_2$ satisfy \eqref{lipassume}, we have for all $T>0$
$$
\dZ(T)\leq \eta(T)\dZ(T)
$$
for some function $\eta$ going to $0$ at $0.$
This implies uniqueness on a small time interval. Then a
standard continuation argument yields uniqueness on the whole
interval where the two solutions are defined.

%%%%%%%%%%%%%%%%%%%%%%%%%%%%%%%%%%%%%%%%%%%%%

\setcounter{equation}{0}
\section{Remarks on the bounded domain case}\label{sect5}

This  section aims at  extending partially  our main theorem to the
initial boundary value problem  \eqref{INS} in a $C^2$ bounded
domain $\Om$ of $\R^d.$ Before we present the main result, let us
introduce a few notation. We set \beno X^q(\Om)\eqdefa \Bigl\{\ u\in
\bigl(L^q(\Om)\bigr)^d\ |\ \dive u=0\ \ \mbox{and}\ \ u\cdot \vec
n=0\ \ \mbox{on}\ \ \p\Om\ \ \Bigr\}, \eeno where $\vec n$ stands
for the unit normal exterior vector at $\pa\Omega.$ Denoting by
$P_q$  the projection operator from $L^q(\Omega)$ onto
$X^q(\Omega),$ the Stokes operator on $L^q(\Omega)$ is the unbounded
operator (see e. g. \cite{Giga}) \beno
A_q\eqdefa-P_q\D\quad\mbox{with domain}\quad D(A_q)\eqdefa
W^{2,q}(\Om)\cap W^{1,q}_0(\Om)\cap X^q(\Om). \eeno
\begin{defi}\label{defi5.1}
Let $1<q<\infty.$ For $\al\in (0,1)$ and $s\in(1,\infty),$ we set
\beno \|u\|_{D^{\al,s}_{A_q}}\eqdefa
\|u\|_{L^q}+\Bigl(\int_0^\infty\|t^{1-\al}A_qe^{-tA_q}u\|_{L^q}^s\f{dt}{t}\Bigr)^{\f1s},
\eeno where $e^{-tA_q}$ stands for the analytic semigroup generated by $A_q.$ We then define the inhomogeneous fractional domain
$D^{\al,s}_{A_q}$ as the completion of $D(A_q)$ under
$\|u\|_{D^{\al,s}_{A_q}}.$
\end{defi}

\begin{rmk}\label{rmk5.1}{\sl
Let $\al\in (0,1)$ and $1<q, s<\infty.$ Let $\cB^\beta_{q,s}$ be the
completion of $C_c^\infty(\Om)$ in $B^\beta_{q,s}(\R^d).$ Then we have
\beno (\cB^{2\al}_{q,s}(\Om))^d\cap X^q(\Om)\hookrightarrow D^{\al,s}_{A_q}
\hookrightarrow ({B}^{2\al}_{q,s}(\Om))^d\cap X^q(\Om). \eeno If moreover $2\alpha<1/q$ then the three
sets coincide. }
\end{rmk}
The main result of this section reads:
\begin{thm}\label{th5.1}
{\sl Let $\Om$ be a bounded domain with $C^{2}$ boundary.
 Let $r\in (1,\infty),$ $a_0\in
L^\infty(\Om)$ and $ u_0\in D^{1-\f1r,r}_{A_{p}}$ with
$p=\f{dr}{3r-2}\cdotp$ There exists a positive constant $c_0$ so that if \beq \label{5.2}
\mu\|a_0\|_{L^\infty(\Om)}+\|u_0\|_{D^{1-\f1r,r}_{A_{p}}}\leq
c_0\mu, \eeq then \eqref{INS} has a global weak solution $(a,u,\nabla\Pi)$ in the
sense of Definition \ref{defi2.1}, which satisfies
 \beq\label{5.3}
\begin{aligned}
\mu^{1-\f1{r}}&\|u\|_{L^{\infty}(\R_+;D^{1-\f1r,r}_{A_{p}})}+
\|(\p_t u,\mu\nabla^2u,\nabla\Pi)\|_{L^r(\R_+;L^{p}(\Om))}
\\&\hspace{4cm}+\mu^{\f1{2r}}\|\nabla u\|_{L^{2r}(\R_+;L^{\f{dr}{2r-1}}(\Omega))} \leq C\mu^{1-\f1r}\|u_0\|_{D^{1-\f1r,r}_{A_{p}}}\end{aligned}
\eeq for some sufficiently large positive constant $C$.}
\end{thm}
\begin{rmk}  {\sl Uniqueness would require our using  Lagrangian coordinates, hence  investigating
 the evolutionary Stokes system  \emph{with non homogeneous divergence} in a bounded domain.
 We leave this interesting issue to a future work.}
 \end{rmk}
 \begin{rmk}{\sl
 In general bounded domains, we do not expect an anisotropic smallness condition such as \eqref{small1}
 to have any relevancy. }
\end{rmk}
The proof of the theorem mainly relies on  the following result (see   Theorem 3.2 of \cite{Dan}):
\begin{prop}\label{prop5.1}
{\sl Let $\Om$ be a $C^{2}$ bounded domain of $\R^d$ and
$1<q,s<\infty.$ Assume that $u_0\in D^{1-\f1s,s}_{A_q}$ and $f\in
L^s(\R_+; L^q(\Om)).$ Then the system
 \beno\left\{\begin{array}{lll}
\displaystyle \pa_t u -\mu\D u+\na\Pi=f&\quad\hbox{in }&\R_+\times\Om, \\
\displaystyle \dive\, u = 0&\quad\hbox{in }&\R_+\times\Om,\\
 \displaystyle  u|_{\p\Om}=0&\quad\hbox{on }&\R_+\times\p\Om,
\end{array}\right.
\eeno
with initial data $u_0$
has a unique solution $(u,\Pi)$  with
$$u\in\cC_b(\R_+;D^{1-\f1s,s}_{A_q}),\quad
\pa_tu,\nabla^2u,\nabla\Pi\in L^s(\R_+;L^q(\Omega))\ \hbox{ and }\ \int_\Omega\Pi\,dx=0.$$
Furthermore,  for all $t>0,$
$$
\displaylines{
\mu^{1-\f1s}\|u(t)\|_{D^{1-\f1s,s}_{A_q}}+\|(\p_tu,\mu\nabla^2u,\nabla\Pi)\|_{L^s_t(L^q(\Om))}\leq
C\bigl(\mu^{1-\f1s}\|u_0\|_{D^{1-\f1s,s}_{A_q}}+\|f\|_{L^s_t(L^q(\Om))}\bigr).}
$$ }
\end{prop}

Now we are in a position to prove Theorem \ref{th5.1}.

\begin{proof}[Proof of Theorem \ref{th5.1}]
We first solve System \eqref{INS} with regularized data,
according to e.g. \cite{Dan}. We  get a sequence $(a^n,u^n,\nabla\Pi^n)_{n\in\N}$ of smooth solutions to \eqref{INS}.
In particular, as
 \beno \p_tu^n-\mu\D
u^n+\na\Pi^n=F^n \eqdefa a^n(\mu\D u^n-\na\Pi^n)-u^n\cdot\na u^n,
\eeno Proposition \ref{prop5.1} implies that
$$
\begin{aligned}
Z_n(t)\eqdefa&\mu^{1-\f1r}\|u^n(t)\|_{D^{1-\f1r,r}_{A_{p}}}+\|\p_tu^n\|_{L^r_t(L^{p})}+\mu\|\nabla^2 u^n\|_{L^r_t(L^{p})}+\|\na\Pi^n\|_{L^r_t(L^{p})}\\
&\leq C\bigl(\mu^{1-\f1r}\|u_0\|_{D^{1-\f1r,r}_{A_{p}}}+\|F^n\|_{L^r_t(L^{p})}\bigr).
\end{aligned}
$$
At this stage, one may apply the following Gagliardo-Nirenberg inequality
$$
\|\nabla u^n\|_{L^{\f{dr}{2r-1}}(\Om)}\leq C\|u^n\|_{B^{2-\f2r}_{p,r}(\Om)}^{\f12}\|\nabla^2u^n\|_{L^{p}(\Om)}^{\f12}.
$$
which holds true for any smooth function compactly supported in $\Omega$ (as it coincides
with the corresponding Gagliardo-Nirenberg inequality on $\R^n$) and may thus be
extended by density to any function in $W^{1,p}_0(\Omega)\cap W^{2,p}(\Omega).$

%By interpolation and embedding  we have for all $z\in W^{1,p}(\Omega)$ the following Gagliardo-Nirenberg inequality:
%$$\|z\|_{L^{\f{dr}{2r-1}}(\Om)}\leq C\|z\|_{B^{1-\f2r}_{p,r}(\Om)}^{\f12}\|z\|_{W^{1,p}(\Om)}^{\f12}.$$
%Applying this inequality to $z=\nabla u^n$ and taking advantage of Poincar\'e inequality (here we use that $u^n$ vanishes at the boundary), we discover that

Therefore, using Remark \ref{rmk5.1} and the above bound for $Z_n(t)$ yields
\begin{equation}\label{eq:nablaun}
\|\nabla u^n\|_{L^{2r}_t(L^{\f{dr}{2r-1}}(\Om))}\leq C\bigl(\mu^{1-\f1r}\|u_0\|_{D^{1-\f1r,r}_{A_{p}}}+\|F^n\|_{L^r_t(L^{p}(\Om))}\bigr).
\end{equation}
Now,  taking advantage of H\"older inequality and of the Sobolev embedding
$$
W^{1,\f{dr}{3r-2}}_0(\Om)\hookrightarrow L^{\f{dr}{2r-1}}(\Om),
$$
we may write
\beno
\begin{aligned}
\|F^n\|_{L^r_t(L^{p}(\Om))}&\leq
\|a^n\|_{L^\infty_t(L^\infty(\Om))}\bigl(\mu\|\D
u^n\|_{L^r_t(L^{p}(\Om))}+\|\na\Pi^n\|_{L^r_t(L^{p}(\Om))}\bigr)\\
&\hspace{5cm}+\|u^n\|_{L^{2r}_t(L^{\f{dr}{r-1}}(\Om))}\|\na
u^n\|_{L^{2r}_t(L^{\f{dr}{2r-1}}(\Om))}\\
&\leq \|a_0\|_{L^\infty(\Om)}\bigl(\mu\|\D
u^n\|_{L^r_t(L^{p}(\Om))}+\|\na\Pi^n\|_{L^r_t(L^{p}(\Om))}\bigr)
+C\|\na u^n\|_{L^{2r}_t(L^{\f{dr}{2r-1}}(\Om))}^2.
\end{aligned}
\eeno Therefore taking $c_0$ small enough in \eqref{5.2}, we get by
using \eqref{eq:nablaun} that \beno Z_n(t)\leq
C\bigl(\mu^{1-\f1r}\|u_0\|_{D^{1-\f1r,r}_{A_{p}}}+\mu^{-(2-\f1r)}Z_n^2(t)\bigr),
\eeno so that as long as
$\|u_0\|_{D^{1-\f1r,r}_{A_{p}}}\leq \frac{1}{4C^2} \mu,$ we
have \beno Z_n(t)\leq 2C\mu^{1-\f1r}\|u_0\|_{D^{1-\f1r,r}_{A_{p}}}.\eeno
Granted with this estimate,
we can follow the lines of the proof of Theorem \ref{th:main} to
complete the proof of Theorem \ref{th5.1}.
\end{proof}

%%%%%%%%%%%%%%%%%%%%%%%%%%%%%%%%%%%%%%%%%%%%%%%%

\appendix\section{}

Here we establish several $L^p -  L^q$ or maximal regularity type
estimates involving the heat semigroup in the whole space, or the
Stokes semigroup in a bounded $C^2$ domain $\Omega.$ Although those
estimates belong to the mathematical folklore (as a matter of fact
the heat semigroup case in $\R^d$ has been treated in \cite{HPZ3}),
we did not find any reference where they are proved with this degree
of generality.

As in \cite{HPZ3}, the key to the proof of maximal regularity
estimates is the following proposition (see e. g. Th. 2.34 of
\cite{BCD})  enabling us to  characterize Besov spaces with negative
indices by means of the heat semigroup.
\begin{prop} \label{prop3.1}
{\sl Let $s$ be a negative real number and $(p,r)\in [1,\infty]^2.$
A constant $C$ exists such that for all $\mu>0,$ we have \beno
C^{-1}\mu^{\f{s}2}\|f\|_{\dot B^{s}_{p,r}(\R^d)}\leq
\bigl\|\|t^{-\f{s}2}e^{\mu t\D}f\|_{L^p(\R^d)}\bigr\|_{L^r(\R_+;\f{dt}{t})}\leq
C\mu^{\f{s}2}\|f\|_{\dot B^{s}_{p,r}(\R^d)}. \eeno }
\end{prop}
 We shall often use the following consequence of the above proposition:
\begin{cor}\label{c:besov}{\sl
For any $(p,r)\in[1,\infty]^2$ with $r$ finite,   there exists a constant $C$ so that for all $\mu>0,$
 $$
C^{-1}\|f\|_{\dot B^{-\f2r}_{p,r}(\R^d)}\leq
\mu^{\f1r}\|e^{\mu t\D}f\|_{L^r(\R_+;L^p(\R^d))}\leq C\|f\|_{\dot B^{-\f2r}_{p,r}(\R^d)}.
$$
Besides, for all  $T\geq0,$ we have
$$
\|e^{\mu T\Delta}f\|_{\dot B^{-\f2r}_{p,r}(\R^d)} \leq C\|f\|_{\dot B^{-\f2r}_{p,r}(\R^d)},
$$
and, if in addition $p<\infty$ then the map $T\mapsto e^{\mu T\Delta} f$ is continuous on $\R_+$ with values in $\dot B^{2-\f2r}_{p,r}(\R^d),$
whenever   $f$ is in $\dot B^{2-\f2r}_{p,r}(\R^d).$
}
\end{cor}
\begin{proof}
The first item corresponds to the previous proposition with
$s=-2/r.$ Given that $(e^{\mu t\Delta})_{t>0}$ is a contracting
semigroup over $L^p(\R^d),$ we readily get the second item. The
continuity result is a consequence of the density of smooth
functions in $\dot B^{2-\f2r}_{p,r}(\R^d)$ if both $p$ and $r$ are
finite.
\end{proof}

To prove Theorem \ref{th:main}, we also need the following result:

\begin{lem}\label{lem3.1}
{\sl The operator $\cA$ defined by
\begin{equation}\label{eq:A}
\cA: f\longmapsto \Bigl\{t\mapsto\int_0^t\nabla^2e^{\mu(t-\tau)\Delta}f\,d\tau\Bigr\}\end{equation}
 is bounded from  $L^r(0,T;L^p(\R^d))$ to $L^r(0,T; L^p(\R^d))$ for every $T\in (0,\infty]$
and $1 < p,r<\infty$. Moreover, there exists some constant $C$ so that for all positive $\mu$ and $T$ we have
 $$\displaylines{ \mu\|\cA
f\|_{L^r_T(L^p(\R^d))}\leq C\|f\|_{L^r_T(L^p(\R^d))}\cr
\mu^{1-\frac1r}\Bigl\|\int_0^Te^{\mu(T-t)\Delta}f\,dt\Bigr\|_{\dot
B^{2-\f2r}_{p,r}(\R^d)} \leq C\|f\|_{L^r_T(L^p(\R^d))}.}
$$
Furthermore, the map $T\mapsto \int_0^Te^{\mu(T-t)\Delta}f\,dt$ is continuous on $\R_+$ with
values in $\dot B^{2-\f2r}_{p,r}(\R^d).$}
\end{lem}
\begin{proof}
 The first part of the statement is Lemma 7.3 of \cite{Lem}.
 For establishing the second inequality, we just have to notice that
 $$
 \Bigl\|\int_0^Te^{\mu(T-t)\Delta}f\,dt\Bigr\|_{\dot B^{2-\f2r}_{p,r}(\R^d)}\leq C\|\cA f(T)\|_{\dot B^{-\f2r}_{p,r}(\R^d)}.
 $$
 Using that $(e^{\mu t\Delta})_{t>0}$ is contracting over $L^p(\R^d),$ and Corollary \ref{c:besov},
 we may write
 $$
 \|\cA f(T)\|_{\dot B^{-\f2r}_{p,r}(\R^d)}\leq C\mu^{\f1r}\|\cA f\|_{L^r(\R_+;L^p(\R^d))}.
 $$
 By changing $f$ to  $f1_{[0,T]},$ one gets the desired inequality. The continuity result follows by density.
\end{proof}
{}From now on, to simplify the presentation, we agree that $A_p$ denotes either the Stokes operator on $L^p(\Omega)$
with $\Omega$ a $C^2$ bounded domain (see just above Definition \ref{defi5.1}), or the heat operator on $L^p(\R^d).$
 Lemma \ref{lem3.1} extends as follows:
\begin{lem}\label{l:A}
Let $1<p,r<\infty$ and  Operator $\wt \cA_p$ be defined by
$$
\wt \cA_p: f\longmapsto \biggl[t\mapsto \int_0^tA_pe^{-(t-\tau)A_p}f(\tau)\,d\tau\biggr]\cdotp
$$
Then for all real number $\alpha\in[0,1-1/r)$  there exists a constant $C$ so that for all $T>0,$
$$
\|t^\alpha\wt\cA_p f\|_{L^r_T(L^p(\Omega))}\leq C\|t^\alpha f\|_{L^r_T(L^p(\Omega))}.
$$
\end{lem}
\begin{proof}
As regards the heat semigroup in $\R^d,$   this result has been established in \cite{HPZ3}.
 We here propose another proof that also works for the Stokes semigroup in bounded domains
 (and, more generally, whenever maximal regularity estimates are available).
Let  \begin{equation}\label{eq:duhamel}
v(t)\eqdefa \Int_0^te^{-(t-\tau)A_p}f(\tau)\,d\tau.\end{equation}
Because
$$
\p_t(t^\alpha v)+A_p(t^\alpha v)=t^\alpha f+\alpha t^{\alpha-1}v\quad\hbox{and}\quad (t^\alpha v)|_{t=0}=0,
$$
we readily have
\begin{equation}\label{eq:AA}
\|t^\alpha \wt \cA_pv\|_{L^r_T(L^p(\Omega))}\leq C\bigl(\|t^\alpha f\|_{L^r_T(L^p(\Omega))}+\alpha\|t^{\alpha-1}v\|_{L^r_T(L^p(\Omega))}\bigr).
\end{equation}
{}From the definition of $v$ and the fact that $(e^{-\lambda A_p})_{\lambda>0}$ is contracting on $L^p(\Omega),$ we infer that
$$
\|v(t)\|_{L^p(\Omega)}\leq \int_0^t\|f(\tau)\|_{L^p(\Om)}\,d\tau.
$$
Therefore,
$$
\|t^{\alpha-1} v(t)\|_{L^p(\Omega)}\leq \int_0^t \biggl(\frac t\tau\biggr)^\alpha F(\tau)\,\frac{d\tau}t
\quad\hbox{with}\quad F(\tau)\eqdefa \|\tau^\alpha f(\tau)\|_{L^p(\Omega)},
$$
that is to say,
$$
\|t^{\alpha-1} v(t)\|_{L^p(\Omega)}\leq \int_0^1(\tau')^{-\alpha} F(t\tau')\,d\tau'.
$$
Hence taking the norm in $L^r(0,T),$ and using Minkowski inequality and $\alpha+1/r<1,$
$$\begin{array}{lll}
\|t^{\alpha-1}v\|_{L^r_T(L^p(\Omega))}&\leq& \Int_0^1(\tau')^{-\alpha}\biggl(\int_0^T F^r(t\tau')\,dt\biggr)^{\frac1r}\,d\tau'\\[1ex]
&\leq&\Int_0^1(\tau')^{-\alpha-1/r}\biggl(\Int_0^{\tau' T} F^r(t')\,dt'\biggr)^{\frac1r}\,d\tau'\\[2ex]
&\leq& C\|F\|_{L^r(0,T)}.
\end{array}
$$
Plugging this inequality in \eqref{eq:AA} completes the proof.
\end{proof}
\begin{rmk}{\sl
Applying the above result in the Stokes case implies that the function $v$
defined in \eqref{eq:duhamel} and the corresponding gradient term $\nabla \Pi$ satisfy
\begin{equation}\label{eq:cA}\|t^\alpha(\p_tv,\nabla^2v,\nabla\Pi)\|_{L^r_T(L^p(\Omega))}\leq C\|t^\alpha f\|_{L^r_T(L^p(\Omega))}.\end{equation}}
\end{rmk}

The next lemma provides estimates on the gradient.
\begin{lem}\label{l:B} Let $1<p,q,r<\infty$ with $p\leq q.$
Let Operator $\cB$ be defined by
$$
\cB: f\longmapsto \biggl[t\mapsto \int_0^t\nabla e^{-(t-\tau)A_p}f(\tau)\,d\tau\biggr]\cdotp
$$
\begin{enumerate}
\item   If in addition $\frac dp-\frac dq\leq1$ then for all real numbers $\alpha$ and $\beta$ satisfying
\begin{equation}\label{eq:B1rel}
\beta=\alpha+\frac d2\biggl(\frac1p-\frac1q\biggr)-\frac12\quad\hbox{and}\quad \alpha r'<1
\end{equation}
 there exists a constant $C$ so that for all $T>0,$
\begin{equation}\label{eq:B1}
\|t^\beta\cB f\|_{L^r_T(L^q(\Omega))}\leq C\|t^\alpha f\|_{L^r_T(L^p(\Omega))}.
\end{equation}
\item
   If in addition $\frac dp-\frac dq < 1-\frac2{r}$ then for all real numbers $\alpha$ and $\beta$ satisfying
\begin{equation}\label{eq:B2rel}
\beta=\alpha+\frac d2\biggl(\frac1p-\frac1q\biggr)-\frac12+\frac1r\quad\hbox{and}\quad \alpha r'<1
\end{equation}
 there exists a constant $C$ so that for all $T>0,$
\begin{equation}\label{eq:B2}
\|t^\beta\cB f\|_{L^\infty_T(L^q(\Omega))}\leq C\|t^\alpha f\|_{L^r_T(L^p(\Omega))}.
\end{equation}
\item More generally, for any $s\in[r,\infty],$  if  $\frac dp-\frac dq < 1-\frac2{r}+\frac2{s}$
and $(\alpha,\beta)$ satisfy
\begin{equation}\label{eq:B3rel}
\beta=\alpha+\frac d2\biggl(\frac1p-\frac1q\biggr)-\frac12+\frac1r-\frac1s\quad\hbox{and}\quad \alpha r'<1
\end{equation}
 then there exists a constant $C$ so that for all $T>0,$
\begin{equation}\label{eq:B3}
\|t^\beta\cB f\|_{L^s_T(L^q(\Omega))}\leq C\|t^\alpha f\|_{L^r_T(L^p(\Omega))}.
\end{equation}
\end{enumerate}
\end{lem}
\begin{proof}
The limit case $\beta=\alpha$ of the first inequality is a consequence
of Lemma \ref{l:A} and Sobolev embedding.
To treat the case $\beta<\alpha,$ the starting point is the following inequality
\begin{equation}\label{eq:lplqgrad}
\|\nabla e^{-t A_p}f\|_{L^q(\Omega)}\leq Ct^{-\delta}\|f\|_{L^p(\Omega)}
\quad\hbox{with}\quad \delta\eqdefa\frac d2\biggl(\frac1p-\frac1q\biggr)+\frac12
\end{equation}
which holds true whenever $1<p\leq q<\infty.$ It has been proved in \cite{Giga} for the Stokes operator in bounded domains,
and follows from an explicit computation for the heat operator in $\R^d.$
\medbreak
This inequality obviously implies that
\begin{equation}\label{eq:B4}
t^\beta\|\cB f(t)\|_{L^q(\Omega)}\leq Ct^\beta\int_0^t (t-\tau)^{-\delta}\tau^{-\alpha}
F(\tau)\,d\tau\quad\hbox{with }\
F(\tau)\eqdefa\|\tau^\alpha f(\tau)\|_{L^p(\Omega)}.
\end{equation}
Therefore, making a change of variables, and using the relationship between $\alpha$ and $\beta,$
$$
t^\beta\|\cB f(t)\|_{L^q(\Omega)}\leq C\int_0^1(1-\tau')^{-\delta}(\tau')^{-\alpha}
 F(\tau' t)\,d\tau'.
 $$
 Now, taking the $L^r_T$ norm of both sides and applying Minkowski inequality implies that
 $$
 \|t^\beta\cB f\|_{L^r_T(L^q(\Omega))} \leq C\int_0^1 (1-\tau')^{-\delta}(\tau')^{-\alpha}
 \biggl(\int_0^T F^r(\tau't)\,dt\biggr)^{\frac1r}\,d\tau'.
 $$
 Then arguing exactly as in the proof of the previous lemma, we get Inequality \eqref{eq:B1} whenever  \eqref{eq:B1rel}
 is satisfied.
 \medbreak
 In order to establish \eqref{eq:B2},  we start from \eqref{eq:B4} and apply H\"older inequality.
 We obtain for all $t\in[0,T],$
 $$
 t^\beta\|\cB f(t)\|_{L^q(\Omega)}\leq Ct^\beta \biggl(\int_0^t(t-\tau)^{-r'\delta}\tau^{-\alpha r'}\,d\tau\biggr)^{\frac1{r'}}
 \|F\|_{L^r(0,T)}.
 $$
 Then  making the usual change of variable and taking advantage of the relationship between $\alpha$ and $\beta,$ and of the
 definition of $F,$ we get for all $t\in[0,T],$
 $$
 t^\beta\|\cB f(t)\|_{L^q(\Omega)}\leq C \biggl(\int_0^t(1-\tau')^{-r'\delta}(\tau')^{-\alpha r'}\biggr)^{\frac1{r'}}
 \|t^\alpha f\|_{L^r_T(L^p(\Omega))}.
 $$
 It is now clear that we get Inequality \eqref{eq:B2} under the constraints of \eqref{eq:B2rel}
 and $\delta r'<1.$
 \medbreak
 In order to treat the general case, we have to combine the methods for proving \eqref{eq:B1} and \eqref{eq:B2}.
 Starting from \eqref{eq:B4}, we write
 $$
I\eqdefa \|t^\beta\cB f(t)\|_{L^s_T(L^q(\Omega))}\leq C\biggl(\int_0^Tt^{\beta s}\biggl(\int_0^t(t-\tau)^{-\delta\varphi}
 \:(t-\tau)^{-\delta(1-\varphi)} F(\tau)\tau^{-\alpha}\,d\tau\biggr)^sdt\biggr)^{\frac1s}
 $$
 where $\varphi$ is a parameter in $[0,1],$ to be fixed hereafter.
 \medbreak
 Applying H\"older inequality to the inner integral, we get (with obvious notation)
 $$
 I^s\leq C\int_0^Tt^{\beta s}\biggl(\int_0^t(t-\tau)^{-\delta\varphi\frac sr}F(\tau)\tau^{-\alpha}\,d\tau\biggr)^r
 \biggl(\int_0^t(t-\tau)^{-\delta(1-\varphi)(\frac sr)'} F(\tau)\tau^{-\alpha}\,d\tau\biggr)^{s-r}dt\cdotp
 $$
 Applying again H\"older inequality, in the last integral only, we thus find that the above r.h.s. is bounded by
  $$
 \int_0^T\!t^{\beta s}\biggl(\int_0^t(t-\tau)^{-\delta\varphi\frac sr}F(\tau)\tau^{-\alpha}\,d\tau\biggr)^r
 \biggl(\int_0^t\! F^r(\tau)\,d\tau\biggr)^{\frac sr-1}
 \!\biggl(\int_0^t(t-\tau)^{-\delta(1-\varphi)(\frac sr)'r'} \tau^{-\alpha r'}\,d\tau\biggr)^{\frac{s-r}{r'}}dt\cdotp
 $$
 Then we perform the same change of variables as above to get
 $$
 I^s\leq C\|F\|_{L^r(0,T)}^{s-r}\int_0^T\biggl(\int_0^1(1-\tau)^{-\delta\varphi\frac sr}F(t\tau)\tau^{-\alpha}\,d\tau\biggr)^r
 \biggl(\int_0^1(1-\tau)^{-\delta(1-\varphi)(\frac sr)'r'} \tau^{-\alpha r'}\,d\tau\biggr)^{\frac{s-r}{r'}}dt
 $$
  where we have used the fact that $\beta=\alpha+\delta-1/r'-1/s.$
  If we assume that $\alpha r'<1$ and that $\delta(1-\varphi)(\frac sr)'r'<1$ then
  the last integral is bounded.
  Then applying Minkowski inequality to swap the integrals on $[0,T]$ and on $[0,1],$ we eventually get
  $$
  I\leq C\|F\|_{L^r(0,T)}^{1-\frac rs}
  \biggl(\int_0^1(1-\tau)^{-\delta\varphi\frac sr}\tau^{-\alpha}\biggl(\int_0^T F^r(t\tau)\,dt\biggr)^{\frac1r}\,d\tau\biggr)^{\frac rs},
 $$
 whence
 $$
 \|t^\beta\cB f(t)\|_{L^s_T(L^q(\Omega))}\leq C\|F\|_{L^r(0,T)}^{1-\frac rs}
   \biggl(\int_0^1(1-\tau)^{-\delta\varphi\frac sr}\tau^{-\alpha-\frac1r}
   \biggl(\int_0^{\tau T} F^r(t')\,dt'\biggr)^{\frac1r}\,d\tau\biggr)^{\frac rs},
 $$
 which implies Inequality \eqref{eq:B3} provided $\alpha+1/r<1$ and $\delta\varphi s/r<1.$
 In order to complete the proof, it is only a matter of
 taking $\varphi=\frac{r^2}{r+sr-s}$ so that the conditions  $\delta(1-\varphi)(\frac sr)'r'<1$
 and  $\delta\varphi s/r<1$ are equivalent.
 \end{proof}

Finally we need a lemma involving the following operator
$$
\cC: f\longmapsto \biggl[t\mapsto \int_0^te^{-(t-\tau)A_p}f(\tau)\,d\tau\biggr]\cdotp
$$
\begin{lem}\label{l:C} Let $1<p,q,r<\infty$ with $q\geq p.$
\begin{enumerate}
\item    If in addition $\frac dp-\frac dq\leq2$ then for all real numbers $\alpha$ and $\gamma$ satisfying
\begin{equation}\label{eq:C1rel}
\gamma=\alpha+\frac d2\biggl(\frac1p-\frac1q\biggr)-1\quad\hbox{and}\quad \alpha r'<1
\end{equation}
 there exists a constant $C$ so that for all $T>0,$
\begin{equation}\label{eq:C1}
\|t^\gamma\cC f\|_{L^r_T(L^q(\Omega))}\leq C\|t^\alpha f\|_{L^r_T(L^p(\Omega))}.
\end{equation}
\item If in addition $\frac dp-\frac dq<2-\frac 2{r}$ then
 for all real numbers $\alpha$ and $\gamma$ satisfying
\begin{equation}\label{eq:C2rel}
\gamma=\alpha+\frac d2\biggl(\frac1p-\frac1q\biggr)-\frac1{r'}\quad\hbox{and}\quad \alpha r'<1
\end{equation}
 there exists a constant $C$ so that for all $T>0,$
\begin{equation}\label{eq:C2}
\|t^\gamma\cC f\|_{L^\infty_T(L^q(\Omega))}\leq C\|t^\alpha f\|_{L^r_T(L^p(\Omega))}.
\end{equation}
\item More generally, for any $s\in[r,\infty],$  if  $\frac dp-\frac dq<2-\frac2{r}+\frac2{s}$
and $(\alpha,\gamma)$ satisfy
\begin{equation}\label{eq:C3rel}
\gamma=\alpha+\frac d2\biggl(\frac1p-\frac1q\biggr)-1+\frac1r-\frac1s\quad\hbox{and}\quad \alpha r'<1
\end{equation}
then  there exists a constant $C$ so that for all $T>0,$
\begin{equation}\label{eq:C3}
\|t^\gamma\cC f\|_{L^s_T(L^q(\Omega))}\leq C\|t^\alpha f\|_{L^r_T(L^p(\Omega))}.
\end{equation}
\end{enumerate}
\end{lem}
\begin{proof}
The limit case $\gamma=\alpha$ of the first inequality is a consequence
of Lemma \ref{l:B} and Sobolev embedding.
To treat the case $\gamma<\alpha,$ the starting point is the following inequality
\begin{equation}\label{eq:lplq}
\|e^{-tA_p}f\|_{L^q(\Omega)}\leq Ct^{-\delta}\|f\|_{L^p(\Omega)}\quad\hbox{with}\quad
\delta\eqdefa \frac d2\biggl(\frac1p-\frac1q\biggr)
\end{equation}
whenever $1<p\leq q<\infty,$
which has been proved in \cite{Giga} for the Stokes operator in bounded domains,
and follows from an explicit computation for the heat operator in $\R^d.$
\medbreak
This inequality obviously implies that
\begin{equation}\label{eq:C4}
t^\gamma\|\cC f(t)\|_{L^q(\Omega)}\leq Ct^\gamma\int_0^t (t-\tau)^{-\delta}\tau^{-\alpha}
F(\tau)\,d\tau\quad\hbox{with }\ F(\tau)\eqdefa\|\tau^\alpha f(\tau)\|_{L^p(\Omega)}.
\end{equation}
Therefore, making a change of variables, and using the relationship between $\alpha$ and $\gamma,$
$$
t^\gamma\|\cB f(t)\|_{L^q(\Omega)}\leq C\int_0^1(1-\tau')^{-\delta}(\tau')^{-\alpha}F(\tau't)\,d\tau'.
 $$
 Now, taking the $L^r_T$ norm of both sides and applying Minkowski inequality as in the above lemmas
 yields  Inequality \eqref{eq:C1}.
 \medbreak
 In order to establish \eqref{eq:C2},  we start from \eqref{eq:C4} and apply  H\"older inequality.
 We obtain for all $t\in[0,T],$
 $$
 t^\gamma\|\cB f(t)\|_{L^q(\Omega)}\leq Ct^\gamma \biggl(\int_0^t(t-\tau)^{-r'\delta}\tau^{-\alpha r'}\,d\tau\biggr)^{\frac1{r'}}
 \|F\|_{L^r(0,T)}.
 $$
 Then arguing as in the previous lemma easily leads to
 \eqref{eq:C2} under the constraints of \eqref{eq:C2rel}.
The general case $s\geq r$ follows from similar arguments. The details are left to the reader.
\end{proof}

We finally  want to establish  decay estimates for the free solution
to heat equation in the whole space.
\begin{lem}\label{l:D}
Assume that $u_0\in\dot B^{s}_{p,r}(\R^d)$ with $1\leq p,r\leq\infty.$ The following
inequalities hold true:
\begin{enumerate}
\item If  $s<2$ then
\begin{equation}\label{eq:D1}
 \|t^\alpha\nabla^2 e^{t\Delta} u_0\|_{L^r(\R_+;L^{p}(\R^d))}\leq C\|u_0\|_{\dot B^{s}_{p,r}(\R^d)}
 \quad\hbox{with }\ \alpha\eqdefa 1-\f s2-\f1r\cdotp
 \end{equation}
 \item If  $s<1$  then
\begin{equation}\label{eq:D2}
 \|t^\beta\nabla e^{t\Delta} u_0\|_{L^r(\R_+;L^{p}(\R^d))}\leq C\|u_0\|_{\dot B^{s}_{p,r}(\R^d)}
  \quad\hbox{with }\ \beta\eqdefa \f12-\f s2-\f1r\cdotp
 \end{equation}\item If  $s<0$  then
\begin{equation}\label{eq:D3}
 \|t^\gamma e^{t\Delta} u_0\|_{L^r(\R_+;L^{p}(\R^d))}\leq C\|u_0\|_{\dot B^{s}_{p,r}(\R^d)}
  \quad\hbox{with }\ \gamma\eqdefa -\f s2-\f1r\cdotp
 \end{equation}
\end{enumerate}
\end{lem}
\begin{proof}
The assumption ensures that $\nabla^2 u_0\in \dot B^{s-2}_{p,r}(\R^d).$
Because  $s-2<0,$ Proposition \ref{prop3.1} yields
$$
  \|t^{1-\frac s2-\frac1r} e^{t\Delta}\nabla^2 u_0\|_{L^r(\R_+;L^p(\R^d))}
  =   \|t^{1-\frac s2} \|e^{t\Delta}\nabla^2 u_0\|_{L^p(\R^d)}\|_{L^r(\R_+;\frac{dt}t)}
\approx   \|\nabla^2 u_0\|_{\dot B^{s-2}_{p,r}(\R^d)}.
$$
The proof of the other inequalities is totally similar.
\end{proof}
\begin{rmk} {\sl In this paper, we mainly consider the case where $s=-1+\f dp\cdotp$
The corresponding values of $(\alpha,\beta,\gamma)$ are
$$\alpha= \frac32-\frac d{2p}-\frac1r \ \hbox{ if }\  p>\f d3,\quad\
\beta= 1-\frac d{2p}-\frac1r\ \hbox{ if }\ p>\f d2,\quad\
\gamma= \frac12-\frac d{2p}-\frac1r\ \hbox{ if }\ p>d.$$}
\end{rmk}

\begin{prop}\label{p:stokesconv}
Let  $u_0$ be in $\dot\cB^{-1+\f dp}_{p,r}(\R^d_+)\cap \dot\cB^{-1+\f dp}_{\wt p,r}(\R^d_+)$
 and  $f\in L^r(\R_+;L^p_+\cap L^{\wt p}_+)$ with $p=\f{dr}{3r-2},$
$r\in(1,\infty)$ and $d<\wt p\leq\f{dr}{r-1}\cdotp$
Let the function $v^d$ be in $L^{2r}(\R_+;L^{\f{dr}{2r-1}}_+\cap L^\alpha_+)$ with $\f1\alpha\eqdefa\f1{\wt p}-\f{r-1}{dr}\cdotp$

 Then  the following system:
\begin{equation}\label{eq:stokesconv}
 \quad\left\{\begin{array}{l}
\displaystyle \pa_t u^{h}+v^{d}\partial_du^{h}-\mu\D u^{h}
+\na_h\Pi=f^{h}\\
\displaystyle \pa_t u^{d}+u^{h}\cdot\nabla_hv^{d}-v^{d}\dive_hu^{h}-\mu\D u^{d}
+\partial_d\Pi=f^{d}\\
\displaystyle \dive\, u = 0, \\
\displaystyle u|_{\partial\R^d_+}=0,\\
 \displaystyle u|_{t=0}=u_{0},
\end{array}\right.
\end{equation}
admits a unique solution $(u,\nabla\Pi)$ in $X^{p,r}\cap X^{\wt p,r}.$
Furthermore, there exist two positive constants $\la_0$ and $C$ such that the following inequalities hold true:
\begin{eqnarray}\label{eq:aux1}
&&\|\wt u_\la^{h}\|_{\mathfrak{X}^{p,r}_t}+\mu^{1-\f1{2r}}\|\na\wt u_\la^{h}\|_{L^{2r}_t(L^{\f{dr}{2r-1}}_+)}\leq
 C\bigl(\mu^{1-\f1r}\|u_0^h\|_{\dot
B^{-1+\f{d}p}_{p,r}(\R^d_+)}+\|\wt f_\la\|_{L^{r}_t(L^{p}_+)}\bigr),\\\label{eq:aux2}
&&\|\wt u^{h}_\la\|_{\mathfrak{X}^{\wt p,r}_t}+\mu^{1-\f1{2r}}\|\na\wt u_\la^{h}\|_{L^{2r}_t(L^{\alpha}_+)}
\leq
 C\Bigl(\mu^{1-\f1r}\|u_0^h\|_{\dot
B^{-1+\f{d}p}_{\wt p,r}(\R^d_+)}+\|\wt f_\la\|_{L^{r}_t(L^{\wt p}_+)}\\&&\hspace{7cm}+\|\nabla v^d\|_{L_t^{2r}(L^\alpha_+)}\|\nabla\wt u^h_\la\|_{L^{2r}_t(L^{\f{dr}{2r-1}}_+)}\Bigr),\nonumber
%\bigl(\mu^{-\f1{2r}}\|u_0^h\|_{\dot B^{-1+\f{d}p}_{p,r}(\R^d_+)}+\mu^{\f1{2r}-1}\|\wt f\|_{L^{r}_t(L^{p}_+)}\bigr)\|\nabla v^d\|_{L_t^{2r}(L^\alpha_+)}\Bigr),\nonumber
%\\\label{eq:aux3}&&\|(\wt u,\nabla\wt \Pi)\|_{X^{p,r}_t}+\mu^{1-\f1{2r}}\|\na\wt u\|_{L^{2r}_t(L^{\f{dr}{2r-1}}_+)}\leq C\bigl(\mu^{1-\f1r}\|u_0\|_{\dot B^{-1+\f{d}p}_{p,r}(\R^d_+)}+\|\wt f\|_{L^{r}_t(L^{p}_+)}\bigr),
%\\\label{eq:aux4}&&\|(\wt u,\nabla\wt \Pi)\|_{X^{\wt p,r}_t}+\mu^{1-\f1{2r}}\|\na\wt u\|_{L^{2r}_t(L^{\alpha}_+)}\leq C\Bigl(\mu^{1-\f1r}\|u_0\|_{\dot B^{-1+\f{d}p}_{\wt p,r}(\R^d_+)}+\|\wt f\|_{L^{r}_t(L^{\wt p}_+)}\\&&\hspace{4cm}+
%\bigl(\mu^{-\f1{2r}}\|u_0^h\|_{\dot B^{-1+\f{d}p}_{p,r}(\R^d_+)}+\mu^{\f1{2r}-1}\|\wt f\|_{L^{r}_t(L^{p}_+)}\bigr)\|\nabla v^d\|_{L_t^{2r}(L^\alpha_+)}\Bigr),\nonumber
\end{eqnarray}
with $\,\displaystyle (\wt u_\lambda,\nabla\wt\Pi_\lambda,\wt f_\lambda)(t)\eqdefa\exp\Bigl(-\la\mu^{1-2r}\int_0^t\|\nabla v^d\|_{L^{\f{dr}{2r-1}}}^{2r}\,d\tau\Bigr)
(u,\nabla\Pi,f)(t)$ and $\lambda\geq \la_0.$
\end{prop}
\begin{proof}
We introduce the map $\Psi:w\mapsto u$ where $w$ stands for the solution to
\begin{equation}\label{eq:stokesconvbis}
 \quad\left\{\begin{array}{l}
\displaystyle \pa_t u-\mu\D u +\na\Pi=g,\\
\displaystyle \dive\, u = 0, \\
\displaystyle u|_{\partial\R^d_+}=0,\\
 \displaystyle u|_{t=0}=u_{0},
\end{array}\right.
\end{equation}
with $g\eqdefa (g^h,g^d)=(f^h-v^{d}\partial_dw^{h},f^d-w^{h}\cdot\nabla_hv^{d}+v^{d}\dive_hw^{h}).$
\medbreak
We claim that $\Psi$  admits a unique fixed point in $X^{p,r}\cap X^{\wt p,r}.$ To prove it,
we fix some parameter $\la>0$ and set
$$
h_\la(s,t)\eqdefa\exp\biggl(-\la\mu^{1-2r}\int_s^t  \|\na v^{d}(\tau)\|_{L^{\f{dr}{2r-1}}_+\cap L^\alpha_+}^{2r}\,d\tau\biggr)\cdotp
$$
 Then it follows from \eqref{2.5}  that
\begin{equation}\label{eq:aux5a}
\wt u_\la^h=h_\la(0,t)u_L^h-SU\!\int_0^th_\la(\tau,t)e^{\mu(t-\tau)\D}e_a(\wt{N}\wt g_\la)\,d\tau+r\!\int_0^th_\la(\tau,t)e^{\mu(t-\tau)\D}e_a(\wt{M}\wt g_\la)\,d\tau
\end{equation}
where $u_L$ stands for the free solution of the Stokes system with initial data $u_0.$
Of course, formulae similar to \eqref{2.5} and \eqref{2.6} (involving $h_\la$) may be derived for $\wt u^d_\la$ and $\nabla\wt\Pi_\la.$
Hence, arguing exactly as in the proof of  Proposition \ref{p:stokes}, we  get
$$\displaylines{\quad
\|(\wt u_\la,\nabla\wt\Pi_\la)\|_{{X}^{p,r}_t}+\mu^{1-\f1{2r}}\|\na
\wt u_\la\|_{L^{2r}_t(L^{\f{dr}{2r-1}}_+)}\leq C\Bigl(\mu^{1-\f1r}\|u_0\|_{\dot B^{-1+\f{d}p}_{p,r}(\R^d_+)}\hfill\cr\hfill
+\Bigl(\int_0^t(h_\la(\tau,t))^r\|\wt g_\la(\tau)\|_{L^p_+}^r\,d\tau\Bigr)^{\f1r}\Bigr)\cdotp}
$$
{}From the expression of $g,$  Sobolev embedding \eqref{eq:sobemb}
and H\"older inequality, we infer  that
$$\displaylines{
\quad\Bigl(\int_0^t(h_\la(\tau,t))^r\|\wt g_\la(\tau)\|_{L^p_+}^r\,d\tau\Bigr)^{\f1r}\leq
\|\wt f_\lambda\|_{L^{r}_t(L^{p}_+)}\hfill\cr\hfill+\Bigl(\int_0^t(h_\la(\tau,t))^r\|\nabla v^{d}(\tau)\|_{L^{\f{dr}{2r-1}}_+}^r\|\nabla\wt w^{h}_\la(\tau)\|_{L^{\f{dr}{2r-1}}_+}^{r}\,d\tau\Bigr)^{\f1r}.\quad}
$$
Now, as a consequence of Cauchy-Schwarz inequality, we see that
\begin{equation}\label{eq:aux5}
\biggl(\int_0^t(h_\la(\tau,t))^r\|\na
v^{d}(\tau)\|_{L^{\f{dr}{2r-1}}_+}^r\|\nabla\wt w^{h}_\la(\tau)\|_{L^{\f{dr}{2r-1}}_+}^{r}\,d\tau\biggr)^{\f1r}\leq
\f{\mu^{1-\f1{2r}}}{(2\la r)^{\f1{2r}}}\|\nabla\wt w^{h}_\la\|_{L^{2r}_t(L^{\f{dr}{2r-1}}_+)},
 \end{equation}
whence
$$ \displaylines{\quad
\|(\wt u_\la,\nabla\wt \Pi_\la)\|_{{X}^{p,r}_t}+\mu^{1-\f1{2r}}\|\na\wt u_\la\|_{L^{2r}_t(L^{\f{dr}{2r-1}}_+)}\leq
 C\biggl(\mu^{1-\f1r}\|u_0\|_{\dot
B^{-1+\f{d}p}_{p,r}(\R^d_+)}+\|\wt f_\lambda\|_{L^{r}_t(L^{p}_+)}\hfill\cr\hfill+\f{\mu^{1-\f1{2r}}}{(2\la r)^{\f1{2r}}}\|\nabla\wt w^{h}_\la\|_{L^{2r}_t(L^{\f{dr}{2r-1}}_+)}\biggr)\cdotp\quad} $$
Choosing $\lambda$ so that
$4C\leq(2\la r)^{\frac1{2r}},$ we thus obtain
\begin{multline}\label{eq:aux6}
\|(\wt u_\la,\nabla\wt\Pi_\la)\|_{{X}^{p,r}_t}+\mu^{1-\f1{2r}}\|\na\wt u_\la\|_{L^{2r}_t(L^{\f{dr}{2r-1}}_+)}\cr\leq
 C\bigl(\mu^{1-\f1r}\|u_0\|_{\dot
B^{-1+\f{d}p}_{p,r}(\R^d_+)}+\|\wt f_\lambda\|_{L^{r}_t(L^{p}_+)}\bigr)+\frac14\mu^{1-\frac1{2r}} \|\nabla\wt w^{h}_\la\|_{L^{2r}_t(L^{\f{dr}{2r-1}}_+)}.
\end{multline}
Exactly along the same lines, we have
$$\displaylines{\quad
\|(\wt u_\la,\nabla\wt\Pi_\la)\|_{{X}^{\wt p,r}_t}+\mu^{1-\f1{2r}}\|\na
\wt u_\la\|_{L^{2r}_t(L^{\alpha}_+)}\leq C\Bigl(\mu^{1-\f1r}\|u_0\|_{\dot B^{-1+\f{d}p}_{\wt p,r}(\R^d_+)}\hfill\cr\hfill
+\Bigl(\int_0^t(h_\la(\tau,t))^r\|\wt g_\la(\tau)\|_{L^{\wt p}_+}^r\,d\tau\Bigr)^{\f1r}\Bigr),\quad}
$$
and combining Sobolev  embedding \eqref{eq:sobemb}
and H\"older inequality yields
$$
\|\wt g_\la\|_{L^{\wt p}_+}\leq C\bigl(\|\nabla\wt w^h_\la\|_{L^{\f{dr}{2r-1}}_+}\|\nabla v^d\|_{L^\alpha_+}
+\|\nabla v^d\|_{L^{\f{dr}{2r-1}}_+}\|\nabla\wt w^h_\la\|_{L^\alpha_+}\bigr).
$$
Hence, using an obvious variation on  \eqref{eq:aux5}, we end up with
\begin{multline}\label{eq:aux7}
\|(\wt u_\la,\nabla\wt \Pi_\la)\|_{{X}^{\wt p,r}_t}+\mu^{1-\f1{2r}}\|\na\wt u_\la\|_{L^{2r}_t(L^{\alpha}_+)}\cr\leq
 C\bigl(\mu^{1-\f1r}\|u_0\|_{\dot
B^{-1+\f{d}p}_{\wt p,r}(\R^d_+)}+\|\wt f_\lambda\|_{L^{r}_t(L^{\wt p}_+)}\bigr)
+\frac14\mu^{1-\frac1{2r}} \|\nabla\wt w^{h}_\la\|_{L^{2r}_t(L^{\f{dr}{2r-1}}_+\cap L^{\alpha}_+)}.
\end{multline}
Adding up Inequalities \eqref{eq:aux6} and \eqref{eq:aux7}, we conclude that for all $t>0,$
$$\displaylines{
\|(\wt u_\la,\nabla\wt \Pi_\la)\|_{X^{p,r}_t\cap {X}^{\wt p,r}_t}+\mu^{1-\f1{2r}}\|\na\wt u_\la\|_{L^{2r}_t(L^{\f{dr}{2r-1}}_+\cap L^{\alpha}_+)}
\hfill\cr\hfill\leq
 C\bigl(\mu^{1-\f1r}\|u_0\|_{\dot
B^{-1+\f{d}p}_{p,r}(\R^d_+)\cap\dot B^{-1+\f{d}p}_{\wt p,r}(\R^d_+) }+\|\wt f_\lambda\|_{L^{r}_t(L^p_+\cap L^{\wt p}_+)}\bigr)
+\frac12\mu^{1-\frac1{2r}} \|\nabla\wt w^{h}_\la\|_{L^{2r}_t(L^{\f{dr}{2r-1}}_+\cap L^{\alpha}_+)}.}
$$
Therefore  the linear map $\Psi$ is a contraction on the   Banach space $X^{p,r}\cap X^{\wt p,r}$ endowed with  the norm
$$
\sup_{t>0}\bigl( \|(\wt u_\la,\nabla\wt\Pi_\la)\|_{X^{p,r}_t\cap {X}^{\wt p,r}_t}+\mu^{1-\f1{2r}}\|\na\wt u_\la\|_{L^{2r}_t(L^{\f{dr}{2r-1}}_+\cap L^{\alpha}_+)}\bigr).
$$
The contracting mapping theorem thus ensures the existence and uniqueness of a solution  $(u,\nabla\Pi)$ in $X^{p,r}\cap X^{\wt p,r}$
for System \eqref{eq:stokesconv}.
\medbreak
In order to establish Inequalities  \eqref{eq:aux1} and \eqref{eq:aux2}, we  just have to modify the definition of $h_\la$ as follows:
$$h_\la(s,t)\eqdefa\exp\biggl(-\la\int_s^t  \|\na v^{d}(\tau)\|_{L^{\f{dr}{2r-1}}_+}^{2r}\,d\tau\biggr)\cdotp
$$
Then   starting from Formula \eqref{eq:aux5a}
with $$\wt g_\la=(\wt f^h_\la-v^d\pa_d\wt u^h_\la,\wt f^d_\la-\wt u^h_\la \cdot\nabla_hv^d+v^d\dive_h \wt u^h_\la),$$
 following the computations leading to \eqref{eq:aux6},
and using  \eqref{add.2} which ensures that the horizontal components of $u_L^h$ may be expressed in terms
of $u_0^h$ only, yields \eqref{eq:aux1}.
%Proving \eqref{eq:aux3} is the same as proving \eqref{eq:aux6}.
 %\medbreak
By a similar device, keeping in mind the (new) definition of $h_\la,$ we get Inequality \eqref{eq:aux2}.
%$$\displaylines{
%\|u^{h}_\la\|_{\mathfrak{X}^{\wt p,r}_t}+\mu^{1-\f1{2r}}\|\na u^{h}_\la\|_{L^{2r}_t(L^{\alpha}_+)}\leq
 %C\bigl(\mu^{1-\f1r}\|u_0^h\|_{\dot
%B^{-1+\f{d}p}_{\wt p,r}(\R^d_+)}\hfill\cr\hfill
%+\|f_\la\|_{L^{r}_t(L^{\wt p}_+)}+\|\nabla v^d\|_{L_t^{2r}(L^\alpha_+)}\|\nabla u^h_\la\|_{L^{2r}(L^{\f{dr}{2r-1}}_+)}\bigr),}$$
%then combining with \eqref{eq:aux1} gives \eqref{eq:aux2}.
%\medbreak
%In the same spirit, arguing as for proving \eqref{eq:aux7}, but with the new definition of $h_\la$ yields
%$$\displaylines{\quad\|(u_\la,\nabla\Pi_\la)\|_{X^{\wt p,r}_t}+\mu^{1-\f1{2r}}\|\na u_\la\|_{L^{2r}_t(L^{\alpha}_+)}\leq C\bigl(\mu^{1-\f1r}\|u_0\|_{\dot B^{-1+\f{d}p}_{\wt p,r}(\R^d_+)}\hfill\cr\hfill+\|f_\la\|_{L^{r}_t(L^{\wt p}_+)}+\|\nabla v^d\|_{L_t^{2r}(L^\alpha_+)}\|\nabla u^h_\la\|_{L^{2r}(L^{\f{dr}{r-1}}_+)}\bigr)\quad}$$
%and  combining with \eqref{eq:aux1} completes the proof of \eqref{eq:aux4}.
\end{proof}

\noindent {\bf Acknowledgments.} The authors warmly thank Marius Paicu for pointing them
out  Formula \eqref{add.2} which enabled them  to improve the original (isotropic) smallness condition \eqref{small3}
into  \eqref{small1} in the present version.

 This work was initiated while the
second author was visiting the laboratory  of mathematics of
Universit\'e Paris-Est Cr\'eteil in April 2012 and has benefitted
from the ANR grant awarded  to the Labex B\'ezout ANR-10-LABX-58. He
would like to thank the hospitality of the laboratory. The second
author  is  also partially supported by NSF of China under Grant
10421101 and 10931007,  and innovation grant from National Center
for Mathematics and Interdisciplinary Sciences.

\end{document}